\documentclass[11pt, twoside]{article}

\usepackage[english]{babel}

\usepackage{amssymb}
\usepackage{mathrsfs}
\usepackage{amsmath}
\usepackage{amsthm}
\usepackage{amsfonts}
\usepackage{latexsym}
\usepackage{indentfirst}
\usepackage{color}
\usepackage{graphicx}
\usepackage{txfonts}
\usepackage{enumerate}

\usepackage[colorlinks=true,
linkcolor=blue,
citecolor=red,
urlcolor=magenta,
]{hyperref}

\usepackage{txfonts}
\usepackage{anysize}

\allowdisplaybreaks

\newtheorem{theorem}{Theorem}[section]
\newtheorem{lemma}[theorem]{Lemma}

\newtheorem{proposition}[theorem]{Proposition}

\theoremstyle{definition}
\newtheorem{remark}[theorem]{Remark}

\newtheorem{definition}[theorem]{Definition}
\newcounter{assum}

\renewcommand{\appendix}{\par
\setcounter{section}{0}%
\setcounter{subsection}{0}%
\setcounter{subsubsection}{0}%
\gdef\thesection{\@Alph\c@section}%
\gdef\thesubsection{\@Alph\c@section.\@arabic\c@subsection}%
\gdef\theHsection{\@Alph\c@section.}%
\gdef\theHsubsection{\@Alph\c@section.\@arabic\c@subsection}%
\csname appendixmore\endcsname
}

\numberwithin{equation}{section}

\begin{document}

\title{\bf\Large
 Gamma-Convergence of Weak-Type Nonlocal Functionals\footnotetext{\hspace{-0.35cm} 2020 {\it
Mathematics Subject Classification}.
Primary 49J45; Secondary 46E35, 26B30, 26D10.
\endgraf {\it Key words and phrases.}
Gamma-convergence, nonlocal functional, Sobolev space, bounded variation.
\endgraf This project is partially supported by the National
Natural Science Foundation of China (Grant Nos. 12431006, 12371093,
12501118, and 12601167), the Natural Science Foundation of Fujian
Province (Grant No. 2026J008197),
the Beijing Natural Science Foundation (Grant No. 1262011),
the Fundamental Research Funds for the Central Universities
(Grant No. 2253200028), and Longyuan Young Talents of Gansu Province.}}
\author{Xiaosheng Lin, Dachun Yang, Sibei Yang, Wen Yuan
and Yangyang Zhang}
\date{\today}
\maketitle

\vspace{-0.8cm}

\begin{center}
\begin{minipage}{13cm}
{\small {\bf Abstract.}\quad
Let  $p\in[1,\infty)$, and let  $\Omega\subset\mathbb R^N$ be
a bounded open interval when $N=1$ or a bounded Lipschitz domain when $N\ge2$.
For $\gamma\in\mathbb R\setminus\{0\}$ and $\lambda\in(0,\infty)$,
consider the weak-type nonlocal functional
\begin{align*}
F_{\lambda,p,\gamma}(u;\Omega):=\lambda^p
\underset{\scriptscriptstyle |u(x)-u(y)|>\lambda|x-y|^{1+\gamma/p}}
{\int_\Omega\int_\Omega}|x-y|^{\gamma-N}\,dy\,dx
\end{align*}
with $u:\Omega\to\mathbb{R}$ measurable.
We prove that, as $\lambda\to\infty$ for $\gamma\in(0,\infty)$
or as $\lambda\to0^+$ for $\gamma\in(-\infty,-1]$,
$F_{\lambda,p,\gamma}(\cdot;\Omega)$ converges, in the sense of
$\Gamma$-convergence in $L^p(\Omega)$, to
\begin{align*}
\Psi_{p,\gamma}^{\Omega}(u)
:=\begin{cases}
C_{N,p,\gamma}^{\rm cell}\displaystyle\int_\Omega|\nabla u|^p\,dx,
&p\in(1,\infty)\ \hbox{and}\ u\in W^{1,p}(\Omega),\\
C_{N,1,\gamma}^{\rm cell}|Du|(\Omega),
&p=1\ \hbox{and}\ u\in BV(\Omega),\\
\infty,&\hbox{otherwise},
\end{cases}
\end{align*}
where the positive constants $C_{N,p,\gamma}^{\mathrm{cell}}$
are independent of $\Omega$.
We also establish the $\Gamma$-convergence of
$F_{\lambda,p,\gamma}(\cdot;\mathbb R^N)$ in
$L^1_{\mathrm{loc}}(\mathbb R^N)$ for all $\gamma\in\mathbb R\setminus\{0\}$,
with the limit functional having an additional term
when $(N,p,\gamma)=(1,1,-1)$. These results affirmatively
answer [Open Problem~9.3, Rend. Lincei Mat. Appl. 2023] of
Brezis and also the other problem by Brezis et al. in [Section~7C, Anal. PDE 2024].}
\end{minipage}
\end{center}

\vspace{0.1cm}

\tableofcontents

\section{Introduction}

Let $p\in [1,\infty)$, and let $\Omega\subset \mathbb{R}^N$ be an open set.
Recall that the \emph{homogeneous Sobolev space} $\dot{W}^{1,p}(\Omega)$ is defined by setting
\begin{align*}
\dot{W}^{1,p}(\Omega):=\left\{f\in L^p_{\mathrm{loc}}(\Omega):\ |\nabla f|\in L^p(\Omega)\right\},
\end{align*}
where the \emph{notation $L^p_{\mathrm{loc}}(\Omega)$} denotes the set
of all locally $p$-integrable functions on
$\Omega$ and the \emph{notation $\nabla f$} denotes the \emph{distributional gradient} of $f$.
Moreover, the \emph{nonhomogeneous Sobolev space} $W^{1,p}(\Omega)$ is defined as
$$W^{1,p}(\Omega):= L^p(\Omega)\cap \dot{W}^{1,p}(\Omega).$$
Meanwhile, denote by the \emph{notation $\mathcal M(\Omega;\mathbb R^N)$}
the space of all finite $\mathbb R^N$-valued Radon
measures on $\Omega$. Then the homogeneous and the
inhomogeneous \emph{bounded-variation spaces} $\dot{BV}(\Omega)$
and $BV(\Omega)$ are respectively defined by setting
$$
\dot{BV}(\Omega):=\left\{u\in L^1_{\mathrm{loc}}(\Omega):Du\in\mathcal M(\Omega;\mathbb R^N)\right\}
$$
and
$$ BV(\Omega):=\dot{BV}(\Omega)\cap L^1(\Omega)$$
(see, for example, \cite{afp00}). Here, the \emph{notation $Du$}
denotes the distributional derivative of $u$
and the \emph{notation $|Du|$} its total variation measure.

In the seminal articles \cite{bbm01, b02}, Bourgain, Brezis, and Mironescu
showed that,
for any $f\in W^{1,p}(\mathbb{R}^N)$,
\begin{align*}
\lim_{s\to 1^-}(1-s)\int_{\mathbb{R}^N}\int_{\mathbb{R}^N}\frac{|f(x)-f(y)|^p}{|x-y|^{N+sp}}\,dx\,dy
=\frac{K_{N,p}}{p}\|\,|\nabla f|\,\|_{L^p(\mathbb{R}^N)}^p,
\end{align*}
which is nowadays called the \emph{Bourgain--Brezis--Mironescu} (BBM) \emph{formula}. Here
\begin{align}\label{eq:gpgeo}
 K_{N,p}:=\int_{\mathbb S^{N-1}}|\omega\cdot e_N|^p\,d\mathcal H^{N-1}(\omega),
\end{align}
where $\mathbb S^{N-1}$ and $\mathcal H^{N-1}$ respectively
denote the \emph{unit sphere} in $\mathbb R^{N}$ and
the \emph{surface measure} on $\mathbb S^{N-1}$ and where
$e_N:=(0,\ldots,0,1)\in\mathbb S^{N-1}$.
Since then, the BBM formula and its various generalized versions have attracted a lot of attention
(see, for example, \cite{bn18,dgpyyz24,dlyyz23,dlyyz22,dm22,dm23,dm26,dssvy23,kmx05,zyy23}).
At the endpoint case $s=1$, Brezis et al. \cite{bvy21} proved that, for any
$f\in C_{\rm c}^\infty(\mathbb R^N)$ (the space of all infinitely
differentiable functions on $\mathbb R^N$ with compact support),
\begin{align*}
&\lim_{\lambda\to\infty}\lambda^p\mathcal L^{2N}
\left(\left\{(x,y)\in\mathbb R^N\times\mathbb R^N:
x\neq y,\ \frac{|f(x)-f(y)|}{|x-y|^{1+N/p}}\geq\lambda\right\}\right)
=\frac{K_{N,p}}{N}\|\,|\nabla f|\,\|_{L^p(\mathbb R^N)}^p,
\end{align*}
where $\mathcal L^{2N}$ denotes the Lebesgue measure on ${\mathbb R}^{2N}$.
For any given $\lambda\in(0,\infty)$, $p\in[1,\infty)$,
and $\gamma\in\mathbb{R}\backslash\{0\}$
and for any measurable function $u:\Omega\to\mathbb R$, let
\begin{align}\label{F-def}
F_{\lambda,p,\gamma}(u;\Omega)
:=
\lambda^p
\underset{\scriptscriptstyle |u(x)-u(y)|>\lambda |x-y|^{1+\gamma/p}}
{\int_{\Omega}\int_{\Omega}}
|x-y|^{\gamma-N}\,dy\,dx.
\end{align}
In what follows, to simplify the presentation,
we always use the following \emph{abbreviation}
$$\underset{\scriptscriptstyle g(x,y)>h(x,y)}
{\int_{X}\int_{Y}}\cdots\,dy\,dx:=
\int_{X}\left[\int_{\{y\in Y:g(x,y)>h(x,y)\}}\cdots\,dy\right]\,dx.$$
When $\Omega:=\mathbb R^N$, Brezis et al. \cite{bsvy24}
studied the family \eqref{F-def} in a more general framework. Precisely, they obtained
sharp asymptotic estimates as $\lambda\to\infty$ when $\gamma\in(0,\infty)$ and as
$\lambda\downarrow0$ when $\gamma\in(-\infty,0)$ and, among other consequences, derived
the characterization of homogeneous Sobolev spaces in terms of \eqref{F-def}.

These results led to a natural variational question.  In
\cite[Section~7C]{bsvy24}, Brezis et al. asked whether, for any given
$p\in[1,\infty)$ and $\gamma\in\mathbb{R}\backslash\{0\}$, there exists a positive constant
$c:=c(p,\gamma,N)$, depending only on $p$, $\gamma$, and $N$, such that the family
$F_{\lambda,p,\gamma}(\cdot;\mathbb{R}^N)$ converges, in the sense of $\Gamma$-convergence in $L^1_{\mathrm{loc}}(\mathbb R^N)$,
to the functional
\begin{align*}
\Psi_{\ast,c}(u):=\begin{cases}
 c\displaystyle\int_{\mathbb R^N}|\nabla u|^p\,dx,
&p\in(1,\infty)\ \hbox{and}\ u\in\dot{W}^{1,p}(\mathbb R^N),\\
c|Du|(\mathbb R^N),
&p=1\ \hbox{and}\ u\in\dot{BV}(\mathbb R^N),\\
\infty,&\hbox{otherwise},
\end{cases}
\end{align*}
as $\lambda\to\infty$ when $\gamma\in(0,\infty)$ and as $\lambda\downarrow0$ when
$\gamma\in(-\infty,0)$.  The case $p=1$ is especially delicate because the pointwise
limits on $\dot{W}^{1,1}(\mathbb R^N)$ and on $\dot{BV}(\mathbb R^N)$ may
differ. Furthermore, Brezis subsequently formulated in \cite[Open Problem~9.3]{b23} a
bounded-domain version of the case $\gamma=N$ of the above convergence of the family
$F_{\lambda,p,\gamma}(\cdot;\Omega)$ in \eqref{F-def}.
Brezis \cite[Open Problem~9.3]{b23} asked whether the family
$F_{\lambda,p,N}(\cdot;\Omega)$ converges as
$\lambda\to\infty$, in the sense of $\Gamma$-convergence
in $L^p(\Omega)$, to the functional
\begin{align*}
\Psi_p(u;\Omega):=\begin{cases}
c_p\displaystyle\int_\Omega|\nabla u|^p\,dx,
&p\in(1,\infty)\ \hbox{and}\ u\in W^{1,p}(\Omega),\\
c_1|Du|(\Omega),
&p=1\ \hbox{and}\ u\in BV(\Omega),\\
\infty,&\hbox{otherwise}
\end{cases}
\end{align*}
for some positive constants $c_p$ when $p\in(1,\infty)$ and $c_1$ when $p=1$.

Recall that, for the classical BBM energies, Ponce \cite{p04} showed that the corresponding
functionals $\Gamma$-converge with respect to the strong $L^p$ topology and that
the $\Gamma$-limit has the same constant as the pointwise BBM limit.
A different phenomenon occurs when
$\gamma:=-p$ in \eqref{F-def}: Nguyen \cite{n07,n11} proved that, as
$\lambda\downarrow0$, the family $F_{\lambda,p,-p}(\cdot;\mathbb{R}^N)$
$\Gamma$-converges in $L^p(\mathbb R^N)$ to a positive multiple of the
Sobolev energy when $p\in(1,\infty)$, or of the total variation when $p=1$,
and this positive multiplicative constant is strictly smaller than the pointwise-limit constant.
The positive multiplicative constant in this $\Gamma$-limit was given by an implicit cell
formula in \cite[formula~(1.3)]{n11}; its explicit value was later computed by
Antonucci et al. \cite[Theorem~1.1]{agmp20}.
Related nonlocal nonconvex approximations
and  their connections with image processing
 were further studied by Brezis and Nguyen
\cite{bn18,bn20ccm,bn20na}.
More recently, Nguyen
\cite[Theorem~29 and Remark~30]{n25} showed the following result:
if $p\in[1,\infty)$ and $\gamma\in(-1,0)$, then the family
$F_{\lambda,p,\gamma}(\cdot;\mathbb{R}^N)$ $\Gamma$-converges in
$L^p(\mathbb R^N)$, as $\lambda\downarrow0$, to the zero functional.
Nguyen \cite[Remark~31]{n25} further remarked that it would be \emph{interesting}
to revisit the arguments in \cite{bn18,n11} to establish the general
$\Gamma$-convergence
of the family $F_{\lambda,p,\gamma}(\cdot;\mathbb{R}^N)$.
 For $\gamma\in(0,\infty)$, Gobbino and Picenni \cite{gp25}
established a positive lower bound for the $\Gamma$-liminf with respect to
$L^1_{\mathrm{loc}}(\mathbb R^N)$ convergence as $\lambda\to\infty$.
In the case $p=1$ and $\gamma\in(0,\infty)$, Lahti and Li \cite{ll26}
established the $\Gamma$-limit of
$F_{\lambda,1,\gamma}(\cdot;\Omega)$ with respect to area-strict convergence.
However, both the bounded-domain
$\Gamma$-convergence problem \cite[Open Problem~9.3]{b23}
in the strong $L^p(\Omega)$ topology
and the whole-space problem \cite[Section~7C]{bsvy24} in the
$L^1_{\mathrm{loc}}(\mathbb R^N)$ topology remain partially open.

In this article, we first establish the $\Gamma$-limits on bounded
domains with respect to the strong $L^p(\Omega)$ topology.
We then extend these results to the whole space with respect
to $L^1_{\mathrm{loc}}(\mathbb R^N)$ convergence.
Throughout this article, we always let
\begin{align*}
I_0:=\left(-\frac{1}{2},\frac{1}{2}\right)\ \ \text{and}\ \ Q_0:=I_0^N=\left(-\frac{1}{2},\frac{1}{2}\right)^N.
\end{align*}
For any $\gamma\in\mathbb{R}\setminus\{0\}$, define
 \begin{align*}
\lambda_\gamma:=\begin{cases}\infty,&\gamma\in(0,\infty),\\0,&\gamma\in(-\infty,0).\end{cases}
\end{align*}
The \emph{notation} $\lambda\to\lambda_\gamma$ always means $\lambda\to0^+$ when $\gamma\in (-\infty,0)$
and $\lambda\to \infty$ when $\gamma\in (0,\infty).$

\begin{definition}
Let $p\in [1,\infty)$ and $\gamma\in \mathbb{R}\setminus\{0\}$.
For any $u\in L^p(Q_0)$, define
\begin{align}\label{eq:mp}
m_{p,\gamma}(u):=\inf\varliminf_{j\to\infty}F_{\lambda_j,p,\gamma}(u_j;Q_0),
\end{align}
where the infimum is taken over all sequences
$\{\lambda_j\}_{j\in\mathbb N}\subset(0,\infty)$ and
$\{u_j\}_{j\in\mathbb N}\subset L^p(Q_0)$ satisfying
$u_j\to u$ in $L^p(Q_0)$ as
$\lambda_j\to\lambda_\gamma$.
Moreover, let
\begin{align}\label{baochang}
C_{N,p,\gamma}^{\mathrm{cell}}:=m_{p,\gamma}(\ell),
\end{align}
where $\ell(x):=x_N$ for any $x:=(x_1,\ldots,x_N)\in\mathbb R^N$.
\end{definition}

Let $p\in [1,\infty)$ and $\gamma\in\mathbb{R}\setminus \{0\}$.
Define the \emph{functional}
$\Psi_{p,\gamma}^{\Omega}:L^p(\Omega)\to[0,\infty]$ by setting, for any
$u\in L^p(\Omega)$,
\begin{align*}
\Psi_{p,\gamma}^{\Omega}(u)
:=
\begin{cases}
C_{N,p,\gamma}^{\rm cell}\displaystyle\int_\Omega|\nabla u|^p\,dx,
&p\in(1,\infty)\ \hbox{and}\ u\in W^{1,p}(\Omega),\\
C_{N,1,\gamma}^{\rm cell}|Du|(\Omega),
&p=1\ \hbox{and}\ u\in BV(\Omega),\\
\infty,&\hbox{otherwise}.
\end{cases}
\end{align*}
The following theorem is our first main result, which completely
answers \cite[Open Problem~9.3]{b23}.

\begin{theorem}\label{thm-negdom}
Let  $p\in[1,\infty)$, $\gamma\in\left[(-\infty,-1]\setminus\{-p\}\right]\cup(0,\infty)$, and
$\Omega\subset\mathbb R^N$ be a bounded open interval when $N=1$ or a
bounded Lipschitz domain when $N\geq2$.
Then the family $F_{\lambda,p,\gamma}(\cdot;\Omega)$ converges, in the sense
of  $\Gamma$-convergence in $L^p(\Omega)$, to
$\Psi_{p,\gamma}^{\Omega}$ as $\lambda\to\lambda_\gamma$.
More precisely, let $\{\lambda_j\}_{j\in\mathbb N}\subset(0,\infty)$
satisfy $\lambda_j\to\lambda_\gamma$. Then the following statements hold.
\begin{enumerate}[{\rm(i)}]
\item The $\Gamma$-liminf inequality: for any $u\in L^p(\Omega)$ and any
sequence $\{u_j\}_{j\in\mathbb N}\subset L^p(\Omega)$ satisfying
$u_j\to u$ in $L^p(\Omega)$,
\begin{align*}
\varliminf_{j\to\infty}F_{\lambda_j,p,\gamma}(u_j;\Omega)
\geq\Psi_{p,\gamma}^{\Omega}(u).
\end{align*}
\item The $\Gamma$-limsup inequality: for any $u\in L^p(\Omega)$, there
exists a sequence $\{u_j\}_{j\in\mathbb N}\subset L^p(\Omega)$ satisfying
$u_j\to u$ in $L^p(\Omega)$ and
\begin{align*}
\varlimsup_{j\to\infty}F_{\lambda_j,p,\gamma}(u_j;\Omega)
\leq\Psi_{p,\gamma}^{\Omega}(u).
\end{align*}
Such a sequence is called a \emph{recovery sequence} for $u$.
\end{enumerate}
\end{theorem}
Theorem~\ref{thm-negdom} follows directly from
the $\Gamma$-liminf inequality in Theorems~\ref{thm-pliminf} and~\ref{thm-onelinf}
and the $\Gamma$-limsup inequality in Theorem~\ref{thm-limsup}.
Combining Theorem~\ref{thm-negdom} with other known results,
we obtain the following conclusion.

\begin{theorem}\label{thm-domfull}
Let  $p\in[1,\infty)$, $\gamma\in\mathbb R\setminus\{0\}$, and
$\Omega\subset\mathbb R^N$ be a bounded open interval when $N=1$ or a
bounded Lipschitz domain when $N\geq2$.
Then the following assertions hold.
\begin{enumerate}[{\rm(i)}]
\item If $\gamma\in(-\infty,-1]\cup(0,\infty)$, then the family
$F_{\lambda,p,\gamma}(\cdot;\Omega)$ converges, in the sense of
$\Gamma$-convergence in $L^p(\Omega)$, to $\Psi_{p,\gamma}^{\Omega}$ as
$\lambda\to\lambda_\gamma$.
\item If $\gamma\in(-1,0)$, then the family
$F_{\lambda,p,\gamma}(\cdot;\Omega)$ converges, in the sense of
$\Gamma$-convergence in $L^p(\Omega)$, to the zero functional as
$\lambda\to0^+$.
\end{enumerate}
\end{theorem}

\begin{remark}
The pointwise limit of the BSVY-type functionals \eqref{F-def}
need not exist for $BV$ functions. More precisely, Lahti
\cite[Example 3.22]{la25} proved that, for any given $\gamma\in (0,\infty)$, there exists
$u\in BV(0,1)$ such that the limit
$\lim_{\lambda\to\infty}
F_{\lambda,1,\gamma}(u;(0,1))$
does not exist. However, Theorem \ref{thm-negdom} shows that, for any fixed
$\gamma\in(0,\infty)$, the family $F_{\lambda,1,\gamma}(\cdot;\Omega)$
$\Gamma$-converges in $L^1(\Omega)$, as $\lambda\to\infty$, to the functional $\Psi_{1,\gamma}^{\Omega}$.
Thus, although the pointwise limit $F_{\lambda,1,\gamma}(u;\Omega)$ may fail to
exist for a fixed $u\in BV(\Omega)$, the corresponding $\Gamma$-limit exists
and satisfies
\begin{align*}
\left(
\Gamma\mathop{\!-\!\lim}_{\lambda\to\infty}
F_{\lambda,1,\gamma}(\cdot;\Omega)
\right)(u)
=
C_{N,1,\gamma}^{\mathrm{cell}}|Du|(\Omega).
\end{align*}
For further results on the pointwise asymptotic behavior of
these functionals on $BV$ spaces, we refer to Lahti \cite{la25}
and Picenni \cite{pic24}.
\end{remark}

Let $u\in\dot{BV}(\mathbb R)$.
Recall that,  by
\cite[Theorems~3.27 and~3.28]{afp00}, we may choose a  good
representative $\widetilde u$, which satisfies, for any $x,y,R\in (0,\infty)$ with $R<x<y$,
\begin{align*}
|\widetilde u(y)-\widetilde u(x)|
\leq |Du|[x,y]\leq |Du|(R,\infty).
\end{align*}
This,  together with the Cauchy criterion, implies that the limits
\begin{align*}
u_-:=\lim_{x\to-\infty}\widetilde u(x)
\quad\text{and}\quad
u_+:=\lim_{x\to\infty}\widetilde u(x)
\end{align*}
exist and are finite. It is easy to see that $u_-$ and $u_+$ are independent
of the choice of the good representative. As an application of Theorem \ref{thm-domfull}, we obtain our second main result, which answers the problem in \cite[Section~7C]{bsvy24}.

\begin{theorem}\label{thm-whlfull}
Let  $p\in[1,\infty)$ and $\gamma\in\mathbb R\setminus\{0\}$.
Define the functional
$\Psi_{p,\gamma}^{\rm hom}:L^1_{\rm loc}(\mathbb R^N)\to[0,\infty]$ by
setting, for any $u\in L^1_{\rm loc}(\mathbb R^N)$,
\begin{align*}
\Psi_{p,\gamma}^{\rm hom}(u)
:=
\begin{cases}
C_{N,p,\gamma}^{\rm cell}\displaystyle\int_{\mathbb R^N}|\nabla u|^p\,dx,
&p\in(1,\infty)\ \hbox{and}\ u\in\dot W^{1,p}(\mathbb R^N),\\
C_{N,1,\gamma}^{\rm cell}|Du|(\mathbb R^N),
&p=1,\ (N,\gamma)\neq(1,-1),\ \hbox{and}\ u\in\dot{BV}(\mathbb R^N),\\
C_{1,1,-1}^{\rm cell}\left[|Du|(\mathbb R)+|u_+-u_-|\right],
&(N,p,\gamma)=(1,1,-1)\ \hbox{and}\ u\in\dot{BV}(\mathbb R),\\
\infty,&\hbox{otherwise}.
\end{cases}
\end{align*}
Then the following statements hold.
\begin{enumerate}[{\rm(i)}]
\item If $\gamma\in(-\infty,-1]\cup(0,\infty)$, then the family
$F_{\lambda,p,\gamma}(\cdot;\mathbb R^N)$ converges, in the sense of
$\Gamma$-convergence in $L^1_{\rm loc}(\mathbb R^N)$, to
$\Psi_{p,\gamma}^{\rm hom}$ as $\lambda\to\lambda_\gamma$.
\item If $\gamma\in(-1,0)$, then the family
$F_{\lambda,p,\gamma}(\cdot;\mathbb R^N)$ converges, in the sense of
$\Gamma$-convergence in $L^1_{\rm loc}(\mathbb R^N)$, to
the zero functional as $\lambda\to0^+$.
\end{enumerate}
\end{theorem}

Inspired by Nguyen's approach \cite{n11}, we use an affine cell formula
to characterize the coefficient in the $\Gamma$-limit and piecewise affine
approximations to construct matching recovery sequences.
Following this strategy, we modify near-minimizing sequences for the affine cell
constant so that they agree with the affine function near the boundary of the cube,
and then place scaled copies of these modified sequences into a continuous
piecewise affine approximation. The boundary modifications use a common Lipschitz
cutoff estimate, while the interactions between different cubes are estimated
separately in the positive and negative ranges.

Let $p\in[1,\infty)$. To establish the $\Gamma$-liminf inequality, we first
address the regularity of limits with finite energy liminf. While regularity
estimates for $\gamma\in(0,\infty)$ were obtained by Gobbino and Picenni
\cite{gp25}, the corresponding estimates for
$\gamma\in(-\infty,-1]\setminus\{-p\}$ are established in the present article.
These estimates show that an $L^p$ limit has Sobolev or $BV$ regularity
whenever the energy liminf is finite, and also imply that the affine cell
constant is strictly positive.
Our approach differs from the dyadic discretization used in \cite{gp25}.
For $\gamma\in(-p,-1]$, we obtain the required one-dimensional estimates
by transferring an oscillation estimate from the critical case
$(p,\gamma)=(1,-1)$; for $\gamma\in(-\infty,-p)$, we derive integral estimates for averaged
difference quotients in terms of the nonlocal energy.
A slicing argument then yields the higher-dimensional estimates.

For $p=1$ and $\gamma\in(-\infty,-1)\cup(0,\infty)$, our proof of the
optimal lower bound follows a different route from Nguyen's argument
based on a characterization of $BV$ functions in terms of essential
variation \cite{n11}.
By the fine structure of $BV$ functions, normalized blow-ups need not
be affine and may converge to general one-dimensional nondecreasing
$BV$ functions.
We therefore prove that the cell energy of any such function is bounded
from below by its total variation times the same affine cell constant
$C_{N,1,\gamma}^{\mathrm{cell}}$.

For $p\in[1,\infty)$ and $\gamma\in(-\infty,-1]\cup(0,\infty)$,
we obtain the Sobolev and $BV$ lower bounds on $\mathbb R^N$
with respect to $L^1_{\mathrm{loc}}(\mathbb R^N)$ convergence
by combining truncation with the bounded-domain results. Recovery sequences are constructed through
approximation in the homogeneous spaces and estimates for interactions
with the exterior of bounded domains.
At the exceptional endpoint $(N,p,\gamma)=(1,1,-1)$, the $\Gamma$-limit
contains the additional term
$C_{1,1,-1}^{\mathrm{cell}}|u_+-u_-|$.
We identify its contribution to the lower bound by using inversion
to turn the difference between the limits at infinity into a local jump,
and construct recovery sequences with the matching cost.
This term records a contribution from infinity that vanishes for
globally integrable $BV$ functions.

The remainder of this article is organized as follows.

In Section~\ref{sec:tools}, we give the basic properties of the cell
functional and the cutoff estimate.
In Section~\ref{sec-coerc}, we establish lower bounds which imply that,
whenever the liminf of the energies is finite, the limit function has
Sobolev regularity for $p\in(1,\infty)$ and $BV$ regularity for $p=1$.
In Section~\ref{sec-liminf}, we show the optimal lower bounds on bounded
domains for $p\in(1,\infty)$ and $p=1$, respectively,
in Subsections~\ref{sec-pinf} and~\ref{proofone}.
In Subsection~\ref{sub4.3}, we describe the one-dimensional limits
of rescaled $BV$ functions and, in Subsection~\ref{sub4.4}, we compare
the cell energy of monotone profiles with the affine cell constant.
In Section~\ref{sec-limsup}, we prove the upper bounds on bounded domains.
In Subsection~\ref{sub5.1}, we modify the sequences realizing the affine
cell constant near the boundary, while, in Subsection~\ref{sub5.2},
we use these sequences to construct recovery
sequences for continuous finite piecewise affine functions.
In Subsection~\ref{subsec-recovery-density}, we extend this construction
to Sobolev and $BV$ functions by approximation.
In Subsection~\ref{sec-domain}, we treat the critical case and the zero
limit, completing the classification on bounded domains.
Finally, in Section~\ref{sec-whole}, we establish the whole-space limits,
beginning with the zero limit.
In Subsection~\ref{sec-nondeg}, we show the whole-space result in the
nondegenerate ranges apart from the exceptional endpoint and, in
Subsection~\ref{sec-endpoint}, we identify the additional term at this endpoint.

At the end of this section, we make some notational conventions. Let
${\mathbb N}:=\{1,2,\ldots\}$.
We always denote by $C$ a positive constant
which is independent of the main parameters involved,
but it may vary from line to line.
We also use $C_{\alpha,\beta,\ldots}$ to denote a
positive constant depending on the indicated parameters
$\alpha$, $\beta,\ldots$.
The \emph{notation $f\lesssim g$} means that $f\leq Cg$.
If $f\lesssim g$ and $g\lesssim f$,
we then write $f\sim g$. If $f\leq Cg$ and $g=h$
or $g\leq h$, we then write $f\lesssim g=h$ or $f\lesssim g\leq h$.
For any $x\in{\mathbb R}^N$ and $r\in(0,\infty)$,
we denote by
$B(x,r):=\{y\in{\mathbb R}^N:\ |y-x|<r\}$
the ball with center $x$ and radius $r$.
We use $\mathbf{1}_E$ to denote the
characteristic function of a measurable set $E\subset{\mathbb R}^N$
and $\mathbf{0}$ to denote the \emph{origin} of ${\mathbb R}^N$.
We denote by $\mathcal L^N$ the $N$-dimensional Lebesgue measure
and by $\mathcal H^k$ the $k$-dimensional Hausdorff measure.
For any $x\in\mathbb{R}^N$ and any nonempty set $E\subset\mathbb{R}^N$, let
$\mathrm{dist}(x,E):=\inf\{|x-y|:y\in E\}$. If
$\mu$ is a Radon measure and $A$ is a Borel set, then
$\mu\llcorner A$ denotes the \emph{restriction} of $\mu$ to $A$, namely
$(\mu\mathopen{\scalebox{1.6}{$\llcorner$}}A)(E):=\mu(E\cap A)$
for every Borel set $E$. If $\mu$ is a positive Radon measure and
$f\in L^1(\mu;\mathbb R^m)$, then $f\mu$ denotes the
$\mathbb R^m$-valued Radon measure defined by
$(f\mu)(E):=\int_E f\,d\mu.$
For a finite signed or vector-valued Radon measure $\sigma$, we
denote its \emph{total variation measure} by $|\sigma|$.
We use $\mathcal O(N)$ to denote the set of all $N\times N$ orthogonal matrices.
For any measurable set $E\subset\mathbb{R}^N$
and any $f\in L^1(E)$, we denote the integral $\int_{E}f(x)\,dx$ simply by
$\int_{E}f\,dx$ and, when $|E|\in(0,\infty)$, we always use the following notation
$$
\fint_Ef\,dx:=\frac{1}{|E|}\int_{E}f(x)\,dx.
$$
For any bounded nonempty open interval $I\subset\mathbb R$ and any
measurable function $f:I\to\mathbb R$, define its essential supremum
and essential infimum by
\begin{align*}
\operatorname*{ess\,sup}_I f
:=\inf\left\{M\in\mathbb R:f(x)\leq M\text{ for almost every }x\in I\right\}
\end{align*}
and
\begin{align*}
\operatorname*{ess\,inf}_I f
:=\sup\left\{m\in\mathbb R:f(x)\geq m\text{ for almost every }x\in I\right\}.
\end{align*}
Here $\operatorname*{ess\,inf}\emptyset:=\infty$ and $\operatorname*{ess\,sup}\emptyset:=-\infty$.
For any given measurable set $\Omega \subset\mathbb{R}^N$,
the notation
$C^\infty_{\mathrm{c}}(\Omega)$
denotes the set
of all infinitely
differentiable functions on $\Omega$ with compact support.
In what follows, the \emph{notation $r\to 0^{+}$} means that $r\in (0,\infty)$ and $r\to 0$.
We use $\varliminf$ and $\varlimsup$ to denote the
\emph{limit inferior} and the \emph{limit superior}, respectively.
For any given function $f$  and any $x\in\mathbb R$,
we use the notation $f(x+)$ and $f(x-)$ to denote the limits
$\lim_{t\to x^+}f(t)$ and $\lim_{t\to x^-}f(t)$, respectively,
whenever the corresponding limits exist.
Finally, in all proofs we consistently retain the notation
introduced in the original theorem (or related statement).

\section{Preliminary}\label{sec:tools}
In this section, we give some  basic properties of the cell functional defined in \eqref{eq:mp}
and the cutoff estimate used in the boundary modifications.
It is easy to verify that $m_{p,\gamma}$ has the following scaling properties; we omit the details.
\begin{lemma}\label{lem-scaling}
Let  $p\in[1,\infty)$, $\gamma\in\mathbb R\setminus\{0\}$,
$U\subset\mathbb R^N$ be an open set, and $w:U\to\mathbb R$ be a measurable
function. Let $x_0\in\mathbb R^N$, $R\in\mathcal O(N)$,
$r\in(0,\infty)$, $a\in\mathbb R\setminus\{0\}$, and $b\in\mathbb R$.
Define the function $u:x_0+rRU\to\mathbb R$ by setting, for any $z\in U$,
$u(x_0+rRz):=arw(z)+b$. Then, for any $\lambda\in(0,\infty)$,
\begin{align*}
F_{\lambda,p,\gamma}(u;x_0+rRU)
=|a|^pr^NF_{\frac{\lambda r^{\gamma/p}}{|a|},p,\gamma}(w;U).
\end{align*}
Consequently, when $U=Q_0$ and $w\in L^p(Q_0)$,
\begin{align}\label{bisuo}
m_{p,\gamma}(aw+b)=|a|^pm_{p,\gamma}(w).
\end{align}
\end{lemma}

\begin{lemma}\label{lem-mplsc}
Let $p\in [1,\infty)$ and $\gamma\in\mathbb R\setminus\{0\}.$ The map $m_{p,\gamma}:L^p(Q_0)\to[0,\infty]$ is
lower semicontinuous, i.e., for any sequence $\{w_n\}_{n\in\mathbb{N}}\subset
L^p(Q_0)$ converging to $w$ in $L^p(Q_0)$,
\begin{align}\label{genggao}
m_{p,\gamma}(w)\leq\varliminf_{n\to\infty}m_{p,\gamma}(w_n).
\end{align}
\end{lemma}

\begin{proof}
Let $L:=\varliminf\limits_{n\to\infty}m_{p,\gamma}(w_n)\in[0,\infty].$ If $L=\infty$, then
the desired conclusion holds automatically. Assume $L<\infty$.
Passing to a subsequence and relabelling if necessary, we may assume that $\lim_{n\to\infty}m_{p,\gamma}(w_n)= L$ and $m_{p,\gamma}(w_n)<\infty$ for any $n\in\mathbb{N}.$
For any given $n\in\mathbb{N}$, choose an admissible sequence $\{(\lambda_{n,k},v_{n,k})\}_{k\in\mathbb{N}}$ in \eqref{eq:mp} for $w_n$
such that
\begin{align*}
\varliminf_{k\to\infty}F_{\lambda_{n,k},p,\gamma}(v_{n,k};Q_0)\leq m_{p,\gamma}(w_n)+\frac{1}{2n}.
\end{align*}
Using this, the convergence of both $\lambda_{n,k}$ and $v_{n,k}$, and the definition
of the limit inferior, we conclude that, for any given  $n\in\mathbb N$,
there exists $k(n)\in\mathbb{N}$ sufficiently large such that
$\lambda_{n,k(n)}\geq n$ if $\gamma>0$, $\lambda_{n,k(n)}\leq1/n$ if $\gamma<0$, $\|v_{n,k(n)}-w_n\|_{L^p(Q_0)}\leq\frac1n,$
and \begin{align*}
F_{\lambda_{n,k(n)},p,\gamma}(v_{n,k(n)};Q_0)
\leq m_{p,\gamma}(w_n)+\frac1n.
\end{align*}
This implies that \eqref{genggao} holds, which completes the proof of Lemma
\ref{lem-mplsc}.
\end{proof}

\begin{lemma}\label{lem-cellbds}
Let  $p\in[1,\infty)$ and
$\gamma\in\mathbb R\setminus\{0\}$. Then the following assertions hold.
\begin{itemize}
\item[$\mathrm{(i)}$]
$\lim_{\lambda\to\lambda_\gamma}F_{\lambda,p,\gamma}(\ell;Q_0)=K_{N,p}/|\gamma|$.
Consequently, $0\leq C_{N,p,\gamma}^{\rm cell}\leq\frac{K_{N,p}}{|\gamma|}.$

\item[$\mathrm{(ii)}$] There exists  sequences
$\{\lambda_j\}_{j\in\mathbb N}\subset(0,\infty)$
and
$\{v_j\}_{j\in\mathbb N}\subset L^p(Q_0)$ such that
$\lim_{j\to\infty}\lambda_j=\lambda_\gamma$,
$\lim_{j\to\infty}=\ell$ in $L^p(Q_0)$, and
\begin{align*}
\lim_{j\to\infty}F_{\lambda_j,p,\gamma}(v_j;Q_0)
=C_{N,p,\gamma}^{\rm cell}.
\end{align*}
\end{itemize}
\end{lemma}

\begin{proof}
By Fubini's theorem, we find that, for any $\lambda\in(0,\infty)$,
\begin{align*}
F_{\lambda,p,\gamma}(\ell;Q_0)
&=\lambda^p\underset{\scriptscriptstyle |h_N|>\lambda|h|^{1+\gamma/p}}{\int_{\mathbb R^N}}
|Q_0\cap(Q_0-h)|\,|h|^{\gamma-N}\,dh\\
&=\underset{\scriptscriptstyle |z_N|>|z|^{1+\gamma/p}}{\int_{\mathbb R^N}}
\left|Q_0\cap\left(Q_0-\lambda^{-p/\gamma}z\right)\right|\,|z|^{\gamma-N}\,dz.
\end{align*}
This, combined with the dominated convergence theorem, implies that
\begin{align*}
\lim_{\lambda\to\lambda_\gamma}F_{\lambda,p,\gamma}(\ell;Q_0)
=\underset{\scriptscriptstyle |z_N|>|z|^{1+\gamma/p}}{\int_{\mathbb R^N}}|z|^{\gamma-N}\,dz
=\frac{1}{|\gamma|}\int_{\mathbb S^{N-1}}|\omega\cdot e_N|^p\,d\mathcal H^{N-1}(\omega)
=\frac{K_{N,p}}{|\gamma|}.
\end{align*}
This proves (i).
Finally, (ii) follows from the definition of $C_{N,p,\gamma}^{\mathrm{cell}}$ and  the standard
diagonal argument used in the proof of Lemma~\ref{lem-mplsc}.
This completes the proof of Lemma~\ref{lem-cellbds}.
\end{proof}

For any $p\in[1,\infty)$, $\gamma\in\mathbb R\setminus\{0\}$, any measurable sets $E,D\subset\mathbb R^N$,
and any measurable function $u:E\cup D\to\mathbb R$, let
\begin{align*}
F_{\lambda,p,\gamma}(u;E,D):=\lambda^p
\underset{\scriptscriptstyle |u(x)-u(y)|>\lambda|x-y|^{1+\gamma/p}}{\int_E\int_D}|x-y|^{\gamma-N}\,dy\,dx.
\end{align*}
We then have the following cutoff estimate.
\begin{lemma}\label{lem-cutoff}
Let $p\in[1,\infty)$, $\gamma\in\mathbb R\setminus\{0\}$, and $\lambda\in(0,\infty)$. Then the following statements hold.
\begin{itemize}
\item[$\mathrm{(i)}$] For any  measurable sets $E, D\subset\mathbb R^N$,
\begin{align*}
F_{\lambda,p,\gamma}(\ell;E,D)\leq\frac{K_{N,p}}{|\gamma|}|E|.
\end{align*}
Consequently, $F_{\lambda,p,\gamma}(\ell;E)\leq\frac{K_{N,p}}{|\gamma|}|E|$.
\item[$\mathrm{(ii)}$] Let $U\subset\mathbb R^N$ be a bounded open set,
$u,v\in L^p(U)$, and $\eta:U\to[0,1]$ be Lipschitz.
Then, for any $\alpha\in(0,1)$,
\begin{align}\label{eq:pcutoff}
F_{\lambda,p,\gamma}(w;U)
&\leq\frac{F_{\alpha\lambda,p,\gamma}(u;U)}{\alpha^p}
+\frac{F_{\alpha\lambda,p,\gamma}(v;E_\eta,U)}{\alpha^p}\nonumber\\
&\quad+\frac{\mathcal H^{N-1}(\mathbb S^{N-1})L_\eta^p\left\|u-v\right\|_{L^p(U)}^p}
{|\gamma|(1-\alpha)^p},
\end{align}
where $w:=\eta u+(1-\eta)v$, $E_\eta:=\{x\in U:\eta(x)<1\}$, and
\begin{align*}
L_\eta:=\sup_{\genfrac{}{}{0pt}{}{x,y\in U}{x\neq y}}
\frac{|\eta(x)-\eta(y)|}{|x-y|}<\infty.
\end{align*}
\end{itemize}
\end{lemma}

\begin{proof}
By a  straightforward computation, we have, for any  measurable sets $E, D\subset\mathbb R^N$,
\begin{align*}
F_{\lambda,p,\gamma}(\ell;E,D)
\leq
\lambda^p\underset{\scriptscriptstyle |(x-y)\cdot e_N|>\lambda|x-y|^{1+\gamma/p}}{\int_{E}\int_{\mathbb R^N}}|x-y|^{\gamma-N}\,dy\,dx=\frac{K_{N,p}}{|\gamma|}|E|.
\end{align*}
This shows (i).

Observe that, for any $x,y\in U$,
\begin{align*}
w(x)-w(y)=A(x,y)+B(x,y),
\end{align*}
where
\begin{align*}
 A(x,y):=\eta(x)[u(x)-u(y)]+[1-\eta(x)][v(x)-v(y)]
\end{align*}
and
\begin{align*}
 B(x,y):=[\eta(x)-\eta(y)][u(y)-v(y)].
\end{align*}
Let
$$E_w:=\left\{(x,y)\in U\times U:\ \frac{|w(x)-w(y)|}{|x-y|^{\frac{p+\gamma}{p}}}> \lambda\right\},$$
$$E_{A}:=\left\{(x,y)\in U\times U:\ \frac{|A(x,y)|}{|x-y|^{\frac{p+\gamma}{p}}}> \alpha \lambda\right\},$$
and
 $$E_{B}:=\left\{(x,y)\in U\times U:\ \frac{|B(x,y)|}{|x-y|^{\frac{p+\gamma}{p}}}> (1-\alpha) \lambda\right\}.$$
Then
\begin{align}\label{joining2}
E_w\subset E_A\cup E_B.
\end{align}
From the convexity of the function $f(t):=|t|^p$ in $\mathbb{R}$, we deduce that, for any $x,y\in U,$
\begin{align*}
|A(x,y)|^p\leq\eta(x)|u(x)-u(y)|^p+[1-\eta(x)]|v(x)-v(y)|^p.
\end{align*}
This further implies that
\begin{align}\label{eq:joining}
E_A\subset E_{A,1}\cup E_{A,2},
\end{align}
where
\begin{align*}
E_{A,1}:=\left\{(x,y)\in U\times U:\ \frac{|u(x)-u(y)|^p}{|x-y|^{p+\gamma}}>\alpha^p\lambda^p\right\}
\end{align*}
and
\begin{align*}
 E_{A,2}:=\left\{(x,y)\in E_\eta\times U:\ \frac{|v(x)-v(y)|^p}{|x-y|^{p+\gamma}}>\alpha^p\lambda^p\right\}.
\end{align*}
By the definition of $L_\eta$, we conclude that, for any $(x,y)\in E_B$,
\begin{align*}
(1-\alpha)^p\lambda^p|x-y|^{p+\gamma}<|\eta(x)-\eta(y)|^p|u(y)-v(y)|^p
\leq L_\eta^p|x-y|^p|u(y)-v(y)|^p
\end{align*}
and hence
\begin{align*}
|x-y|^\gamma<\frac{L_\eta^p}{(1-\alpha)^p\lambda^p}|u(y)-v(y)|^p,
\end{align*}
which implies that
\begin{align*}
\lambda^p\iint_{E_B}|x-y|^{\gamma-N}\,dy\,dx
\leq\frac{\mathcal H^{N-1}(\mathbb S^{N-1}) L_\eta^p\|u-v\|^p_{L^p(U)}}{|\gamma|(1-\alpha)^p}.
\end{align*}
From this, \eqref{eq:joining}, and \eqref{joining2}, we infer that \eqref{eq:pcutoff} holds.
This completes the proof of Lemma \ref{lem-cutoff}.
\end{proof}

\section{Positive Lower Bounds}\label{sec-coerc}

The main purpose of this section is to establish $\Gamma$-liminf inequalities with positive constants
that need not be optimal. In the subsequent proofs, we mainly use these inequalities to obtain the
Sobolev regularity of the limit function when the liminf of the energies is finite, or the  $BV$ regularity
when $p=1$. The precise values of these constants are not needed in the subsequent proofs.
We first recall the lower bound for $\gamma\in(0,\infty)$.
The following lemma is a part of  \cite[Theorem~1.1]{gp25}.
\begin{lemma}\label{lem-gp}
Let $p\in[1,\infty)$, $\gamma\in(0,\infty)$, and $U\subset\mathbb R^N$ be an
open set. Then, for any $u\in L^1_{\mathrm{loc}}(U)$ and any sequences $\{\lambda_j\}_{j\in\mathbb{N}}$ and
$\{u_j\}_{j\in\mathbb{N}}$ satisfying $\lim_{j\to\infty}\lambda_j=\infty$ and
$\lim_{j\to\infty}u_j=u$ in $L^1_{\mathrm{loc}}(U),$
\begin{align*}
\varliminf_{j\to\infty}F_{\lambda_j,p,\gamma}(u_j;U)
\geq K_{N,p}\frac{\log 2}{2^{\gamma+1}-1}
\begin{cases}
\displaystyle\int_U|\nabla u|^p\,dx,&p\in(1,\infty)\hbox{ and }u\in\dot W^{1,p}(U),\\
|Du|(U),&p=1\hbox{ and }u\in\dot{BV}(U),\\
\infty,&\hbox{otherwise},
\end{cases}
\end{align*}
where $K_{N,p}$ is as in \eqref{eq:gpgeo}.
\end{lemma}

The following proposition
is the main result of this section.

\begin{proposition}\label{prop-coerc}
Let  $p\in[1,\infty)$ and
$\gamma\in(-\infty,-1]\setminus\{-p\}$. Let
$\Omega\subset\mathbb R^N$ be a bounded open interval when $N=1$ or a
bounded open set when $N\geq2$. Let
$\{\lambda_j\}_{j\in\mathbb N}\subset(0,\infty)$ satisfy
$\lambda_j\to0$, and let $\{u_j\}_{j\in\mathbb N}\subset L^p(\Omega)$
satisfy $u_j\to u$ in $L^p(\Omega)$ for some $u\in L^p(\Omega)$. Then the
following assertions hold.
\begin{itemize}
\item[$\mathrm{(i)}$] If $p\in(1,\infty)$, then
\begin{align*}
\varliminf_{j\to\infty}F_{\lambda_j,p,\gamma}\left(u_j;\Omega\right)
\geq c_{N,p,\gamma}\int_\Omega|\nabla u|^p\,dx,
\end{align*}
where the right-hand side is understood to be $\infty$ if
$u\notin W^{1,p}(\Omega)$ and where  $c_{N,p,\gamma}$ is a positive constant depending only on $N,p,$ and $\gamma.$
\item[$\mathrm{(ii)}$] If $p=1$, then
\begin{align*}
\varliminf_{j\to\infty}F_{\lambda_j,1,\gamma}\left(u_j;\Omega\right)
\geq c_{N,\gamma}|Du|(\Omega),
\end{align*}
where the right-hand side is understood to be $\infty$ if
$u\notin BV(\Omega)$ and where  $c_{N,\gamma}$ is a positive constant depending only on $N$ and $\gamma.$
\end{itemize}
\end{proposition}
In the remainder of this subsection, we prove Proposition~\ref{prop-coerc}
by first establishing the one-dimensional case and then extending the result
to higher dimensions through a slicing argument.
To this end,
we need
the  following lemma, which follows from \cite[Proposition~3.2]{agmp20} with $p:=1$.
\begin{lemma}\label{lem-critosc}
Let $I\subset\mathbb R$ be a bounded open interval. Let
$\{\delta_j\}_{j\in\mathbb N}\subset(0,\infty)$ satisfy
$\delta_j\to0$, and let $\{f_j\}_{j\in\mathbb N}\subset L^1(I)$
satisfy $f_j\to f$ in $L^1(I)$
for some $f\in L^1(I).$
Then there exists a positive
constant $c$ such that
\begin{align*}
\varliminf_{j\to\infty}F_{\delta_j,1,-1}\left(f_j;I\right)
\geq c\left(\operatorname*{ess\,sup}_I f-
\operatorname*{ess\,inf}_I f\right),
\end{align*}
where the right-hand side is understood to be $\infty$ when the essential
oscillation is infinite.
\end{lemma}

Applying  Lemma \ref{lem-critosc}, we can obtain the following result.
\begin{lemma}\label{lem-seqosc}
Let $p\in(1,\infty)$, $\gamma\in(-p,-1]$, and
$I\subset\mathbb R$ be a bounded  open interval. Let
$\{\lambda_j\}_{j\in\mathbb N}\subset(0,\infty)$ satisfy
$\lambda_j\to0$, and let $\{f_j\}_{j\in\mathbb N}\subset L^1(I)$
satisfy $f_j\to f$ in $L^1(I)$ for some $f\in L^1(I)$.
Then there exists a positive constant $c_p$, depending only on $p$, such that
\begin{align}\label{xuzhen}
\varliminf_{j\to\infty}F_{\lambda_j,p,\gamma}\left(f_j;I\right)
\geq\frac{c_p}{|I|^{p-1}}
\left(\operatorname*{ess\,sup}_If-
\operatorname*{ess\,inf}_If\right)^p,
\end{align}
where the right-hand side is understood to be $\infty$ when the essential
oscillation is infinite.
\end{lemma}

\begin{proof}
Fix  $j\in\mathbb N$ and  let $r_j\in (0,\infty)$ be chosen later.
Since $1+\gamma/p>0$,
it follows that,
for any $x,y\in I$ with  $0<|x-y|\leq r_j$,
\begin{align*}
\lambda_j|x-y|^{1+\gamma/p}
\leq\lambda_jr_j^{1+\gamma/p}.
\end{align*}
Consequently,
\begin{align*}
&\left\{(x,y)\in I\times I:
0<|x-y|\leq r_j,|f_j(x)-f_j(y)|>
\lambda_jr_j^{1+\gamma/p}\right\}\\
&\quad\subset
\left\{(x,y)\in I\times I: x\not=y,|f_j(x)-f_j(y)|>
\lambda_j|x-y|^{1+\gamma/p}\right\}.
\end{align*}
Moreover, by $\gamma+1\leq0$, we conclude that,
for any $x,y\in I$ with  $0<|x-y|\leq r_j$,
\begin{align*}
|x-y|^{\gamma-1}=|x-y|^{\gamma+1}|x-y|^{-2}\geq r_j^{\gamma+1}|x-y|^{-2}.
\end{align*}
Using these comparisons and the identity
\begin{align*}
\lambda_j^pr_j^{\gamma+1}
=\left(\lambda_jr_j^{\gamma/p}\right)^{p-1}
\lambda_jr_j^{1+\gamma/p},
\end{align*}
we obtain
\begin{align}\label{oscnear}
F_{\lambda_j,p,\gamma}\left(f_j;I\right)
&=\lambda_j^p\underset{\scriptscriptstyle |f_j(x)-f_j(y)|>\lambda_j|x-y|^{1+\gamma/p}
}{\int_I\int_I}|x-y|^{\gamma-1}\,dy\,dx\notag\\
&\geq\lambda_j^pr_j^{\gamma+1}
\underset{\substack{\scriptscriptstyle 0<|x-y|\leq r_j\\
\scriptscriptstyle |f_j(x)-f_j(y)|>\lambda_jr_j^{1+\gamma/p}}}{\int_I\int_I}
\frac{dy\,dx}{|x-y|^2}.
\end{align}
Let $\delta_j:=\lambda_jr_j^{1+\gamma/p}.$
Then
\begin{align}\label{huiwu}
F_{\delta_j,1,-1}\left(f_j;I\right)=\delta_j
\underset{\substack{\scriptscriptstyle 0<|x-y|\leq r_j\\
\scriptscriptstyle |f_j(x)-f_j(y)|>\delta_j}}{\int_I\int_I}
\frac{dy\,dx}{|x-y|^2}+\delta_j
\underset{\substack{\scriptscriptstyle |x-y|>r_j\\
\scriptscriptstyle |f_j(x)-f_j(y)|>\delta_j}}{\int_I\int_I}
\frac{dy\,dx}{|x-y|^2}
=:E_j+F_j.
\end{align}
Note that
\begin{align*}
F_j&\leq\delta_j\underset{\scriptscriptstyle |x-y|>r_j}
{\int_I\int_{\mathbb R}}
\frac{dy\,dx}{|x-y|^2}=2|I|\lambda_jr_j^{\gamma/p}.
\end{align*}
From this, \eqref{huiwu}, and  \eqref{oscnear}, we deduce that
\begin{align}\label{osccomp}
F_{\lambda_j,p,\gamma}\left(f_j;I\right)\geq
\lambda_j^{p-1}r_{j}^{\gamma-\gamma/p}E_j\geq\lambda_j^{p-1}r_{j}^{\gamma-\gamma/p}
\left[F_{\delta_j,1,-1}\left(f_j;I\right)-2|I|\lambda_jr_j^{\gamma/p}\right].
\end{align}
If $\operatorname*{ess\,sup}_If-\operatorname*{ess\,inf}_If=0$,
then \eqref{xuzhen} holds automatically. Let
\begin{align*}
0<A<\operatorname*{ess\,sup}_If-
\operatorname*{ess\,inf}_If\leq\infty.
\end{align*}
Now, we choose $r_j$ such that $2|I|\lambda_jr_j^{\gamma/p}=\frac{cA}{2}.$ Then
\begin{align*}
r_j=\left(\frac{4|I|\lambda_j}{cA}\right)^{-\frac p \gamma}\to 0
\ \ \text{and}\ \ \delta_j= \lambda_jr_j^{1+\gamma/p}
=\frac{cA}{4|I|}r_j\to0
\end{align*}
as $j\to\infty$. By this and  \eqref{osccomp}, we find that
\begin{align*}
F_{\lambda_j,p,\gamma}\left(f_j;I\right)\geq\left(\frac{cA}{4|I|}\right)^{p-1}
\left[F_{\delta_j,1,-1}(f_j;I)-\frac{cA}{2}\right].
\end{align*}
Letting $j\to\infty$ in the above inequality and using  Lemma \ref{lem-critosc},  we conclude that
\begin{align*}
\varliminf_{j\to\infty}F_{\lambda_j,p,\gamma}\left(f_j;I\right)
\geq \left(\frac{cA}{4|I|}\right)^{p-1}
\left(cA-\frac{cA}{2}\right)=\frac{c^p}{2\cdot4^{p-1}}\frac{A^p}{|I|^{p-1}}.
\end{align*}
Letting $A\to \operatorname*{ess\,sup}_If-
\operatorname*{ess\,inf}_If$ in the estimate as above, we then find that \eqref{xuzhen} holds. This  completes
the proof of Lemma \ref{lem-seqosc}.
\end{proof}

For the remaining range $\gamma<-p$, we use the following estimate
for averaged difference quotients.

\begin{lemma}\label{lem-diffbd}
Let $p\in[1,\infty)$, $\gamma\in(-\infty,-p)$, and $I:=(a,b)$
be a bounded  open interval. Then, for any $g\in L^\infty(I)$,
 $\lambda\in(0,\infty)$, and  $r\in(0,\frac{|I|}{2})$,
\begin{align}\label{chashang}
&\int_r^{2r}\int_a^{b-h}
\left|\frac{g(x+h)-g(x)}h\right|^p\,dx\,\frac{dh}h
\leq\frac{2^{-\gamma}-1}{2}F_{\lambda,p,\gamma}(g;I)
+\frac{2^\gamma-1}{\gamma}\lambda^p|I|r^\gamma.
\end{align}
\end{lemma}

\begin{proof}
Let $g\in L^\infty(I)$, $\lambda\in(0,\infty)$, and  $r\in(0,\frac{|I|}{2})$. Then
\begin{align*}
F_{\lambda,p,\gamma}(g;I)
=2\lambda^p\int_0^{|I|}h^{\gamma-1}\int_a^{b-h}
\mathbf1_{\{|g(x+h)-g(x)|>\lambda h^{1+\gamma/p}\}}\,dx\,dh.
\end{align*}
Let
$s_+:=\max\{s,0\}$ for any $s\in\mathbb R$.
Then, for any $t\in\mathbb{R}$ and $h\in (0,\infty)$,
\begin{align*}
|t|^p
\leq\lambda^ph^{p+\gamma}
+\left(|t|^p-\lambda^ph^{p+\gamma}\right)_+.
\end{align*}
Consequently,
\begin{align*}
&\int_r^{2r}\int_a^{b-h}
\left|\frac{g(x+h)-g(x)}h\right|^p\,dx\,\frac{dh}h\\
&\quad\leq\int_r^{2r}\int_a^{b-h}
\frac{(|g(x+h)-g(x)|^p-\lambda^ph^{p+\gamma})_+}{h^{p+1}}
\,dx\,dh
+\frac{2^\gamma-1}{\gamma}\lambda^p|I|r^\gamma.
\end{align*}
Thus, to prove \eqref{chashang}, it suffices  to show that
\begin{align}\label{wangji}
\int_r^{2r}\int_a^{b-h}
\frac{(|g(x+h)-g(x)|^p-\lambda^ph^{p+\gamma})_+}{h^{p+1}}
\,dx\,dh
\leq\frac{2^{-\gamma}-1}{2}F_{\lambda,p,\gamma}(g;I).
\end{align}
Using the estimate  $|A+B|^p\leq2^{p-1}(|A|^p+|B|^p)$ and the definition of $s_+$,
we find that, for any $A,B\in\mathbb R$ and $h\in(0,\infty)$,
\begin{align*}
\left(|A+B|^p-\lambda^ph^{p+\gamma}\right)_+
&\leq2^{p-1}\left(|A|^p-\lambda^p(h/2)^{p+\gamma}\right)_+
+2^{p-1}\left(|B|^p-\lambda^p(h/2)^{p+\gamma}\right)_+\\
&\quad+(2^{-\gamma}-1)\lambda^ph^{p+\gamma}
\mathbf1_{\left\{|A+B|>\lambda h^{1+\gamma/p}\right\}}.
\end{align*}
Applying this inequality with $A:=g(x+h/2)-g(x)$ and $B:=g(x+h)-g(x+h/2)$,
we conclude  that, for any $h\in (0,|I|)$,
\begin{align*}
&\int_a^{b-h}
(|g(x+h)-g(x)|^p-\lambda^ph^{p+\gamma})_+\,dx\\
&\quad\leq
2^p\int_a^{b-h/2}(|g(x+h/2)-g(x)|^p
-\lambda^p(h/2)^{p+\gamma})_+\,dx\\
&\qquad+(2^{-\gamma}-1)\lambda^ph^{p+\gamma}
\int_a^{b-h}\mathbf1_{\{|g(x+h)-g(x)|>\lambda h^{1+\gamma/p}\}}\,dx,
\end{align*}
which further implies that
\begin{align}\label{yishan}
&\int_r^{2r}\int_a^{b-h}\frac{(|g(x+h)-g(x)|^p-\lambda^ph^{p+\gamma})_+}{h^{p+1}}\,dx\,dh\notag\\
&\quad\leq\int_{r/2}^r\int_a^{b-s}\frac{(|g(x+s)-g(x)|^p-\lambda^ps^{p+\gamma})_+}{s^{p+1}}\,dx\,ds\notag\\
&\quad\quad+(2^{-\gamma}-1)\lambda^p\int_r^{2r}h^{\gamma-1}\int_a^{b-h}
\mathbf1_{\{|g(x+h)-g(x)|>\lambda h^{1+\gamma/p}\}}\,dx\,dh.
\end{align}
Replacing $r$ with $2^{-k}r$ in \eqref{yishan} for $k\in\{0,\ldots,K-1\}$
and summing over $k$,  we obtain, for any $K\in\mathbb N$,
\begin{align}\label{zhida}
&\int_r^{2r}\int_a^{b-h}\frac{(|g(x+h)-g(x)|^p-\lambda^ph^{p+\gamma})_+}{h^{p+1}}
\,dx\,dh\notag\\
&\quad\leq\int_{2^{-K}r}^{2^{1-K}r}\int_a^{b-h}
\frac{(|g(x+h)-g(x)|^p-\lambda^ph^{p+\gamma})_+}{h^{p+1}}\,dx\,dh\notag\\
&\quad\quad+(2^{-\gamma}-1)\lambda^p\int_{2^{1-K}r}^{2r}h^{\gamma-1}\int_a^{b-h}
\mathbf1_{\{|g(x+h)-g(x)|>\lambda h^{1+\gamma/p}\}}\,dx\,dh.
\end{align}
Since $1+\gamma/p<0$,
we can  choose $K$ sufficiently large such that
$\lambda(2^{1-K}r)^{1+\gamma/p}\geq2\|g\|_{L^\infty(I)}.$
Then, for any $h\in(2^{-K}r,2^{1-K}r)$ and almost every $x\in(a,b-h)$,
\begin{align*}
|g(x+h)-g(x)|\leq2\|g\|_{L^\infty(I)}\leq\lambda(2^{1-K}r)^{1+\gamma/p}
\leq\lambda h^{1+\gamma/p},
\end{align*}
and hence
\begin{align*}
\int_{2^{-K}r}^{2^{1-K}r}\int_a^{b-h}
\frac{(|g(x+h)-g(x)|^p-\lambda^ph^{p+\gamma})_+}{h^{p+1}}
\,dx\,dh=0.
\end{align*}
From this and \eqref{zhida}, we deduce that
\begin{align*}
&\int_r^{2r}\int_a^{b-h}\frac{(|g(x+h)-g(x)|^p-\lambda^ph^{p+\gamma})_+}{h^{p+1}}\,dx\,dh\\
&\quad\leq(2^{-\gamma}-1)\lambda^p\int_0^{|I|}h^{\gamma-1}\int_a^{b-h}
\mathbf1_{\{|g(x+h)-g(x)|>\lambda h^{1+\gamma/p}\}}\,dx\,dh
=\frac{2^{-\gamma}-1}{2}F_{\lambda,p,\gamma}(g;I).
\end{align*}
This proves \eqref{wangji} and hence completes the proof of Lemma \ref{lem-diffbd}.
\end{proof}

Let  $\Omega$ be a measurable set, $\sigma\in\mathbb S^{N-1}$,
and $z\in\sigma^\perp:=\{y\in \mathbb{R}^N:\ y\cdot \sigma=0\}$.
Define
\begin{align*}
\Omega_{z,\sigma}:=\left\{t\in\mathbb R:z+t\sigma\in\Omega\right\}.
\end{align*}
Let $v$ be a nonnegative measurable function on $\Omega$
and, for any
$t \in \Omega_{z,\sigma}$,
define $v_{z,\sigma}(t):=v(z+t\sigma)$. Then
\begin{align}\label{wangzheni}
\int_\Omega v(x)\,dx=\int_{\sigma^\perp}\int_{\Omega_{z,\sigma}}
v(z+t\sigma)\,dt\,dz.
\end{align}
By this and the polar coordinate decomposition, we have the following
slicing formula; we omit
the details here.

\begin{lemma}\label{lem-slicing}
Let $N\geq2$, $p\in[1,\infty)$, $\gamma\in\mathbb{R}\setminus\{0\}$, and
$\Omega\subset\mathbb R^N$ be an open set. Then, for any
measurable function $g:\Omega\to\mathbb R$ and any $\lambda\in(0,\infty)$,
\begin{align}\label{qiege}
F_{\lambda,p,\gamma}(g;\Omega)
=\frac12\int_{\mathbb S^{N-1}}\int_{\sigma^\perp}
F_{\lambda,p,\gamma}(g_{z,\sigma};\Omega_{z,\sigma})\,dz\,d\sigma.
\end{align}
\end{lemma}

Now, we prove Proposition \ref{prop-coerc}.
\begin{proof}[Proof of Proposition \ref{prop-coerc}]
Let
\begin{align*}
L:=\varliminf_{j\to\infty}F_{\lambda_j,p,\gamma}\left(u_j;\Omega\right).
\end{align*}
If $L=\infty$, there is nothing to prove. We now assume $L<\infty$.
Passing to a subsequence and relabelling, we may assume that
\begin{align}\label{kongge}
\lim_{j\to\infty}F_{\lambda_j,p,\gamma}\left(u_j;\Omega\right)=L
\ \ \text{and}\ \
\sum_{j=1}^\infty\|u_j-u\|_{L^p(\Omega)}^p<\infty.
\end{align}
Let $p'$ be the conjugate exponent of $p$.
We claim that, for any $\varphi\in C_{\rm c}^\infty(\Omega;\mathbb R^N)$ (the set of all
$\mathbb R^N$-valued infinitely differentiable functions on $\Omega$ with compact support),
\begin{align}\label{duoceshi}
\left|\int_\Omega u\operatorname{div}\varphi\,dx\right|
\lesssim L^{1/p}\|\varphi\|_{L^{p'}(\Omega;\mathbb R^N)},
\end{align}
where the implicit positive constant depends only on $N,p,$ and $\gamma$.
If this claim holds, the conclusions of the present proposition follow from
$L^p$ duality when $p\in(1,\infty)$ and from the structure theorem for $BV$ functions
when $p=1$.
We first prove \eqref{duoceshi} in the case $N=1$.
In this case, let $I:=\Omega$. Without loss of generality, we may assume that
$\lim_{j\to\infty}u_j(x)=u(x)$ for almost every $x\in I$.
We show \eqref{duoceshi} by considering two cases for $\gamma$.

\textit{Case 1: $-p<\gamma\leq-1$.} In this case, we have  $p\in(1,\infty)$.
Observe  that, if $J\subset I$ is an open interval and $t\in J$ is a
Lebesgue point of $u$, then
\begin{align}\label{kuxiaozhe}
\operatorname*{ess\,inf}_J u\leq u(t)=\lim_{\rho\to 0^+}\frac1{2\rho}
\int_{t-\rho}^{t+\rho}u(y)\,dy\leq\operatorname*{ess\,sup}_J u.
\end{align}
Let $I'\subset I$, and let $0<h<\frac12\mathrm{dist}(I',\partial I)$.
For any $x\in I'$, let $J_{x,h}:=(x-h,x+2h)\subset I.$
Using \eqref{kuxiaozhe} and applying  Lemma \ref{lem-seqosc}, we obtain,
for almost every $x\in I',$
\begin{align}\label{quan}
\left|\frac{u(x+h)-u(x)}h\right|^p
&\leq h^{-p}\left(\operatorname*{ess\,sup}_{J_{x,h}}u-
\operatorname*{ess\,inf}_{J_{x,h}}u\right)^p
\lesssim h^{-1}\varliminf_{j\to\infty}F_{\lambda_j,p,\gamma}(u_j;J_{x,h}).
\end{align}
On the other hand, by Tonelli's theorem, we find that
\begin{align*}
&\int_{I'}F_{\lambda_j,p,\gamma}(u_j;J_{x,h})\,dx\\
&\quad=\lambda_j^p\underset{\scriptscriptstyle |u_j(y)-u_j(z)|>\lambda_j|y-z|^{{1+\gamma/p}}}
{\int_I\int_I}|y-z|^{\gamma-1}\left[\int_{I'}\mathbf1_{J_{x,h}}(y)
\mathbf1_{J_{x,h}}(z)\,dx\right]\,dz\,dy
\leq3hF_{\lambda_j,p,\gamma}(u_j;I).
\end{align*}
From this and \eqref{quan}, it follows that
\begin{align*}
\int_{I'}\left|\frac{u(x+h)-u(x)}h\right|^p\,dx&\lesssim
h^{-1}\int_{I'}\varliminf_{j\to\infty}F_{\lambda_j,p,\gamma}(u_j;J_{x,h})\,dx
\lesssim L,
\end{align*}
which further implies that, for any $\varphi\in C^\infty_{\rm c}(I)$ supported in $I'$,
\begin{align*}
\left|\int_I u\varphi'\,dx\right|^p&=\lim_{h\to 0^+}
\left|\int_{I'}\frac{u(x+h)-u(x)}h\varphi(x)\,dx\right|^p\\
&\leq\|\varphi\|_{L^{p'}(I)}^p\varlimsup_{h\to 0^+}
\int_{I'}\left|\frac{u(x+h)-u(x)}h\right|^p\,dx
\lesssim L\|\varphi\|_{L^{p'}(I)}^p.
\end{align*}
This proves \eqref{duoceshi} in this case.

\textit{Case 2: $\gamma<-p$.}
For any given $M\in (0,\infty)$, let $T_M:\mathbb R\to\mathbb R$ be the symmetric
truncation defined by setting, for any $t\in\mathbb{R}$,
\begin{align}\label{jieding}
T_M(t):=
\begin{cases}
-M, & t\in(-\infty,-M),\\
t, & t\in[-M,M],\\
M, & t\in(M,\infty).
\end{cases}
\end{align}
Then, for any $s,t\in\mathbb{R},$
\begin{align}\label{deng}
|T_M(t)-T_M(s)|
\leq
|s-t|.
\end{align}
By \eqref{deng}, we conclude that, for any $j\in\mathbb N$,
\begin{align*}
F_{\lambda_j,p,\gamma}(T_Mu_j;I)&\leq F_{\lambda_j,p,\gamma}(u_j;I)
\end{align*}
and
\begin{align*}
\|T_Mu_j-T_Mu\|_{L^p(I)}\leq\|u_j-u\|_{L^p(I)}\to 0
\end{align*}
as $j\to\infty$. Let  $I:=(a,b)$ and  $r\in(0,\frac{|I|}{2})$.
Recall  that $\lim_{j\to\infty}u_j(x)=u(x)$ for almost every $x\in I.$
From this and \eqref{deng}, we infer that, for almost every $(x,h)\in I\times (0,\infty)$ with
$x,x+h\in I$,
\begin{align*}
\lim_{j\to\infty}\frac{T_Mu_j(x+h)-T_Mu_j(x)}{h}=\frac{T_Mu(x+h)-T_Mu(x)}{h}.
\end{align*}
Using this  and  applying Lemma \ref{lem-diffbd}, we obtain
\begin{align}\label{junchash}
&\int_r^{2r}\int_a^{b-h}\left|\frac{T_Mu(x+h)-T_Mu(x)}h\right|^p\,dx\,\frac{dh}h
\nonumber\\
&\quad\leq\varliminf_{j\to\infty}\int_r^{2r}\int_a^{b-h}
\left|\frac{T_Mu_j(x+h)-T_Mu_j(x)}h\right|^p\,dx\,\frac{dh}h
\nonumber\\
&\quad\leq\varliminf_{j\to\infty}\left[
\frac{2^{-\gamma}-1}{2}F_{\lambda_j,p,\gamma}(u_j;I)
+\frac{2^\gamma-1}{\gamma}\lambda_j^p|I|r^\gamma
\right]=\frac{2^{-\gamma}-1}{2}L.
\end{align}
It is obvious that, for any $\varphi\in C^\infty_{\rm c}(I)$,
\begin{align*}
\left|\int_I T_Mu\varphi'\,dx\right|^p
=
\lim_{h\to0^+}
\left|\int_a^{b-h}
\frac{T_Mu(x+h)-T_Mu(x)}h\varphi(x)\,dx\right|^p.
\end{align*}
This, combined with the fact that $\int_r^{2r}\frac{dh}h=\log2$, H\"older's inequality, and
\eqref{junchash}, implies that, for any $\varphi\in C^\infty_{\rm c}(I)$,
\begin{align*}
&\left|\int_I T_Mu\varphi'\,dx\right|^p\\
&\quad=\frac1{\log2}\lim_{r\to0^+}\int_r^{2r}
\left|\int_a^{b-h}
\frac{T_Mu(x+h)-T_Mu(x)}h\varphi(x)\,dx\right|^p\frac{dh}h\\
&\quad\leq\frac{\|\varphi\|_{L^{p'}(I)}^p}{\log2}
\varlimsup_{r\to0^+}\int_r^{2r}\int_a^{b-h}
\left|\frac{T_Mu(x+h)-T_Mu(x)}h\right|^p\,dx\,\frac{dh}h
\lesssim L\|\varphi\|_{L^{p'}(I)}^p.
\end{align*}
Letting $M\to\infty$ in the estimate as above, we conclude that \eqref{duoceshi} holds in this case.

Assume now that $N\geq2$. We show \eqref{duoceshi} by applying the one-dimensional
estimate on almost all line sections of $\Omega$ and integrating over
their positions and directions.
Fix $\varphi\in C_{\rm c}^\infty(\Omega;\mathbb R^N)$.
Note that, for any $\sigma\in\mathbb S^{N-1}$ and $i,k\in\{1,\ldots,N\}$,
\begin{align*}
\partial_\sigma(\varphi\cdot\sigma)
=\sum_{i,k=1}^N\sigma_i\sigma_k\partial_i\varphi_k
\ \ \text{and}\ \
\int_{\mathbb S^{N-1}}\sigma_i\sigma_k\,d\sigma
=\frac{|\mathbb S^{N-1}|}{N}\delta_{ik},
\end{align*}
where  $\delta_{ik}$ denotes the Kronecker delta.
Consequently,
\begin{align*}
\int_{\mathbb S^{N-1}}
\partial_\sigma(\varphi(x)\cdot\sigma)\,d\sigma
=\sum_{i,k=1}^N\partial_i\varphi_k(x)
\int_{\mathbb S^{N-1}}\sigma_i\sigma_k\,d\sigma
=\frac{|\mathbb S^{N-1}|}{N}\operatorname{div}\varphi(x).
\end{align*}
This,  together with the fact that
\begin{align*}
\frac{d}{dt}\left(\varphi(z+t\sigma)\cdot\sigma\right)
=\partial_\sigma(\varphi\cdot\sigma)(z+t\sigma),
\end{align*}
implies that
\begin{align}\label{fangcesh}
&\left|\int_\Omega u\operatorname{div}\varphi\,dx\right|\notag\\
&\quad=\frac{N}{|\mathbb S^{N-1}|}
\left|\int_{\mathbb S^{N-1}}\int_{\sigma^\perp}
\int_{\Omega_{z,\sigma}}u(z+t\sigma)
\partial_\sigma(\varphi\cdot\sigma)(z+t\sigma)
\,dt\,dz\,d\sigma\right|\notag\\
&\quad\leq\frac{N}{|\mathbb S^{N-1}|}
\int_{\mathbb S^{N-1}}\int_{\sigma^\perp}
\left|\int_{\Omega_{z,\sigma}}u_{z,\sigma}(t)
\frac{d}{dt}\left(\varphi(z+t\sigma)\cdot\sigma\right)
\,dt\right|\,dz\,d\sigma.
\end{align}

Next, we prove that, for
almost every $(\sigma, z)\in (\mathbb S^{N-1}, \sigma^{\perp}) $ and any
$\psi\in C_{\rm c}^\infty(\Omega_{z,\sigma})$,
\begin{align}\label{qieceshi}
\left|\int_{\Omega_{z,\sigma}}u_{z,\sigma}\psi'\,dt\right|
\lesssim
\left[\varliminf_{j\to\infty}
F_{\lambda_j,p,\gamma}\left((u_j)_{z,\sigma};\Omega_{z,\sigma}\right)\right]^{\frac1p}
\|\psi\|_{L^{p'}(\Omega_{z,\sigma})}.
\end{align}
Assume for the moment that  \eqref{qieceshi} holds.
By \eqref{qieceshi}, \eqref{fangcesh}, and H\"older's inequality,
we conclude that
\begin{align}\label{moment}
&\left|\int_\Omega u\operatorname{div}\varphi\,dx\right|\notag\\
&\quad\lesssim\left[\int_{\mathbb S^{N-1}}\int_{\sigma^\perp}
\varliminf_{j\to\infty}
F_{\lambda_j,p,\gamma}\left((u_j)_{z,\sigma};\Omega_{z,\sigma}\right)
\,dz\,d\sigma\right]^{\frac1p}\notag\\
&\quad\quad\times\left[\int_{\mathbb S^{N-1}}\int_{\sigma^\perp}
\|\varphi(z+\cdot\,\sigma)\cdot\sigma\|_{L^{p'}(\Omega_{z,\sigma})}^{p'}
\,dz\,d\sigma\right]^{\frac1{p'}}.
\end{align}
Furthermore, from \eqref{wangzheni}, we deduce that
\begin{align}\label{houzhi}
&\left[\int_{\mathbb S^{N-1}}\int_{\sigma^\perp}
\|\varphi(z+\cdot\,\sigma)\cdot\sigma\|_{L^{p'}(\Omega_{z,\sigma})}^{p'}
\,dz\,d\sigma\right]^{\frac1{p'}}\notag\\
&\quad=\left[\int_{\mathbb S^{N-1}}\int_\Omega
|\varphi(x)\cdot\sigma|^{p'}\,dx\,d\sigma\right]^{\frac1{p'}}
\lesssim\|\varphi\|_{L^{p'}(\Omega;\mathbb R^N)}.
\end{align}
Using Fatou's lemma and \eqref{qiege}, we find that
\begin{align*}
&\int_{\mathbb S^{N-1}}\int_{\sigma^\perp}
\varliminf_{j\to\infty}F_{\lambda_j,p,\gamma}\left((u_j)_{z,\sigma};
\Omega_{z,\sigma}\right)\,dz\,d\sigma\\
&\quad\leq\varliminf_{j\to\infty}\int_{\mathbb S^{N-1}}\int_{\sigma^\perp}
F_{\lambda_j,p,\gamma}\left((u_j)_{z,\sigma};\Omega_{z,\sigma}\right)\,dz\,d\sigma
=2\varliminf_{j\to\infty}F_{\lambda_j,p,\gamma}\left(u_j;\Omega\right)=2L.
\end{align*}
From this, \eqref{houzhi}, and \eqref{moment}, we infer that
\eqref{duoceshi} holds,
where the usual modification is made when $p=1$.

Now, it remains to show \eqref{qieceshi}.
By \eqref{wangzheni} and \eqref{kongge}, we obtain
\begin{align*}
\int_{\mathbb S^{N-1}}\int_{\sigma^\perp}
\sum_{j=1}^\infty\left\|(u_j)_{z,\sigma}-u_{z,\sigma}\right\|_{L^p(\Omega_{z,\sigma})}^p
\,dz\,d\sigma=\left|\mathbb S^{N-1}\right|\sum_{j=1}^\infty
\|u_j-u\|_{L^p(\Omega)}^p<\infty.
\end{align*}
Consequently, for almost every $(\sigma,z)\in \mathbb{S}^{N-1}\times \sigma^\perp$,
$(u_j)_{z,\sigma}\to u_{z,\sigma}$ in $L^p(\Omega_{z,\sigma})$ as $j\to\infty$.
Since $\Omega$ is a bounded open set,
it follows that $\Omega_{z,\sigma}$
is a bounded open subset of $\mathbb R$. Thus, $\Omega_{z,\sigma}=\bigcup_k I_k,$
where $\{I_k\}_k$ is a sequence of disjoint bounded open intervals.
This, combined with \eqref{duoceshi} in the case $N=1$ and H\"older's inequality, implies that
\begin{align*}
\left|\int_{\Omega_{z,\sigma}}u_{z,\sigma}\psi'\,dt\right|
&\leq\sum_k\left|\int_{I_k}u_{z,\sigma}\psi'\,dt\right|
\lesssim\sum_k\left[\varliminf_{j\to\infty}
F_{\lambda_j,p,\gamma}\left((u_j)_{z,\sigma};I_k\right)\right]^{\frac1p}\|\psi\|_{L^{p'}(I_k)}\\
&\lesssim\left[\sum_k\varliminf_{j\to\infty}
F_{\lambda_j,p,\gamma}\left((u_j)_{z,\sigma};I_k\right)\right]^{\frac1p}
\|\psi\|_{L^{p'}(\Omega_{z,\sigma})}\\
&\lesssim\left[\varliminf_{j\to\infty}
F_{\lambda_j,p,\gamma}\left((u_j)_{z,\sigma};\Omega_{z,\sigma}\right)\right]^{\frac1p}
\|\psi\|_{L^{p'}(\Omega_{z,\sigma})}.
\end{align*}
This proves \eqref{qieceshi}, which completes the proof of Proposition \ref{prop-coerc}.
\end{proof}

\section{The $\Gamma$-Liminf Inequality on Bounded Domains}
\label{sec-liminf}

This section is devoted to the optimal $\Gamma$-liminf inequalities on bounded domains.
In Subsection~\ref{sec-pinf}, we prove the $\Gamma$-liminf inequality for $p\in(1,\infty)$ (Theorem \ref{thm-pliminf}) by using affine blow-ups,
the lower semicontinuity of the cell functional, and a covering argument.
In Subsection~\ref{proofone}, we prove the $\Gamma$-liminf inequality  for $p=1$ (Theorem~\ref{thm-onelinf}) using two results:
the one-dimensional limits of rescaled $BV$ functions in Lemma~\ref{lem-tangent}
and the comparison for monotone profiles in Proposition~\ref{prop-profend}.
We state these two results before the proof of Theorem~\ref{thm-onelinf} and prove them in the last two subsections, namely
Subsections \ref{sub4.3} and \ref{sub4.4}.

\subsection{The Case $p\in(1,\infty)$ \label{sec-pinf}}

In this subsection, we show the following  $\Gamma$-liminf inequality for $p\in(1,\infty)$.

\begin{theorem}\label{thm-pliminf}
Let $p\in(1,\infty)$,
$\gamma\in\left[(-\infty,-1]\setminus\{-p\}\right]\cup(0,\infty)$, and
$\Omega\subset\mathbb R^N$ be a bounded open interval when $N=1$ or a
bounded Lipschitz domain when $N\geq2$.
Let
$\{\lambda_j\}_{j\in\mathbb N}\subset(0,\infty)$ satisfy
$\lambda_j\to\lambda_\gamma$, and let $\{u_j\}_{j\in\mathbb N}\subset L^p(\Omega)$
satisfy $u_j\to u$ in $L^p(\Omega)$ for some $u\in L^p(\Omega)$ as $j\to\infty$. Then
\begin{align}\label{xiaji}
\varliminf_{j\to\infty}F_{\lambda_j,p,\gamma}\left(u_j;\Omega\right)
\geq C_{N,p,\gamma}^{\rm cell}\int_\Omega|\nabla u|^p\,dx,
\end{align}
where the right-hand side is understood to be $\infty$ if
$u\notin W^{1,p}(\Omega)$.
\end{theorem}

For any $R\in\mathcal O(N)$, $x\in\mathbb R^N$, and $r\in(0,\infty)$, let
\begin{align}\label{eq-rotcube}
Q_r^R(x):=x+rRQ_0.
\end{align}
The following Morse type covering lemma is  a special case of
 \cite[Theorem 5.51]{afp00}.
\begin{lemma}\label{lem:morse}
Let $\mu$ be a finite positive Radon measure on a bounded open set
$U\subset\mathbb R^N$ and $A_0\subset U$ be a $\mu$-measurable set.  Let
$V$ be a family of open cubes of the form $Q_r^R(x)$, where $x\in A_0$, $r\in(0,\infty)$,
and $R\in\mathcal O(N)$. Assume that, for any cube $Q\in V$,
$\overline{Q}\subset U$ and $\mu(\partial Q)=0.$ Assume further that, for any given  $x\in A_0$
and   $\delta\in(0,\infty)$, the family $V$ contains a cube centered at $x$ and having
diameter smaller than $\delta$.  Then there exists a countable subfamily
$\{Q_i\}_{i\in\mathcal I}\subset V$ of pairwise disjoint cubes such that
\begin{align*}
 \mu\left(A_0\setminus\left[\bigcup_{i\in\mathcal I}Q_i\right]\right)=0.
\end{align*}
\end{lemma}

We prove Theorem~\ref{thm-pliminf} using  Lemma~\ref{lem:morse}.

\begin{proof}[Proof of Theorem \ref{thm-pliminf}]
If $u\notin W^{1,p}(\Omega)$, then \eqref{xiaji} follows directly from Lemma \ref{lem-gp}
when $\gamma\in(0,\infty)$ and from Proposition~\ref{prop-coerc} when $\gamma\in(-\infty,0)$.
Assume   $u\in W^{1,p}(\Omega)$ and define the measure $d\mu(x):=|\nabla u(x)|^p\,dx$
on $\Omega.$ Without loss of generality, we may assume  $\mu(\Omega)>0.$
Let
\begin{align*}
A_0:=\left\{x\in\Omega:\nabla u(x)\neq0,\ \lim_{\rho\to0}
\fint_{B(x,\rho)}|\nabla u(y)-\nabla u(x)|^p\,dy=0\right\}.
\end{align*}
By the Lebesgue differentiation theorem, we find that  $\mu(\Omega\setminus A_0)=0$.
Let $x\in A_0$ and $\xi:=\nabla u(x).$ Choose $R_x\in\mathcal O(N)$ such that $ R_xe_N=\frac{\xi}{|\xi|}$ and
choose $ r_x\in(0,\infty)$ such that $\overline{Q_r^{R_x}(x)}\subset\Omega$ for any $r\in(0, r_x)$.
Let $r\in (0,r_x)$ and $b_{x,r}:=\fint_{Q_r^{R_x}(x)}u(y)\,dy$. For any $j\in\mathbb{N}$ and $z\in Q_0$,  define
\begin{align*}
u_{x,r}(z):=\frac{u(x+rR_xz)-b_{x,r}}{|\xi|r}\ \ \text{and}\ \ u_{j,x,r}(z):
=\frac{u_j(x+rR_xz)-b_{x,r}}{|\xi|r}.
\end{align*}
Since  $u_j\to u$ in $L^p(\Omega)$, it follows that $u_{j,x,r}\to u_{x,r}$ in $L^p(Q_0)$ as $j\to\infty.$
Using this, the definition of $m_{p,\gamma}$, and  Lemma \ref{lem-scaling}, we conclude that
 \begin{align}\label{eq:yongyou}
\varliminf_{j\to\infty}F_{\lambda_j,p,\gamma}\left(u_j;Q_r^{R_x}(x)\right)&=|\xi|^pr^N
\varliminf_{j\to\infty}F_{\frac{\lambda_jr^{\gamma/p}}{|\xi|},p,\gamma}\left(u_{j,x,r};Q_0\right)\notag\\
&\geq|\xi|^pr^Nm_{p,\gamma}(u_{x,r}).
\end{align}
Fix $\varepsilon\in(0,\min\{1,C_{N,p,\gamma}^{\mathrm{cell}}/2\})$.
By the Lebesgue differentiation theorem again, we find that
\begin{align*}
\lim_{r\to 0^+}\frac{\mu(Q_r^{R_x}(x))}{r^N}=|\nabla u(x)|^p=|\xi|^p.
\end{align*}
Thus, after decreasing $r_x$ if necessary, we obtain
$\mu(Q_r^{R_x}(x))\leq(1+\varepsilon)|\xi|^p r^N.$
For any $y\in\Omega$, let $f(y):=u(y)-\xi\cdot(y-x).$ From the definition of $b_{x,r}$, we deduce that
\begin{align*}
\fint_{Q_r^{R_x}(x)}f(y)\,dy=\fint_{Q_r^{R_x}(x)}u(y)\,dy
-\xi\cdot\fint_{Q_r^{R_x}(x)}(y-x)\,dy=b_{x,r}.
\end{align*}
Using this and applying the  Poincar\'e inequality to $f$ on $Q_r^{R_x}(x)$, we conclude that
\begin{align}\label{eq:lushang}
\int_{Q_r^{R_x}(x)}\left|f(y)-b_{x,r}\right|^p\,dy\lesssim r^p\int_{Q_r^{R_x}(x)}
|\nabla u(y)-\xi|^p\,dy.
\end{align}
Since $R_xe_N=\frac{\xi}{|\xi|},$ it follows that $R_x^T\xi=|\xi|e_N.$
Consequently, for any $z\in Q_0$,
\begin{align*}
\xi\cdot(rR_xz)=r(R_x^T\xi)\cdot z=|\xi|r\,e_N\cdot z=|\xi|r z_N=|\xi|r\ell(z)
\end{align*}
and hence
\begin{align*}
u_{x,r}(z)-\ell(z)=\frac{u(x+rR_xz)-b_{x,r}}{|\xi|r}-\ell(z)
=\frac{u(x+rR_xz)-b_{x,r}-\xi\cdot(rR_xz)}{|\xi|r}.
\end{align*}
From this, the change of variables $y=x+rR_xz,$ and \eqref{eq:lushang}, we infer that
\begin{align*}
\|u_{x,r}-\ell\|_{L^p(Q_0)}^p&=\frac{1}{|\xi|^pr^p}\int_{Q_0}\left|u(x+rR_xz)-b_{x,r}
-\xi\cdot(rR_xz)\right|^p\,dz\\
&=\frac{1}{|\xi|^pr^{N+p}}\int_{Q_r^{R_x}(x)}\left|u(y)-b_{x,r}
-\xi\cdot(y-x)\right|^p\,dy\\
&\lesssim\frac{1}{|\xi|^pr^N}\int_{Q_r^{R_x}(x)}|\nabla u(y)-\xi|^p\,dy.
\end{align*}
Since $x$ is an $L^p$-Lebesgue  point of $\nabla u$, it follows that $\lim_{r\to 0}u_{x,r}=\ell$ in
$L^p(Q_0).$
Using  this and Lemma \ref{lem-mplsc}, we find that
\begin{align*}
C_{N,p,\gamma}^{\mathrm{cell}}=m_{p,\gamma}(\ell)\leq\varliminf_{r\to 0}m_{p,\gamma}(u_{x,r}).
\end{align*}
By this, after decreasing $r_x$ if necessary, we conclude that
\begin{align}\label{eq:mplocal}
C_{N,p,\gamma}^{\mathrm{cell}}-\varepsilon\leq m_{p,\gamma}(u_{x,r}).
\end{align}
From this, \eqref{eq:mplocal}, and \eqref{eq:yongyou}, we deduce that
\begin{align}
\varliminf_{j\to\infty}F_{\lambda_j,p,\gamma}\left(u_j;Q_r^{R_x}(x)
\right)\geq|\xi|^pr^Nm_{p,\gamma}(u_{x,r})
\geq\frac{C_{N,p,\gamma}^{\mathrm{cell}}-\varepsilon}{1+\varepsilon}
\mu\left(Q_r^{R_x}(x)\right).\label{eq:plocal}
\end{align}

Now, define
\begin{align*}
V_\varepsilon:=\left\{Q_r^{R_x}(x):x\in A_0,\ r\in(0,r_x)\right\}.
\end{align*}
Applying Lemma \ref{lem:morse} to $(\mu,\Omega, V_\varepsilon),$ we find that there  exists
a  countable subfamily $\{Q_i\}_{i\in\mathcal{I}}\subset V_\varepsilon$ of pairwise disjoint cubes such that
$\mu(A_0\setminus[\bigcup_{i\in\mathcal I}Q_i])=0.$
This, combined with the fact that $\mu(\Omega\setminus A_0)=0$, implies that
$\mu(\Omega\setminus[\bigcup_{i\in\mathcal I}Q_i])=0.$
Since $\mu(\Omega)<\infty$, we can choose a finite subfamily of $\{Q_i\}_{i\in\mathcal{I}}$,
relabeled as $\{Q_1,\ldots,Q_M\}$, such that
$\mu(\Omega\setminus[\bigcup_{i=1}^M Q_i])<\varepsilon.$
Using this and  \eqref{eq:plocal}, we conclude that
 \begin{align*}
\varliminf_{j\to\infty}F_{\lambda_j,p,\gamma}\left(u_j;\Omega\right)
&\geq\varliminf_{j\to\infty}\sum_{i=1}^M F_{\lambda_j,p,\gamma}\left(u_j;Q_i\right)\\
&\geq\frac{C_{N,p,\gamma}^{\mathrm{cell}}-\varepsilon}{1+\varepsilon}
\sum_{i=1}^M\mu(Q_i)>
\frac{C_{N,p,\gamma}^{\mathrm{cell}}-\varepsilon}{1+\varepsilon}[\mu(\Omega)-\varepsilon].
\end{align*}
Letting $\varepsilon\to0^+$, we obtain \eqref{xiaji}.
This completes the proof of Theorem \ref{thm-pliminf}.
\end{proof}

\subsection{The Case $p=1$ \label{proofone}}

We now show the $\Gamma$-liminf inequality for $p=1$.

\begin{theorem}\label{thm-onelinf}
Let $\gamma\in(-\infty,-1)\cup(0,\infty)$ and
$\Omega\subset\mathbb R^N$ be a bounded open interval when $N=1$ or a
bounded open set when $N\geq2$. Let
$\{\lambda_j\}_{j\in\mathbb N}\subset(0,\infty)$ satisfy
$\lambda_j\to\lambda_\gamma$, and let $\{u_j\}_{j\in\mathbb N}\subset L^1(\Omega)$
satisfy $u_j\to u$ in $L^1(\Omega)$ for some $u\in L^1(\Omega)$ as $j\to\infty$. Then
\begin{align}\label{duanxia}
\varliminf_{j\to\infty}F_{\lambda_j,1,\gamma}\left(u_j;\Omega\right)
\geq C_{N,1,\gamma}^{\rm cell}|Du|(\Omega),
\end{align}
where the right-hand side is understood to be $\infty$ when $u\notin BV(\Omega)$.
\end{theorem}

Let $\Omega\subset\mathbb R^N$ be an open set,
$\{u_k\}_{k\in\mathbb N}\subset BV(\Omega)$, and $u\in BV(\Omega)$.
Recall that $u_k$ \emph{converges strictly }to $u$ in
$BV(\Omega)$ if $\lim_{k\to\infty}u_k=u$ in $L^1(\Omega)$ and  $\lim_{k\to\infty}|Du_k|(\Omega)=|Du|(\Omega)$,
and recall the rotated cubes $Q_r^R(x)$ from \eqref{eq-rotcube}. Let $x\in\Omega$, $R\in\mathcal O(N)$, and $r\in(0,\infty)$ be such that
$Q_r^R(x)\subset\Omega$ and $|Du|(Q_r^R(x))\in (0,\infty)$. For any $z\in Q_0,$ let
\begin{align}\label{rescale}
u_{x,r}^R(z):=\frac{r^{N-1}}{|Du|(Q_r^R(x))}\left[u(x+rRz)-\fint_{Q_r^R(x)}u\,dy\right].
\end{align}
By a change of variables, we find that
\begin{align*}
\left|Du_{x,r}^R\right|(Q_0)=\frac{r^{N-1}}{|Du|(Q_r^R(x))}\frac{|Du|(Q_r^R(x))}{r^{N-1}}=1.
\end{align*}
Let $I_0:=(-\frac{1}{2},\frac{1}{2})$ and, for any $q\in BV(I_0)$ and $z\in Q_0$, define $W_{q}(z):=q(z_N)$.
The following lemma describes the one-dimensional limits used below.
\begin{lemma}\label{lem-tangent}
Let $\Omega\subset\mathbb R^N$ be a bounded open set and
$u\in BV(\Omega)$. Then, for $|Du|$-almost every $x\in\Omega$, there exist
an orthogonal matrix $R_x\in\mathcal O(N)$, a sequence
$\{r_k\}_{k\in\mathbb N}\subset(0,\infty)$ with $r_k\to 0$, and a
nondecreasing function $q_x\in BV(I_0)$ such that the following
assertions hold.
\begin{enumerate}[{\rm(i)}]
\item For any $k\in\mathbb{N},$ it holds that
$\overline{Q_{r_k}^{R_x}(x)}\subset\Omega$,
$|Du|(Q^{R_x}_{r_k}(x))\in (0,\infty)$, and
$|Du|(\partial Q_{r_k}^{R_x}(x))=0.$

\item  $|DW_{q_x}|(Q_0)=|Dq_x|(I_0)=1$.

\item $\lim_{k\to\infty}u_{x,r_k}^{R_x}=W_{q_x}$ in $L^1(Q_0)$.
Consequently, $u_{x,r_k}^{R_x}$
converges strictly to $W_{q_x}$ in $BV(Q_0)$ as $k\to\infty$.
\end{enumerate}
\end{lemma}

The following proposition gives a lower bound for the cell functional at a monotone profile.
\begin{proposition}\label{prop-profend}
Let $\gamma\in(-\infty,-1)\cup(0,\infty)$ and $q\in BV(I_0)$.  Assume that $q$ is nondecreasing.  Then
\begin{align}\label{fuduan}
 m_{1,\gamma}(W_q)\geq C_{N,1,\gamma}^{\mathrm{cell}}|Dq|(I_0),
\end{align}
where $I_0:=(-\frac{1}{2},\frac{1}{2})$ and, for any  $z\in Q_0,$ $W_q(z):=q(z_N).$
\end{proposition}

The proofs of Lemma \ref{lem-tangent} and Proposition \ref{prop-profend}
are given, respectively, in Subsections \ref{sub4.3} and \ref{sub4.4}.
Now, we prove Theorem \ref{thm-onelinf}.

\begin{proof}[Proof of Theorem \ref{thm-onelinf}]
Let $L:=\varliminf\limits_{j\to\infty}
F_{\lambda_j,1,\gamma}(u_j;\Omega).$
If $L=\infty$, then \eqref{duanxia} holds automatically. Assume that $L<\infty$.
After passing to a subsequence and relabeling if necessary, we may assume that $F_{\lambda_j,1,\gamma}(u_j;\Omega)\to L.$
By Lemma~\ref{lem-gp} when $\gamma\in(0,\infty)$ and  by Proposition~\ref{prop-coerc} when $\gamma\in(-\infty,-1)$, we conclude that $u\in BV(\Omega)$.
If $|Du|(\Omega)=0$, then \eqref{duanxia} follows from the
nonnegativity of $F_{\lambda_j,1,\gamma}$. Hence we assume that $|Du|(\Omega)>0$.

Let $A_1\subset\Omega$ be a Borel set with
$|Du|(\Omega\setminus A_1)=0$ such that Lemma~\ref{lem-tangent}
holds at every point of $A_1$. Fix $x_0\in A_1$ and let $k\in\mathbb N$. For any $z\in Q_0$, let
\begin{align*}
Q_k:=Q_{r_k}^{R_{x_0}}(x_0),\ \ A_k:=\frac{|Du|(Q_k)}{r_k^{N-1}},
\ \ b_k:=\fint_{Q_k}u\,dy,
\end{align*}
and
\begin{align*}
v_k(z):=\frac{u(x_0+r_kR_{x_0}z)-b_k}{A_k}.
\end{align*}
From Lemma~\ref{lem-tangent}(i), we infer that $Q_k\subset\Omega$ and
$|Du|(Q_k)\in(0,\infty)$.
For any $z\in Q_0$,  define
\begin{align*}
v_{j,k}(z):=\frac{u_j(x_0+r_kR_{x_0}z)-b_k}{A_k}\ \ \text{and}\ \
\Lambda_{j,k}:=\frac{\lambda_jr_k^{\gamma+1}}{A_k}.
\end{align*}
Then $\lim_{j\to\infty}v_{j,k}=v_k$ in $L^1(Q_0)$ and
$\lim_{j\to\infty}\Lambda_{j,k}=\lambda_\gamma$. Using Lemma~\ref{lem-scaling} with $p:=1$,
$a:=\frac{A_k}{r_k}$, $b:=b_k$, and $w:=v_{j,k}$, we find that
 \begin{align}\label{localscale}
F_{\lambda_j,1,\gamma}(u_j;Q_k)=\frac{A_k}{r_k}r_k^NF_{\frac{\lambda_jr_k^{\gamma+1}}{A_k},1,\gamma}
(v_{j,k};Q_0)=|Du|(Q_k)F_{\Lambda_{j,k},1,\gamma}(v_{j,k};Q_0).
\end{align}
From  \eqref{localscale} and the definition of $m_{1,\gamma}$, it follows  that
 \begin{align}\label{localone}
\varliminf_{j\to\infty}F_{\lambda_j,1,\gamma}(u_j;Q_k)
&= |Du|(Q_k)\varliminf_{j\to\infty}F_{\Lambda_{j,k},1,\gamma}(v_{j,k};Q_0)\notag\\
&\geq |Du|(Q_k)m_{1,\gamma}(v_k).
\end{align}
Furthermore, by Lemma~\ref{lem-tangent}, we find that, as $k\to\infty,$
$v_k\to W_{q_{x_0}}$ in $L^1(Q_0)$ and $|Dq_{x_0}|(I_0)=1$.
This, together with Lemma~\ref{lem-mplsc} and Proposition \ref{prop-profend}, implies that
\begin{align}
\varliminf_{k\to\infty}m_{1,\gamma}(v_k)\geq m_{1,\gamma}(W_{q_{x_0}})
\geq C_{N,1,\gamma}^{\mathrm{cell}}|Dq_{x_0}|(I_0)=C_{N,1,\gamma}^{\mathrm{cell}}.\label{mtangent}
\end{align}
Fix $\varepsilon\in(0,\frac{C_{N,1,\gamma}^{\mathrm{cell}}}{2})$.
Using
\eqref{mtangent}, we conclude  that there exists
$k_\varepsilon(x_0)\in\mathbb{N} $ such that, for any
$k\in\mathbb N\cap(k_\varepsilon(x_0),\infty)$,
$m_{1,\gamma}(v_k)
\geq C_{N,1,\gamma}^{\mathrm{cell}}-\varepsilon.$
From this
and \eqref{localone}, we deduce that, for any
$k\in\mathbb N\cap(k_\varepsilon(x_0),\infty)$,
 \begin{align}\label{onelocal}
\varliminf_{j\to\infty}F_{\lambda_j,1,\gamma}(u_j;Q_k)
\geq(C_{N,1,\gamma}^{\mathrm{cell}}-\varepsilon)|Du|(Q_k).
\end{align}
Now, define
\begin{align*}
V_{\varepsilon}:=\left\{Q_{r_k}^{R_x}(x):x\in A_1,\ k>k_\varepsilon(x)\right\}.
\end{align*}
Then, applying Lemma~\ref{lem:morse} to $(|Du|, \Omega, V_\varepsilon)$,
we further find that there exists a countable subfamily
$\{Q_i\}_{i\in\mathcal I}\subset V_\varepsilon$ of pairwise disjoint cubes such that
$
|Du|(A_1\setminus[\bigcup_{i\in\mathcal I}Q_i])=0.
$
This, combined with the fact that $|Du|(\Omega\setminus A_1)=0$, implies
$
|Du|(\Omega\setminus[\bigcup_{i\in\mathcal I}Q_i])=0.
$
Since $|Du|(\Omega)<\infty$, we can choose a finite subfamily of $\{Q_i\}_{i\in\mathcal I}$,
relabeled as $Q_1,\ldots,Q_M$, such that
$|Du|(\Omega\setminus[\bigcup_{i=1}^M Q_i])<\varepsilon.
$
Using this and \eqref{onelocal}, we conclude  that
 \begin{align*}
\varliminf_{j\to\infty}F_{\lambda_j,1,\gamma}(u_j;\Omega)
&\geq\varliminf_{j\to\infty}\sum_{i=1}^M F_{\lambda_j,1,\gamma}(u_j;Q_i)\\
&\geq(C_{N,1,\gamma}^{\mathrm{cell}}-\varepsilon)\sum_{i=1}^M|Du|(Q_i)
>(C_{N,1,\gamma}^{\mathrm{cell}}-\varepsilon)[|Du|(\Omega)-\varepsilon].
\end{align*}
Letting $\varepsilon\to 0$ yields \eqref{duanxia}.
This completes the
proof of Theorem \ref{thm-onelinf}.
\end{proof}

\subsection{Proof of Lemma \ref{lem-tangent} \label{sub4.3}}

In this subsection, we prove Lemma~\ref{lem-tangent}.
Let $\Omega\subset\mathbb R^N$ be open and $u\in BV(\Omega)$.
We use the standard notation for $BV$ functions from \cite{afp00}.
In particular, $\nabla u$ denotes the density of the absolutely
continuous part of $Du$ with respect to $\mathcal L^N$.
Let $J_u$ denote the approximate jump set of $u$.
For each $x\in J_u$, the jump values $u^+(x)$ and $u^-(x)$,
with $u^+(x)\neq u^-(x)$, and the unit normal
$\nu_u(x)\in\mathbb S^{N-1}$ satisfy
\begin{align*}
\lim_{r\to0^+}
\fint_{B^+(x,r,\nu_u(x))}
\left|u(y)-u^+(x)\right|\,dy=0
\end{align*}
and
\begin{align*}
\lim_{r\to0^+}
\fint_{B^-(x,r,\nu_u(x))}
\left|u(y)-u^-(x)\right|\,dy=0,
\end{align*}
where
$B^\pm(x,r,\nu):=\{y\in B(x,r):\ \pm(y-x)\cdot\nu>0\}$.
The triple $(u^+(x),u^-(x),\nu_u(x))$ is determined up to
interchanging $u^+(x)$ and $u^-(x)$ and reversing $\nu_u(x)$.
For $\nu\in\mathbb S^{N-1}$, write
$\nu^\perp:=\{z\in\mathbb R^N:z\cdot\nu=0\}$.
The following standard structure theorem for $BV$ functions
will be used below
(see \cite[Proposition~3.69, Theorems~3.77, 3.78, and~3.83,
Definition~3.91, and Proposition~3.92]{afp00}).
\begin{lemma}\label{lem:bvstruct}
Let $\Omega\subset\mathbb R^N$ be an open set and $u\in BV(\Omega)$.
\begin{enumerate}[{\rm(i)}]
\item For $\mathcal L^N$-almost every  $x\in\Omega$, the approximate gradient $\nabla u(x)$ exists  and
$\nabla u\in L^1(\Omega;\mathbb R^N)$.

\item The jump set $J_u$ is countably $(N-1)$-rectifiable.  Moreover, for
$\mathcal H^{N-1}$-almost every $x\in J_u$, the approximate tangent
hyperplane $T_xJ_u$ exists and $T_xJ_u=\nu_u(x)^\perp.$

\item There exist unique finite $\mathbb R^N$-valued Radon measures $D^ju$ and
$D^cu$, which are singular with respect to Lebesgue measure, such that
\begin{align*}
Du=\nabla u\,\mathcal L^N+D^ju+D^cu.
\end{align*}
Moreover, the jump part has the representation
\begin{align*}
D^ju=\left(u^+-u^-\right)\nu_u
\left(\mathcal H^{N-1}\mathopen{\scalebox{1.6}{$\llcorner$}}J_u\right).
\end{align*}

\item The positive measures $|\nabla u|\mathcal L^N$, $|D^ju|$, and
$|D^cu|$ are mutually singular, and hence
\begin{align}\label{eq:bvvar}
|Du|=|\nabla u|\,\mathcal L^N+|D^ju|+|D^cu|.
\end{align}
\end{enumerate}
\end{lemma}

Denote by the \emph{notation $C_{\rm c}(\mathbb R^N)$} the set of all continuous
functions on $\mathbb R^N$ with compact support.
Then, by Lemma~\ref{lem:bvstruct} and \cite[Theorem~2.83]{afp00},
we immediately obtain the following result.
\begin{lemma}\label{lem:recttan}
Let $\Omega\subset\mathbb R^N$ be open and $u\in BV(\Omega)$.
Then, for $|D^ju|$-almost every $x\in J_u$,
\begin{align}\label{eq:jball}
\lim_{r\to0^+}\frac{|D^ju|(B(x,r))}{r^{N-1}}
=\left|u^+(x)-u^-(x)\right|\omega_{N-1},
\end{align}
and, for any $\varphi\in C_{\rm c}(\mathbb R^N)$,
\begin{align}\label{eq:jtan}
\lim_{r\to0^+}\frac{1}{r^{N-1}}
\int_\Omega
\varphi\left(\frac{y-x}{r}\right)\,d|D^ju|(y)
=\left|u^+(x)-u^-(x)\right|
\int_{\nu_u(x)^\perp}\varphi(z)\,d\mathcal H^{N-1}(z).
\end{align}
Here, $\omega_{N-1}$ is the volume of the unit ball in
$\mathbb R^{N-1}$, with the convention $\omega_0=1$.
\end{lemma}

The following lemma is a special case of  the Besicovitch
differentiation theorem (see, for  instance, \cite[Theorem~2.22]{afp00}).
\begin{lemma}\label{lem:singdiff}
Let $\mu$ and $\lambda$ be finite positive Radon measures on an open set
$U\subset\mathbb R^N$.  If $\lambda\perp\mu$, then, for $\mu$-almost every  $x\in U$,
\begin{align*}
\lim_{r\to0^+}\frac{\lambda(B(x,r))}{\mu(B(x,r))}=0.
\end{align*}
\end{lemma}

For any $t\in I_0, $ define $q^{\rm ac}(t):=t$. Moreover, for any $t\in I_0$, let
\begin{align*}
q^{\rm j}(t)&:=
\begin{cases}
\displaystyle-\frac{1}{2},& \displaystyle t\in\left(-\frac{1}{2},0\right),\\
0,&t=0,\\
\displaystyle\frac{1}{2},&\displaystyle t\in\left(0,\frac{1}{2}\right).
\end{cases}
\end{align*}
Since $q^{\mathrm{ac}}$ and $q^{\mathrm{j}}$ are nondecreasing, it follows that
\begin{align}\label{profmass}
|Dq^{\rm ac}|(I_0)=|Dq^{\rm j}|(I_0)=1.
\end{align}
Recall that, for any given $q\in BV(I_0)$ and any $z\in Q_0$, $W_q(z):=q(z_N)$.
By a direct computation, we obtain
\begin{align*}
DW_q=e_N\left[\left(\mathcal L^{N-1}\mathopen{\scalebox{1.6}{$\llcorner$}}I_0^{N-1}\right)\otimes Dq\right],
\end{align*}
where $(\mathcal L^{N-1}\mathopen{\scalebox{1.6}{$\llcorner$}}I_0^{N-1})\otimes Dq$ denotes the product measure
of $\mathcal L^{N-1}\mathopen{\scalebox{1.6}{$\llcorner$}}I_0^{N-1}$ and $Dq$,
and hence
\begin{align*}
|DW_q|(Q_0)=|Dq|(I_0)
\end{align*}
This, together with \eqref{profmass}, implies that
\begin{align*}
|DW_{q^{\mathrm{ac}}}|(Q_0)=|DW_{q^{\mathrm{j}}}|(Q_0)=1.
\end{align*}
We first consider the absolutely  continuous part, and  we have the  following conclusion.
\begin{lemma}\label{lem:acpar}
Let $\Omega\subset\mathbb{R}^N$ be an open set and
$u\in BV(\Omega)$.  Then, for $|\nabla u|\mathcal L^N$-almost every  $x\in\Omega$, there exists
$R_x\in\mathcal O(N)$ such that, for any sufficiently small $r\in(0,\infty),$
$Q_r^{R_x}(x)\subset\Omega$, $|Du|(Q_r^{R_x}(x))\in (0,\infty)$, and
\begin{align*}
\lim_{r\to0}u_{x,r}^{R_x}=W_{q^{\rm ac}}\quad\text{in }L^1(Q_0),
\end{align*}
where $u^{R_x}_{x,r}$ is defined as in \eqref{rescale} with $R$ replaced by $R_x.$
\end{lemma}

\begin{proof}
Let $D^su:=D^ju+D^cu$. By Lemma~\ref{lem:bvstruct}(iii),
we find that $|D^su|\perp\mathcal L^N$.
Using this  and Lemma \ref{lem:singdiff}, we conclude  that, for $\mathcal L^N$-almost every
$x\in\Omega$,
\begin{align}\label{eq:acsing}
\lim_{\rho\to0^+}\frac{|D^su|(B(x,\rho))}{\rho^N}=0.
\end{align}
From the Lebesgue differentiation theorem and the fact that $\nabla u\in L^1(\Omega;\mathbb R^N)$,
we infer that, for $\mathcal L^N$-almost every $x\in\Omega,$ $x$ is a Lebesgue point of $\nabla u$. Moreover,
\begin{align*}
\left(|\nabla u|\mathcal L^N\right)\left(\left\{x\in\Omega:\nabla u(x)=0\right\}\right)
=\int_{\left\{y\in\Omega:\ \nabla u(y)=0\right\}}|\nabla u(y)|\,dy=0.
\end{align*}
Thus, for $|\nabla u|\mathcal L^N$-almost every $x\in\Omega$, the vector
$\nabla u(x)$ is nonzero, $x$ is a Lebesgue point of $\nabla u$, and \eqref{eq:acsing} holds.
Fix such a point $x$ and choose $R_x\in\mathcal O(N)$ such that
$R_xe_N=\frac{\nabla u(x)}{|\nabla u(x)|}.$ Then
\begin{align}\label{eq:acfine}
\lim_{r\to0}\fint_{Q_r^{R_x}(x)}|\nabla u(y)-\nabla u(x)|\,dy=0\ \ \text{and}\ \
\lim_{r\to0}\frac{|D^su|(Q_r^{R_x}(x))}{r^N}=0.
\end{align}
Let $$A_r:=\frac{|Du|(Q_r^{R_x}(x))}{r^{N-1}}.$$
By Lemma~\ref{lem:bvstruct}(iv) and \eqref{eq:acfine}, we conclude that
\begin{align}\label{eq:acmass}
\lim_{r\to 0}\frac{A_r}{r}=\lim_{r\to 0}\fint_{Q_r^{R_x}(x)}|\nabla u(y)|\,dy
+\frac{|D^su|(Q_r^{R_x}(x))}{r^N}=|\nabla u(x)|.
\end{align}
Thus, for any sufficiently small $r\in(0,\infty)$, $A_r\in(0,\infty).$
For any $y\in\Omega$, define
$$g(y):=u(y)-\nabla u(x)\cdot(y-x).$$  Then
\begin{align*}
\fint_{Q_r^{R_x}(x)}g(y)\,dy=\fint_{Q_r^{R_x}(x)}u(y)\,dy=:b_r.
\end{align*}
Using this and the  Poincar\'e inequality, we find that, for any sufficiently small $r\in(0,\infty)$,
\begin{align*}
\frac{1}{A_r}\fint_{Q_r^{R_x}(x)}|g(y)-b_r|\,dy&\lesssim
\frac{1}{A_rr^{N-1}}|Dg|(Q_r^{R_x}(x))\\
&\sim\frac{r}{A_r}\left[\fint_{Q_r^{R_x}(x)}|\nabla u(y)-\nabla u(x)|\,dy
+\frac{|D^su|(Q_r^{R_x}(x))}{r^N}\right].
\end{align*}
This, together with \eqref{eq:acmass} and \eqref{eq:acfine}, implies that
\begin{align}\label{eq:acerror}
\lim_{r\to 0} \frac{1}{A_r}\fint_{Q_r^{R_x}(x)}|g(y)-b_r|\,dy=0.
\end{align}
From \eqref{rescale} and the definitions of $u^{R_x}_{x,r}$, $g$, and $R_x$,
it follows that, for any sufficiently small $r\in(0,\infty)$ and  any $z\in Q_0,$
\begin{align*}
u_{x,r}^{R_x}(z)&=\frac{r^{N-1}}{|Du|(Q_r^{R_x}(x))}
\left[u(x+rR_xz)-b_r\right]=\frac{u(x+rR_xz)-b_r}{A_r}\\
&=\frac{g(x+rR_xz)-b_r}{A_r}+\frac{\nabla u(x)\cdot(rR_xz)}{A_r}\\
&=\frac{g(x+rR_xz)-b_r}{A_r}+\frac{|\nabla u(x)|rz_N}{A_r},
\end{align*}
which further implies that
\begin{align*}
\left\|u_{x,r}^{R_x}-W_{q^{\rm ac}}\right\|_{L^1(Q_0)}&\lesssim
\frac{1}{A_r}\fint_{Q_r^{R_x}(x)}|g-b_r|\,dy+
\left|\frac{|\nabla u(x)|r}{A_r}-1\right|\left\|W_{q^{\rm ac}}\right\|_{L^1(Q_0)}.
\end{align*}
Using this, \eqref{eq:acmass}, and \eqref{eq:acerror}, we conclude that $\lim_{r\to 0} u_{x,r}^{R_x}=W_{q^{\rm ac}}$ in
$L^1(Q_0).$
This completes the  proof of Lemma \ref{lem:acpar}.
\end{proof}
Next, we   consider the jump part. We have the  following conclusion.
\begin{lemma}\label{lem:acjump}
Let $\Omega\subset\mathbb{R}^N$ be an open set and $u\in BV(\Omega)$.
For $|D^ju|$-almost every $x\in J_u$, there exists
$R_x\in\mathcal O(N)$ such that, for any
sufficiently small $r\in(0,\infty)$,
$Q_r^{R_x}(x)\subset\Omega$, $|Du|(Q_r^{R_x}(x))\in (0,\infty)$, and
\begin{align}\label{eq:jumptan}
\lim_{r\to0}u_{x,r}^{R_x}=W_{q^{\rm j}}\quad\text{in }L^1(Q_0),
\end{align}
where $u^{R_x}_{x,r}$ is defined as in \eqref{rescale} with $R$ replaced by $R_x.$
\end{lemma}
\begin{proof}
Let  $\eta:=|\nabla u|\mathcal L^N+|D^cu|$. By Lemma~\ref{lem:bvstruct}(iv), we conclude that $|Du|=|D^ju|+\eta$
and $|D^ju|\perp\eta.$ Extend both $|D^ju|$ and $\eta$ by zero from $\Omega$ to $\mathbb R^N$, and denote the
extensions still by $|D^ju|$ and $\eta$. Since $\eta\perp|D^ju|$,
it follows that, for $|D^ju|$-almost every $x\in\Omega$,
\begin{align}\label{eq:jrel}
\lim_{\rho\to0^+}\frac{\eta(B(x,\rho))}{|D^ju|(B(x,\rho))}=0.
\end{align}
Fix a point $x\in J_u$ such that \eqref{eq:jrel},
\eqref{eq:jball}, and \eqref{eq:jtan} hold.
Interchanging $u^+(x)$ and $u^-(x)$ and reversing $\nu_u(x)$ if necessary, we may assume that
$s(x):=u^+(x)-u^-(x)>0.$ Choose $R_x\in\mathcal O(N)$ such that $R_xe_N=\nu_u(x).$ Let
\begin{align*}
A_r:=\frac{|Du|(Q_r^{R_x}(x))}{r^{N-1}}\ \ \text{and}\ \ b_r:=\fint_{Q_r^{R_x}(x)}u(y)\,dy.
\end{align*}
By the definition of $ u_{x,r}^{R_x}$, we find that, for any $z\in Q_0,$
\begin{align}\label{shenying}
u_{x,r}^{R_x}(z)=\frac{u(x+rR_xz)-b_r}{A_r}.
\end{align}
Now, we show that
\begin{align}\label{jumpmass}
\lim_{r\to0^+}A_r=s(x).
\end{align}
For any Borel set $E\subset\mathbb R^N$, define
\begin{align*}
\sigma_r(E):=\frac{1}{r^{N-1}}|D^ju|(x+rE)\ \ \text{and}\ \  \sigma
:=s(x)\left(\mathcal H^{N-1}\mathopen{\scalebox{1.6}{$\llcorner$}}\nu_u(x)^\perp\right).
\end{align*}
Let $A:=R_xQ_0.$ Then
\begin{align*}
\sigma_r(A)=\frac{|D^ju|(Q_r^{R_x}(x))}{r^{N-1}}
\end{align*}
and
\begin{align}\label{eq:chepiao}
\sigma(\partial A)&=s(x)\int_{\nu_u(x)^\perp}
\mathbf 1_{\partial A}(z)\,d\mathcal H^{N-1}(z)=
s(x)\int_{\nu_u(x)^\perp}\mathbf 1_{R_x\partial Q_0}(z)\,d\mathcal H^{N-1}(z)\notag\\
&=s(x)\int_{e_N^\perp}\mathbf 1_{\partial Q_0}(w)\,d\mathcal H^{N-1}(w)
=s(x)\mathcal H^{N-1}\left(e_N^\perp\cap\partial Q_0\right)=0.
\end{align}
On the other hand, from \eqref{eq:jtan}, we deduce that, for any $\varphi\in C_{\rm c}(\mathbb R^N)$,
\begin{align*}
\lim_{r\to0^+}\int_{\mathbb R^N}\varphi(z)\,d\sigma_r(z)=\int_{\mathbb R^N}\varphi(z)\,d\sigma(z).
\end{align*}
Equivalently, $\sigma_r$ converges weakly to $\sigma$ as $r\to 0^+.$ This, combined with
\eqref{eq:chepiao}, implies that
\begin{align}\label{eq:jpart}
\lim_{r\to0^+}\frac{|D^ju|(Q_r^{R_x}(x))}{r^{N-1}}=\lim_{r\to 0^+}\sigma_r(A)=\sigma (A)
=s(x)\mathcal H^{N-1}\left(e_N^\perp\cap Q_0\right)=s(x).
\end{align}
By \eqref{eq:jrel} and \eqref{eq:jball}, we conclude that
\begin{align*}
\lim_{\rho\to 0^+}\frac{\eta(B(x,\rho))}{\rho^{N-1}}=\lim_{\rho\to 0^+}
\frac{\eta(B(x,\rho))}{|D^ju|(B(x,\rho))}\frac{|D^ju|(B(x,\rho))}{\rho^{N-1}}=0.
\end{align*}
From this and \eqref{eq:jpart},  it follows  that
\begin{align*}
\lim_{r\to0^+}A_r=\lim_{r\to0^+}\left[\frac{|D^ju|(Q_r^{R_x}(x))}{r^{N-1}}
+\frac{\eta(Q_r^{R_x}(x))}{r^{N-1}}\right]=s(x).
\end{align*}
This proves \eqref{jumpmass}.

Let
\begin{align*}
 Q_0^\pm:=Q_0\cap\left\{z\in\mathbb R^N:\ \pm z_N>0\right\}.
\end{align*}
By a change of variables and the definitions of $u^+$ and $u^-,$ we find that
\begin{align*}
\int_{Q_0^+}|u(x+rR_xz)-u^+(x)|\,dz\lesssim\fint_{B^+(x,\sqrt{N}r,\nu_u(x))}|u(y)-u^+(x)|\,dy
\to0
\end{align*}
and
\begin{align*}
\int_{Q_0^-}|u(x+rR_xz)-u^-(x)|\,dz\lesssim\fint_{B^-(x,\sqrt{N}r,\nu_u(x))}|u(y)-u^-(x)|\,dy
\to0
\end{align*}
as $r\to 0^+.$ For any $z\in Q_0,$ let
\begin{align*}
w_x(z):=u^-(x)\mathbf 1_{\left\{z_N<0\right\}}+u^+(x)\mathbf 1_{\left\{z_N>0\right\}}.
\end{align*}
Then
\begin{align}\label{jumplone}
\lim_{r\to0}u(x+rR_x\,\cdot)=w_x\ \ \text{in}\ \ L^1(Q_0).
\end{align}
This further implies that
\begin{align*}
\lim_{r\to0^+}b_r=\lim_{r\to0^+}\int_{Q_0}u(x+rR_xz)\,dz=\int_{Q_0}w_x(z)\,dz
 =\frac{u^+(x)+u^-(x)}{2}:=m_x,
\end{align*}
which, together with \eqref{jumplone}, \eqref{jumpmass}, and \eqref{shenying},
implies that
\begin{align}\label{eq:jnorm}
\lim_{r\to 0^+}u_{x,r}^{R_x}=\frac{w_x-m_x}{s(x)}\ \ \text{in}\ \ L^1(Q_0).
\end{align}
Note that,  for any $z\in Q_0$ with $z_N>0$,
\begin{align*}
\frac{w_x(z)-m_x}{s(x)}&=\frac{u^+(x)-\dfrac{u^+(x)+u^-(x)}{2}}
{u^+(x)-u^-(x)}=\frac{1}{2}
\end{align*}
and, for any $z\in Q_0$ with $z_N<0,$
\begin{align*}
\frac{w_x(z)-m_x}{s(x)}&=\frac{u^-(x)-\dfrac{u^+(x)+u^-(x)}{2}}
{u^+(x)-u^-(x)}=-\frac{1}{2}.
\end{align*}
Thus, for $\mathcal L^N$-almost every  $z\in Q_0$,
$$\frac{w_x(z)-m_x}{s(x)}=W_{q^{\rm j}}(z).$$
From this and  \eqref{eq:jnorm}, we deduce that
\eqref{eq:jumptan} holds. This completes the  proof of Lemma \ref{lem:acjump}.
\end{proof}

Finally, we   consider the Cantor part. Since $D^cu$ is an $\mathbb R^N$-valued Radon measure,
it follows that there exists a Borel map $\theta_u^c:\ \Omega\to \mathbb S^{N-1}$ such that
$D^cu=\theta_u^c\,|D^cu|.$
For $|D^cu|$-almost every $x\in\Omega$, let $R_x\in \mathcal O(N)$ be such that $R_xe_N
=\theta_u^c(x).$ The following lemma is a special case of \cite[Lemma~4.5]{cfi24}
with $m:=1$, $\eta:=1$, and $\xi:=\theta_u^c$.

\begin{lemma}\label{lem:cfi}
Let $\Omega\subset\mathbb{R}^N$ be an open set,
$\sigma$  a positive Radon measure on
$\Omega$, and $u\in BV(\Omega)$. Then, for $|D^cu|$-almost every $x\in\Omega$,
there exist a sequence $\{r_k\}_{k\in\mathbb N}\subset(0,\infty)$ with
$r_k\to0$ and a nondecreasing function
$q_x^{\rm c}:I_0\to\mathbb R$ such that
\begin{itemize}
\item [$\mathrm{(i)}$] for any $k\in\mathbb{N}$,
$\sigma(\partial Q^{R_x}_{r_k}(x))=0$;
\item [$\mathrm{(ii)}$]
\begin{align*}
\lim_{k\to\infty}\frac{|Du|(Q_{r_k}^{R_x}(x))}{r_k^N}
=\infty\ \ \text{and}\ \ \lim_{k\to\infty}\frac{|Du|(Q_{r_k}^{R_x}(x))}{r_k^{N-1}}
=0;
\end{align*}
\item [$\mathrm{(iii)}$] $|Dq_x^{\rm c}|(I_0)=1$;
\item [$\mathrm{(iv)}$]
\begin{align*}
\lim_{k\to\infty}u_{x,r_k}^{R_x}(z)=\lim_{k\to\infty}\frac{ u(x+r_kR_xz)
-\fint_{Q_{r_k}^{R_x}(x)}u(y)\,dy}{|Du|\left(Q_{r_k}^{R_x}(x)\right)r_k^{-N+1}
}=q_x^{\rm c}(z_N)
\end{align*}
strictly in $BV(Q_0).$
\end{itemize}
\end{lemma}

Now, we prove Lemma \ref{lem-tangent}.
\begin{proof}[Proof of Lemma \ref{lem-tangent}]
By \eqref{eq:bvvar}, we find that the measures $|\nabla u|\mathcal L^N$, $|D^ju|,$ and $|D^cu|$
are mutually singular and $|Du|=|\nabla u|\mathcal L^N+|D^ju|+|D^cu|.$
Thus, to prove the  present lemma, it suffices to show the desired conclusions on a
full-measure set with respect to each of these three measures $|\nabla u|\mathcal{L}^N$,
$|D^ju|$, and $|D^cu|.$

We first consider the absolutely continuous part. Let $x$ be a point
such that  Lemma~\ref{lem:acpar} holds. Let all the notation be the same as in
Lemma~\ref{lem:acpar},  and  let $q_x:=q^{\rm ac}$.
Since $|Du|$ is finite, it follows that there exist at
most countably many radii $r\in(0,\infty)$ such that
$|Du|(\partial Q_r^{R_x}(x))>0.$
Thus, we may  choose a sequence $\{r_k\}_{k\in\mathbb N}\subset(0,\infty)$
such that $\lim_{k\to\infty}r_k=0$ and, for any $k\in\mathbb{N},$
$|Du|(\partial Q_{r_k}^{R_x}(x))=0.$
This proves the desired conclusions for the absolutely continuous part.
Similarly, the jump part follows from Lemma~\ref{lem:acjump}, and the Cantor part
follows from Lemma~\ref{lem:cfi} with $\sigma:=|Du|$. This completes
the proof of Lemma \ref{lem-tangent}.
\end{proof}

\subsection{Proof of Proposition \ref{prop-profend} \label{sub4.4}}
In this subsection, we  prove Proposition \ref{prop-profend}.
The following lemma follows from \cite[Proposition~2.3]{bsvy24}.
In what follows, for any $x,y\in{\mathbb R}^N$, the \emph{notation $[x,y]$}
denotes the closed segment connecting $x$ and $y$.
\begin{lemma}\label{lem:segment}
Let $f\in C_{\rm c}(\mathbb R^N)$, $\gamma\in(-\infty,-1)\cup(0,\infty)$, and
\begin{align*}
\mathcal E(f,\gamma):=\left\{(x,y)\in\mathbb R^N\times\mathbb R^N:\ x\neq y,\
\int_{[x,y]}|f(s)|\,ds>|x-y|^{\gamma+1}\right\}.
\end{align*}
Then there exists a positive constant
$C_{N,\gamma}$, depending only on $N$ and $\gamma$, such that
\begin{align*}
\iint_{\mathcal E(f,\gamma)}|x-y|^{\gamma-N}\,dx\,dy\leq C_{N,\gamma}\|f\|_{L^1(\mathbb R^N)}.
\end{align*}
\end{lemma}

For any  measurable set $E\subset\mathbb R^N$ and any $r\in (0,\infty)$, let
\begin{align*}
 E^{(r)}:=\left\{x\in\mathbb R^N:\operatorname{dist}(x,E)<r\right\}.
\end{align*}

\begin{lemma}\label{lem-localbv}
Let $\gamma\in(-\infty,-1)\cup(0,\infty)$, $v\in BV(\mathbb R^N)$, and $E\subset\mathbb R^N$ be a bounded Borel set.  Then
there exists a  positive constant $C_{N,\gamma}$, depending only on $N$ and $\gamma$, such that, for any $r,\lambda\in (0,\infty)$,
 \begin{align}\label{eq:localbv}
\lambda
\underset{\scriptscriptstyle |v(x)-v(y)|>\lambda|x-y|^{\gamma+1}}
{\int_E\int_{B(x,r)}}
|x-y|^{\gamma-N}\,dy\,dx
\leq C_{N,\gamma}|Dv|(E^{(3r)}).
\end{align}
\end{lemma}
\begin{proof}
Assume first  $v\in C^\infty(\mathbb R^N)$. Fix $\varepsilon\in(0,1)$.
Choose $\chi\in C_{\rm c}^\infty(E^{(2r)})$ such that $0\leq\chi\leq1$ and $\chi\equiv 1$ on $E^{(r)}$.
Let $x\in E$ and $y\in B(x,r)$. Then the whole segment $[x,y]$ is contained in $E^{(r)}$.
This implies that  $\chi=1$ on $[x,y]$ and hence
\begin{align*}
|v(x)-v(y)|\leq\int_{[x,y]}|\nabla v(s)|\,ds=\int_{[x,y]}\chi(s)|\nabla v(s)|\,ds.
\end{align*}
Consequently, if $|v(x)-v(y)|>\lambda|x-y|^{\gamma+1},$ then
\begin{align*}
\int_{[x,y]}\frac{\chi(s)|\nabla v(s)|}{(1-\varepsilon)\lambda}\,ds> |x-y|^{\gamma+1}.
\end{align*}
Using this and applying Lemma~\ref{lem:segment} with
$f:=\frac{\chi|\nabla v|}{(1-\varepsilon)\lambda},$ we conclude  that
 \begin{align*}
&\underset{\scriptscriptstyle |v(x)-v(y)|>\lambda|x-y|^{\gamma+1}}
{\int_E\int_{B(x,r)}}
|x-y|^{\gamma-N}\,dy\,dx\\
&\quad\leq\iint_{\mathcal E(f,\gamma)}|x-y|^{\gamma-N}\,dx\,dy
\lesssim\frac{1}{(1-\varepsilon)\lambda}
\int_{\mathbb R^N}\chi|\nabla v|\,dx
\leq\frac{1}{(1-\varepsilon)\lambda}\int_{E^{(2r)}}|\nabla v|\,dx.
\end{align*}
Letting $\varepsilon \to0$ in the above inequality, we then obtain
\begin{align}\label{lbvsmooth}
\lambda
\underset{\scriptscriptstyle |v(x)-v(y)|>\lambda|x-y|^{\gamma+1}}
{\int_E\int_{B(x,r)}}
|x-y|^{\gamma-N}\,dy\,dx
\lesssim\int_{E^{(2r)}}|\nabla v|\,dx.
\end{align}
Now, let $v\in BV(\mathbb R^N)$. Choose a standard mollifier $\rho\in C_{\rm c}^\infty(B({\bf0},1))$ such that
$\rho\geq0$ and $\int_{\mathbb R^N}\rho\,dx=1.$ For any $\delta\in (0,\infty)$ and $z\in\mathbb{R}^N,$ let
$\rho_\delta(z):=\frac{1}{\delta^N}\rho(\frac{z}{\delta}).$
Choose a sequence $\delta_k\downarrow0$ with $\delta_k<r$, and
let $v_k:=v*\rho_{\delta_k}.$ After passing to a subsequence, we may assume that, for almost every $x\in\mathbb{R}^N, $
$\lim_{k\to\infty}v_k(x)=v(x).$ Fix $\varepsilon\in(0,1)$. By \eqref{lbvsmooth} and  Fatou's lemma, we find that
\begin{align}\label{lbvfatou}
&\lambda
\underset{\scriptscriptstyle |v(x)-v(y)|>\lambda|x-y|^{\gamma+1}}
{\int_E\int_{B(x,r)}}
|x-y|^{\gamma-N}\,dy\,dx\notag\\
&\quad\leq\varliminf_{k\to\infty}\lambda
\underset{\scriptscriptstyle |v_k(x)-v_k(y)|>(1-\varepsilon)\lambda|x-y|^{\gamma+1}}
{\int_E\int_{B(x,r)}}|x-y|^{\gamma-N}\,dy\,dx
\lesssim\frac{1}{1-\varepsilon}\varliminf_{k\to\infty}\int_{E^{(2r)}}|\nabla v_k|\,dx.
\end{align}
On the other hand, from the fact that $ \nabla v_k=\rho_{\delta_k}*Dv$ and Fubini's theorem,
we deduce that, for any $k\in\mathbb{N},$
\begin{align*}
\int_{E^{(2r)}}|\nabla v_k(x)|\,dx&\leq\int_{\mathbb R^N}\left[
\int_{E^{(2r)}}\rho_{\delta_k}(x-y)\,dx\right]d|Dv|(y)\\
&=\int_{\mathbb R^N}\rho_{\delta_k}(z)\left[\int_{\mathbb R^N}
\mathbf 1_{E^{(2r)}}(y+z)\,d|Dv|(y)\right]dz\\
&\leq\int_{\mathbb R^N}\rho_{\delta_k}(z)|Dv|(E^{(3r)})dz=|Dv|(E^{(3r)}).
\end{align*}
Using  this and  \eqref{lbvfatou}, we conclude that \eqref{eq:localbv} holds. This completes the proof of  Lemma \ref{lem-localbv}.
\end{proof}

\begin{lemma}\label{lem-collar}
Let $\gamma\in(-\infty,-1)\cup(0,\infty)$,  and let $q\in BV(I_0)$ be nondecreasing. Assume that  $q(-\frac{1}{2}+)=0$ and $q(\frac{1}{2}-)=1$, where $q(-\frac{1}{2}+):=\lim_{t\to -\frac{1}{2}^+}q(t)$ and $q(\frac{1}{2}-):=\lim_{t\to\frac{1}{2}^-}q(t).$
For any $\rho\in(0,\infty)$, let $B_\rho:=\left\{z\in Q_0:\operatorname{dist}(z,\partial Q_0)<\rho\right\}$.
Then
 \begin{align}\label{lingbian}
\lim_{\rho\to 0^+}\varlimsup_{\lambda\to\lambda_\gamma}
\lambda
\underset{\scriptscriptstyle |W_q(x)-W_q(y)|>\lambda|x-y|^{\gamma+1}}
{\int_{B_\rho}\int_{Q_0}}
|x-y|^{\gamma-N}\,dy\,dx=0.
\end{align}
\end{lemma}

\begin{proof}
Since  $q$ is  nondecreasing, $q(-\frac{1}{2}+)=0$, and $q(\frac{1}{2}-)=1$,
we may assume, after changing $q$ on a countable set if necessary, that
$0\leq q\leq 1$ on $I_0.$ Extend $q$ constantly outside $I_0$ by setting, for any $t\in\mathbb R$,
\begin{align*}
\widetilde q(t):=\begin{cases}
0,&\displaystyle t\in\left(-\infty,-\frac{1}{2}\right],\\
q(t),&\displaystyle t\in\left(-\frac{1}{2},\frac{1}{2}\right),\\
1,&\displaystyle t\in\left[\frac{1}{2},\infty\right).
\end{cases}
\end{align*}
Then $\widetilde{q}$ is continuous at $-\frac{1}{2}$ and $\frac{1}{2}$, which implies that
\begin{align}\label{eq:colend}
D\widetilde q\left(\left\{-\frac{1}{2}\right\}\right)
=D\widetilde q\left(\left\{\frac{1}{2}\right\}\right)=0.
\end{align}
Choose $\theta\in C_{\rm c}^\infty(\mathbb R)$ such that $\theta\equiv 1$   on a
neighborhood of $[-\frac{1}{2},\frac{1}{2}].$ If $N=1$, let $\zeta\equiv1$ and, if $N\ge2$,
choose $\zeta\in C_{\rm c}^\infty(\mathbb R^{N-1})$ such that $\zeta\equiv 1$   on a neighborhood
of $[-\frac{1}{2},\frac{1}{2}]^{N-1}.$ For any $z:=(z',z_N)\in\mathbb{R}^N,$ define
$\widetilde W_q(z',z_N):=\zeta(z')\theta(z_N)\widetilde q(z_N).$
Then $\widetilde W_q\in BV(\mathbb R^N)$.
Moreover, on a neighborhood $V$ of $\overline{Q}_0$, we find that, for any $(z',z_N)\in V$,
$\widetilde W_q(z',z_N)=\widetilde q(z_N).$
This further implies that, for any $i\in\{1,\ldots, N-1\}$ and
$\varphi\in C_{\rm c}^\infty(V)$,
\begin{align*}
\left\langle D_i\widetilde W_q,\varphi\right\rangle&=-\int_V\widetilde W_q(z)
\frac{\partial\varphi}{\partial z_i}(z)\,dz
=-\int_V \widetilde{q}(z_N)\frac{\partial\varphi}{\partial z_i}(z)\,dz=0
\end{align*}
and
\begin{align*}
\left\langle D_N\widetilde W_q,\varphi\right\rangle=-\int_V\widetilde W_q(z',z_N)
\frac{\partial\varphi}{\partial z_N}(z',z_N)\,dz'\,dz_N=
\int_{\mathbb R}\int_{\mathbb R^{N-1}}\varphi(z',t)\,dz'dD\widetilde q(t).
\end{align*}
Using this, we conclude that, for any $i\in\{1,\ldots, N-1\}$,
\begin{align}\label{eq:dou2}
(D_i\widetilde W_q)\mathopen{\scalebox{1.6}{$\llcorner$}}V=0\ \ \text{and}\ \ (D_N\widetilde W_q)\mathopen{\scalebox{1.6}{$\llcorner$}}V
=(\mathcal L^{N-1}\otimes D\widetilde q)\mathopen{\scalebox{1.6}{$\llcorner$}}V,
\end{align}
where $\mathcal L^{N-1}\otimes D\widetilde q$ denotes the product measure of $\mathcal L^{N-1}$ and $D\widetilde q.$
Consequently, for any Borel sets $A'\subset\mathbb R^{N-1}$ and $B\subset\mathbb R$ satisfying
$A'\times B\subset V$, we have
$|D\widetilde W_q|(A'\times B)=\mathcal L^{N-1}(A')\,|D\widetilde q|(B).$
From this and \eqref{eq:colend},  it follows that
\begin{align}\label{bianling}
|D\widetilde W_q|(\partial Q_0)=0.
\end{align}
Fix $\rho\in(0,\frac1{20})$ and write\begin{align*}
F_{\lambda,1,\gamma}(W_q;B_\rho,Q_0)
=:I_{\lambda,\rho}^{\rm near}+I_{\lambda,\rho}^{\rm far},
\end{align*}
where the domains of integrations in $I_{\lambda,\rho}^{\rm near}$ and
$I_{\lambda,\rho}^{\rm far}$ are restricted to $|x-y|<\rho$ and
$|x-y|\geq\rho$, respectively.
Using $\widetilde W_q=W_q$ on $Q_0$ and applying Lemma~\ref{lem-localbv} with
$E:=B_\rho$ and $r:=\rho$, we conclude that
 \begin{align*}
I_{\lambda,\rho}^{\rm near}
&\leq\lambda
\underset{\scriptscriptstyle |\widetilde W_q(x)-\widetilde W_q(y)|>\lambda|x-y|^{\gamma+1}}
{\int_{B_\rho}\int_{B(x,\rho)}}
|x-y|^{\gamma-N}\,dy\,dx\lesssim\left|D\widetilde W_q\right|\left(B_\rho^{(3\rho)}\right).
\end{align*}
If $\gamma\in(0,\infty)$, then $I_{\lambda,\rho}^{\rm far}=0$ whenever
$\lambda\rho^{\gamma+1}\geq1$.
If $\gamma\in(-\infty,-1)$, then
 \begin{align*}
I_{\lambda,\rho}^{\rm far}
&\leq\lambda\int_{B_\rho}\int_{\mathbb R^N\setminus B(x,\rho)}
|x-y|^{\gamma-N}\,dy\,dx\lesssim\lambda\rho^\gamma|B_\rho|\to0
\end{align*}
as $\lambda\to0^+$. Consequently, for both ranges of $\gamma$,
 \begin{align*}
\varlimsup_{\lambda\to\lambda_\gamma}
F_{\lambda,1,\gamma}(W_q;B_\rho,Q_0)
\lesssim\left|D\widetilde W_q\right|\left(B_\rho^{(3\rho)}\right).
\end{align*}
Letting $\rho\to0^+$ in the above inequality, and using \eqref{bianling},
we find that \eqref{lingbian} holds. This completes the proof of Lemma~\ref{lem-collar}.
\end{proof}

The following lemma gives the boundary approximation needed in the proof
of Proposition \ref{prop-profend}.

\begin{lemma}\label{lem-nearmin}
Let $\gamma\in(-\infty,-1)\cup(0,\infty)$, and  let $q\in BV(I_0)$ be nondecreasing with
$q(-\frac12+)=0$ and $q(\frac12-)=1$. Then, for any $\varepsilon\in(0,\infty)$,
there exist $\rho\in(0,\frac1{20})$ and sequences
$\{\widehat\lambda_j\}_{j\in\mathbb N}\subset(0,\infty)$ and
$\{\widehat v_j\}_{j\in\mathbb N}\subset L^\infty(Q_0)$ such that
the following statements hold.
\begin{itemize}
\item[$\mathrm{(i)}$] For any $j\in\mathbb N$, $0\leq\widehat v_j\leq1$ and
\begin{align}\label{jinbian}
\widehat v_j=W_q\quad\text{on }\left\{z\in Q_0:
\operatorname{dist}(z,\partial Q_0)\leq\rho\right\}.
\end{align}
Moreover, $\lim_{j\to\infty}\widehat v_j=W_q$ in $L^1(Q_0)$.

\item[$\mathrm{(ii)}$] $\lim_{j\to\infty}\widehat\lambda_j=\lambda_\gamma$.

\item[$\mathrm{(iii)}$] It holds that
\begin{align}\label{jineng}
\varlimsup_{j\to\infty}F_{\widehat\lambda_j,1,\gamma}\left(\widehat v_j;Q_0\right)
\leq m_{1,\gamma}(W_q)+\varepsilon.
\end{align}
\end{itemize}
\end{lemma}

\begin{proof}
If $m_{1,\gamma}(W_q)=\infty$, take $\rho:=\frac1{40}$,
$\widehat\lambda_j:=j$ if $\gamma\in(0,\infty)$ and $\widehat\lambda_j:=j^{-1}$ if $\gamma\in(-\infty,-1)$, and $\widehat v_j:=W_q$, and hence the desired conclusions hold.
We therefore assume that $m_{1,\gamma}(W_q)<\infty$. Define $P:\mathbb R\to[0,1]$ by
setting, for any $s\in\mathbb{R}$,
\begin{align*}
P(s):=\begin{cases}
0, & s\in(-\infty,0],\\
s, & s\in(0,1),\\
1, & s\in[1,\infty).
\end{cases}
\end{align*}
Then $P$ is $1$-Lipschitz and $P\circ W_q=W_q$ almost everywhere in $Q_0$.
By the definition of $m_{1,\gamma}(W_q)$, we find that there exist sequences
$\{\lambda_j\}_{j\in\mathbb N}\subset(0,\infty)$ and
$\{u_j\}_{j\in\mathbb N}\subset L^1(Q_0)$ such that
$\lambda_j\to\lambda_\gamma$, $u_j\to W_q$ in $L^1(Q_0)$, and
\begin{align}\label{sishi}
\varliminf_{j\to\infty}F_{\lambda_j,1,\gamma}\left(u_j;Q_0\right)
\leq m_{1,\gamma}(W_q)+\frac\varepsilon2.
\end{align}
Let $j\in\mathbb{N}$ and $v_j:=P\circ u_j$.
Then $v_j\in L^\infty(Q_0)$, $0\leq v_j\leq1$ almost everywhere in $Q_0$, and
\begin{align*}
\left\|v_j-W_q\right\|_{L^1(Q_0)}=\left\|P\circ u_j-P\circ W_q\right\|_{L^1(Q_0)}
\leq\left\|u_j-W_q\right\|_{L^1(Q_0)}\to0
\end{align*}
as $j\to\infty$. Moreover, from \eqref{sishi}, we deduce that
\begin{align*}
\varliminf_{j\to\infty}F_{\lambda_j,1,\gamma}\left(v_j;Q_0\right)
\leq\varliminf_{j\to\infty}F_{\lambda_j,1,\gamma}\left(u_j;Q_0\right)
\leq m_{1,\gamma}(W_q)+\frac\varepsilon2.
\end{align*}
Using this and passing  to a subsequence,
we may assume that
$\lambda_j\to\lambda_\gamma$, $v_j\to W_q$ in $L^1(Q_0)$,
$0\leq v_j\leq1$ almost everywhere in $Q_0$ for any $j\in\mathbb N$, and
\begin{align}\label{hehu}
\lim_{j\to\infty}F_{\lambda_j,1,\gamma}\left(v_j;Q_0\right)
\leq m_{1,\gamma}(W_q)+\frac\varepsilon2.
\end{align}
By Lemma~\ref{lem-collar}, we conclude that there exists
$\rho\in(0,\frac1{20})$ such that
\begin{align}\label{jinceng}
\varlimsup_{\lambda\to\lambda_\gamma}
F_{\lambda,1,\gamma}\left(W_q;B_{2\rho},Q_0\right)
=\varlimsup_{\lambda\to\lambda_\gamma}
\lambda
\underset{\scriptscriptstyle |W_q(x)-W_q(y)|>\lambda|x-y|^{\gamma+1}}
{\int_{B_{2\rho}}\int_{Q_0}}
|x-y|^{\gamma-N}\,dy\,dx
\leq\frac\varepsilon2.
\end{align}
Choose a Lipschitz function $\chi_\rho:Q_0\to[0,1]$ such that
$0\leq\chi_\rho\leq1$, $\chi_\rho(z)=0$ when
$\operatorname{dist}(z,\partial Q_0)\leq\rho$, and $\chi_\rho(z)=1$ when
$\operatorname{dist}(z,\partial Q_0)\geq2\rho$.
For any $j\in\mathbb N$, define
$\widehat v_j:=\chi_\rho v_j+(1-\chi_\rho)W_q.$
Then $0\leq\widehat v_j\leq1$ and
\begin{align*}
\left\|\widehat v_j-W_q\right\|_{L^1(Q_0)}\leq\left\|v_j-W_q\right\|_{L^1(Q_0)}\to0
\end{align*}
as $j\to\infty.$ Choose a sequence $\{\eta_j\}_{j\in\mathbb{N}}\subset (0,\infty)$
such that
\begin{align}\label{qingjie}
\lim_{j\to\infty}\eta_j=0\ \ \text{and}\ \
\lim_{j\to\infty}\frac{\|v_j-W_q\|_{L^1(Q_0)}}{\eta_j}=0.
\end{align}
Define $\widehat\lambda_j:=(1+\eta_j)\lambda_j.$ Applying Lemma~\ref{lem-cutoff}(ii), we obtain
\begin{align*}
F_{\widehat\lambda_j,1,\gamma}\left(\widehat v_j;Q_0\right)
&\leq(1+\eta_j)F_{\lambda_j,1,\gamma}(v_j;Q_0)+(1+\eta_j)
F_{\lambda_j,1,\gamma}\left(W_q;B_{2\rho},Q_0\right)\notag\\
&\quad+\frac{\mathcal H^{N-1}(\mathbb S^{N-1})L_{\chi_\rho}(1+\eta_j)}{|\gamma|\eta_j}\|v_j-W_q\|_{L^1(Q_0)},
\end{align*}
where $L_{\chi_\rho}<\infty$ is the Lipschitz constant of $\chi_\rho$.
Letting $j\to\infty$ in the above inequality and using \eqref{jinceng}, \eqref{qingjie},
and \eqref{hehu}, we conclude that
\begin{align*}
\varlimsup_{j\to\infty}F_{\widehat\lambda_j,1,\gamma}\left(\widehat v_j;Q_0\right)
\leq
\lim_{j\to\infty}F_{\lambda_j,1,\gamma}\left(v_j;Q_0\right)+\frac{\varepsilon}{2}
\leq m_{1,\gamma}(W_q)+\varepsilon.
\end{align*}
This shows \eqref{jineng},
which completes the proof of Lemma \ref{lem-nearmin}.
\end{proof}

\begin{lemma}\label{lem-skeleton}
Let $\gamma\in(-\infty,-1)\cup(0,\infty)$, $L\in\mathbb N$, and $\mathscr Q:=\{Q_1,\ldots,Q_L\}$ be a
finite family of pairwise disjoint open cubes contained in $Q_0$,
and define
\begin{align*}
S_{\mathscr Q}:=\overline{Q_0\setminus\left(\bigcup_{i=1}^LQ_i\right)}.
\end{align*}
Let $ U\in BV(Q_0)\cap L^\infty(Q_0)$ and assume that $U$ admits an extension
$\widetilde U\in BV(\mathbb R^N)\cap L^\infty(\mathbb R^N)$
such that $\widetilde U=U$ almost everywhere in $Q_0$ and
\begin{align}\label{geling}
|D\widetilde U|(S_{\mathscr Q})=0.
\end{align}
For $\lambda\in(0,\infty)$, define
\begin{align*}
 F_{\lambda,1,\gamma}^{\rm cross}(U;\mathscr Q):=\lambda\sum_{\genfrac{}{}{0pt}{}{1\leq i,j\leq L}{i\neq j}}
\underset{\scriptscriptstyle |U(x)-U(y)|>\lambda|x-y|^{\gamma+1}}
{\int_{Q_i}\int_{Q_j}}
|x-y|^{\gamma-N}\,dy\,dx.
\end{align*}
Then
\begin{align}\label{gexiao}
\lim_{\lambda\to\lambda_\gamma}F_{\lambda,1,\gamma}^{\rm cross}(U;\mathscr Q)=0.
\end{align}
\end{lemma}

\begin{proof}
Without loss of generality, we may assume that, for any $x\in Q_0$,
$|U(x)|\leq\|U\|_{L^\infty(Q_0)}$. If $\|U\|_{L^\infty(Q_0)}=0$,
then \eqref{gexiao} holds automatically. Assume now that
$\|U\|_{L^\infty(Q_0)}>0$. Fix $R\in(0,\infty)$ and write
 \begin{align*}
F_{\lambda,1,\gamma}^{\rm cross}(U;\mathscr Q)
=:I_{\lambda,R}^{\rm near}+I_{\lambda,R}^{\rm far},
\end{align*}
where the domains of integrations in $I_{\lambda,R}^{\rm near}$ and
$I_{\lambda,R}^{\rm far}$ are restricted to $|x-y|<R$ and
$|x-y|\geq R$, respectively.
Let $i,j\in\{1,\ldots,L\}$ with $i\neq j$, and fix points
$x\in Q_i$ and $y\in Q_j$ such that $|x-y|<R$.
Since $Q_i$ and $Q_j$ are disjoint, it follows that there exists
a point $z\in[x,y]\cap S_{\mathscr Q}$, which implies that
 \begin{align*}
\operatorname{dist}(x,S_{\mathscr Q})\leq|x-z|\leq|x-y|<R.
\end{align*}
Consequently, $x\in S_{\mathscr Q}^{(R)}$ and $y\in B(x,R)$.
Using this and applying Lemma~\ref{lem-localbv} with
$E:=S_{\mathscr Q}^{(R)}$ and $r:=R$, we conclude that
\begin{align*}
I_{\lambda,R}^{\rm near}&\leq\lambda
\underset{\scriptscriptstyle |\widetilde U(x)-\widetilde U(y)|>\lambda|x-y|^{\gamma+1}}
{\int_{S_{\mathscr Q}^{(R)}}\int_{B(x,R)}}|x-y|^{\gamma-N}\,dy\,dx
\lesssim\left|D\widetilde U\right|\left(S_{\mathscr Q}^{(4R)}\right).
\end{align*}
If $\gamma\in(0,\infty)$, then $I_{\lambda,R}^{\rm far}=0$ whenever
$\lambda R^{\gamma+1}\geq2\|U\|_{L^\infty(Q_0)}$.
If $\gamma\in(-\infty,-1)$, then
\begin{align*}
I_{\lambda,R}^{\rm far}
\leq\lambda|Q_0|\int_{|z|\geq R}|z|^{\gamma-N}\,dz
\lesssim\lambda|Q_0|R^\gamma\to0
\end{align*}
as $\lambda\to0^+$. Thus,
 \begin{align*}
\varlimsup_{\lambda\to\lambda_\gamma}F_{\lambda,1,\gamma}^{\rm cross}
(U;\mathscr Q)
\lesssim\left|D\widetilde U\right|\left(S_{\mathscr Q}^{(4R)}\right).
\end{align*}
Letting $R\to0^+$ in the above inequality and using \eqref{geling},
we find that \eqref{gexiao} holds. This completes the proof of Lemma~\ref{lem-skeleton}.
\end{proof}

Now, we prove Proposition \ref{prop-profend}.

\begin{proof}[Proof of Proposition \ref{prop-profend}]
Since $q$  is nondecreasing, it follows that $|Dq|(I_0)=q(\frac{1}{2}-)-q(-\frac{1}{2}+)$, where
$q(-\frac{1}{2}+):=\lim_{t\to -\frac{1}{2}^+}q(t)$ and $q(\frac{1}{2}-):=\lim_{t\to\frac{1}{2}^-}q(t).$
If $|Dq|(I_0)=0,$ then \eqref{fuduan} holds  automatically. Let $|Dq|(I_0)\in (0,\infty).$ For any $t\in I_0,$ define
\begin{align*}
\widetilde q(t):=\frac{q(t)-q(-\frac{1}{2}+)}{|Dq|(I_0)}.
\end{align*}
Then $0\leq\widetilde q\leq1$ on $I_0$, $\widetilde q(-\frac{1}{2}+)=0$, and $\widetilde q(\frac{1}{2}-)=1.$
By \eqref{bisuo}, we find that
$m_{1,\gamma}(W_q)=|Dq|(I_0)m_{1,\gamma}(W_{\widetilde q}).$
Thus, to prove \eqref{fuduan}, it suffices to show that
\begin{align}\label{poumu}
 m_{1,\gamma}(W_{\widetilde q})\geq C_{N,1,\gamma}^{\mathrm{cell}}.
\end{align}
Fix $m\in\mathbb N$.  For any $k\in\{0,\ldots,m\}$,  let $a_k:=-\frac12+\frac{k}{m}.$
Define the function $q_m:I_0\to\mathbb R$ as follows.  For any
$k\in\{0,\ldots,m-1\}$ and   $t\in(a_k,a_{k+1})$, let
\begin{align*}
q_m(t):=a_k+\frac{1}{m}\widetilde q\left(m[t-a_k]-\frac12\right).
\end{align*}
Moreover, for any $k\in\{1,\ldots, m-1\},$ let $q_m(a_k):=a_k.$ Then we have, for any
$k\in\left\{0,\ldots,m-1\right\}$, $q_m(a_k+)=a_k$ and $q_m(a_{k+1}-)=a_k+\frac{1}{m}=a_{k+1}$,
which imply that, for any $k\in\{1,\ldots,m-1\},$
$Dq_m\left(\left\{a_k\right\}\right)=0$ and
\begin{align*}
|Dq_m|\left(a_k,a_{k+1}\right)=q_m(a_{k+1}-)-q_m(a_k+)=\frac{1}{m}.
\end{align*}
Consequently,
\begin{align*}
|Dq_m|(I_0)=\sum_{k=0}^{m-1}|Dq_m|\left(a_k,a_{k+1}\right)=1.
\end{align*}
Furthermore, for any $k\in \{0,\ldots,m-1\}$ and $t\in(a_k,a_{k+1})$,
\begin{align*}
|q_m(t)-t|=\frac{1}{m}\left|\widetilde q\left(m[t-a_k]-\frac12\right)-m(t-a_k)
\right|\leq\frac{1}{m},
\end{align*}
where we used the facts that $0\leq\widetilde{q}\leq 1$ and $m(t-a_k)\in (0,1)$ in the  last step.
This implies that
\begin{align*}
\sup_{t\in I_0}|q_m(t)-t|\leq\frac{1}{m},
\end{align*}
and hence $W_{q_m}\to \ell$ in $L^1(Q_0)$ as $m\to\infty$. From this and Lemma \ref{lem-mplsc}, we infer that
\begin{align*}
C_{N,1,\gamma}^{\mathrm{cell}}=m_{1,\gamma}(\ell)\leq\varliminf_{m\to\infty}m_{1,\gamma}(W_{q_m}).
\end{align*}
Using this, we conclude that, to prove \eqref{poumu}, it suffices to show that, for any $m\in\mathbb{N},$
\begin{align}\label{zibijiao}
m_{1,\gamma}(W_{q_m})\leq m_{1,\gamma}(W_{\widetilde q}).
\end{align}
Let $\varepsilon\in (0,\infty)$.
Applying Lemma~\ref{lem-nearmin} with $q:=\widetilde q$, we
conclude that there exist
$\rho\in(0,\frac{1}{20})$ and sequences
$\{\widehat\lambda_j\}_{j\in\mathbb N}\subset(0,\infty)$ and
$\{\widehat v_j\}_{j\in\mathbb N}\subset L^\infty(Q_0)$
having the properties stated in Lemma~\ref{lem-nearmin}. In particular,
$\widehat\lambda_j\to\lambda_\gamma$, $\widehat v_j\to W_{\widetilde q}$ in $L^1(Q_0)$ as $j\to\infty$, and
\begin{align}\label{pouxiufu}
\varlimsup_{j\to\infty}F_{\widehat\lambda_j,1,\gamma}\left(\widehat v_j;Q_0\right)
\leq m_{1,\gamma}(W_{\widetilde q})+\varepsilon.
\end{align}

Now, we claim that there exist sequences $\{\Lambda^{(m)}_j\}_{j\in\mathbb{N}}\subset (0,\infty)$ and
$\{V_j^{(m)}\}_{j\in\mathbb{N}}\subset L^1(Q_0)$ such that $\lim_{j\to\infty}\Lambda^{(m)}_j=\lambda_\gamma$,
$\lim_{j\to\infty}V_j^{(m)}=W_{q_m}$ in $L^1(Q_0)$, and
\begin{align}\label{poushen}
m_{1,\gamma}(W_{q_m})\leq\varlimsup_{j\to\infty}F_{\Lambda_j^{(m)},1,\gamma}\left(V_j^{(m)};Q_0\right)
\leq m_{1,\gamma}(W_{\widetilde q})+\varepsilon.
\end{align}
If this claim holds, then, letting $\varepsilon\to0^+$ in \eqref{poushen},
we conclude that \eqref{zibijiao} holds, which completes the proof of the present proposition.

Next, we prove the above claim. Note that the intervals $\{(a_i,a_{i+1})\}_{i=0}^{m-1}$
are pairwise disjoint and cover $I_0$ except for the partition points.
Taking their Cartesian products in the $N$ coordinate directions
gives $m^N$ pairwise disjoint open cubes that cover $Q_0$ except for
their boundaries. More precisely, let
$\mathcal A_m:=\left\{0,\ldots,m-1\right\}^N.$
For any $\alpha=(\alpha_1,\ldots,\alpha_N)\in\mathcal A_m$, define
\begin{align*}
Q_\alpha:=\prod_{i=1}^N(a_{\alpha_i},a_{\alpha_i+1}).
\end{align*}
Then the family $\mathscr Q_m:=\left\{Q_\alpha:\alpha\in\mathcal A_m\right\}$
consists of pairwise disjoint open cubes and covers $Q_0$ except for
their boundaries.  Furthermore, for any $\alpha \in\mathcal{A}_m$, the center of $Q_\alpha$ is
\begin{align*}
c_\alpha&:=\left(\frac{a_{\alpha_1}+a_{\alpha_1+1}}{2},\ldots,\frac{a_{\alpha_N}+a_{\alpha_N+1}}{2}\right)
=-\frac12(1,\ldots,1)+\frac{1}{m}\left(\alpha_1+\frac12,\ldots,\alpha_N+\frac12\right),
\end{align*}
and hence $Q_\alpha=c_\alpha+\frac{1}{m}Q_0.$
From the definitions of $W_{q_m}$ and $q_m$, it follows that, for any $\alpha\in \mathcal{A}_m$
with $k:=\alpha_N$ and for any $x\in Q_\alpha$,
\begin{align}\label{poujuti}
W_{q_m}(x)&=q_m(x_N)=a_k+\frac{1}{m}\widetilde q\left(m[x_N-a_k]-\frac12\right)
=a_k+\frac{1}{m}\widetilde{q}\left(m\left[x_N-(c_\alpha)_N\right]\right)\notag\\
&=a_k+\frac{1}{m}\widetilde{q}\left([m(x-c_\alpha)]_N\right)
=a_k+\frac{1}{m}W_{\widetilde{q}}(m(x-c_\alpha)).
\end{align}
Let $j\in\mathbb{N}$. Replacing $W_{\widetilde q}$ in \eqref{poujuti} by $\widehat{v}_j$,
let
\begin{align*}
V_j^{(m)}(x)
:=\begin{cases}
\displaystyle
a_{\alpha_N}+\frac{1}{m}
\widehat v_j\left(m(x-c_\alpha)\right),
&x\in Q_\alpha\text{ for some }\alpha\in\mathcal A_m,\\
W_{q_m}(x),&\displaystyle x\in Q_0\setminus
\left(\bigcup_{\alpha\in\mathcal A_m}Q_\alpha\right).
 \end{cases}
\end{align*}
By a change of variables, we find that
\begin{align*}
\left\|V_j^{(m)}-W_{q_m}\right\|_{L^1(Q_0)}=\sum_{\alpha\in\mathcal A_m}
\int_{Q_\alpha}\left|V_j^{(m)}-W_{q_m}\right|\,dx=\frac{1}{m}
\left\|\widehat v_j-W_{\widetilde q}\right\|_{L^1(Q_0)}.
\end{align*}
Thus, $\lim_{j\to\infty}V_j^{(m)}=W_{q_m}$ in $L^1(Q_0)$. Define
$\Lambda_j^{(m)}:=m^\gamma\widehat\lambda_j.$
From Lemma~\ref{lem-scaling}, we infer that, for any $\alpha\in\mathcal A_m$,
\begin{align*}
F_{\Lambda_j^{(m)},1,\gamma}\left(V_j^{(m)};Q_\alpha\right)=\frac{1}{m^N}
F_{\widehat\lambda_j,1,\gamma}\left(\widehat v_j;Q_0\right).
\end{align*}
Consequently,
\begin{align}\label{baohe}
\sum_{\alpha\in\mathcal A_m}F_{\Lambda_j^{(m)},1,\gamma}\left(V_j^{(m)};Q_\alpha\right)
=F_{\widehat\lambda_j,1,\gamma}\left(\widehat v_j;Q_0\right).
\end{align}
Let
\begin{align*}
F_{\Lambda_j^{(m)},1,\gamma}^{\rm cross}\left(V_j^{(m)};\mathscr Q_m\right)
&:=\Lambda_j^{(m)}\sum_{\substack{\alpha,\beta\in\mathcal A_m\\\alpha\neq\beta}}
\underset{\scriptscriptstyle \left|V_j^{(m)}(x)-V_j^{(m)}(y)\right|
>\Lambda_j^{(m)}|x-y|^{\gamma+1}}
{\int_{Q_\alpha}\int_{Q_\beta}}
|x-y|^{\gamma-N}\,dy\,dx.
\end{align*}
By \eqref{baohe}, we conclude that
\begin{align*}
F_{\Lambda_j^{(m)},1,\gamma}\left(V_j^{(m)};Q_0\right)&\leq\sum_{\alpha\in\mathcal A_m}
F_{\Lambda_j^{(m)},1,\gamma}\left(V_j^{(m)};Q_\alpha\right)+ F_{\Lambda_j^{(m)},1,\gamma}^{\rm cross}
\left(V_j^{(m)};\mathscr Q_m\right)\notag\\
&=F_{\widehat\lambda_j,1,\gamma}\left(\widehat v_j;Q_0\right)+ F_{\Lambda_j^{(m)},1,\gamma}^{\rm cross}
\left(V_j^{(m)};\mathscr Q_m\right).
\end{align*}
Letting $j\to\infty$ in the above inequality and using \eqref{pouxiufu}, we obtain
\begin{align*}
m_{1,\gamma}\left(W_{q_m}\right)&\leq\varlimsup_{j\to\infty}F_{\Lambda_j^{(m)},1,
\gamma}\left(V_j^{(m)};Q_0\right)
\leq m_{1,\gamma}\left(W_{\widetilde q}\right)+\varepsilon+\varlimsup_{j\to\infty}
F_{\Lambda_j^{(m)},1,\gamma}^{\rm cross}\left(V_j^{(m)};\mathscr Q_m\right).
\end{align*}
Thus, to prove \eqref{poushen}, we just need to show that
\begin{align}\label{yous}
\lim_{j\to\infty} F_{\Lambda_j^{(m)},1,\gamma}^{\rm cross}\left(V_j^{(m)};
\mathscr Q_m\right)=0.
\end{align}
For any $\alpha\in\mathcal A_m$, define
\begin{align*}
Q_{\alpha,\rho}:=\left\{x\in Q_\alpha:
\operatorname{dist}(x,\partial Q_\alpha)>\frac{\rho}{m}\right\}.
\end{align*}
For any $j\in\mathbb N$, write \begin{align*}
F_{\Lambda_j^{(m)},1,\gamma}^{\rm cross}\left(V_j^{(m)};\mathscr Q_m\right)
=:I_j^{\rm near}+I_j^{\rm far},
\end{align*}
where the domains of integrations in $I_j^{\rm near}$ and $I_j^{\rm far}$
are restricted to $|x-y|<\frac\rho m$ and $|x-y|\geq\frac\rho m$, respectively.
By \eqref{jinbian}, we find that
$V_j^{(m)}=W_{q_m}$ on $Q_\alpha\setminus Q_{\alpha,\rho}$.
Let $\alpha,\beta\in\mathcal A_m$ be distinct and let
$(x,y)\in Q_\alpha\times Q_\beta$ satisfy $|x-y|<\frac\rho m$.
Since $y\notin Q_\alpha$ and $x\notin Q_\beta$, it follows that
 \begin{align*}
\max\left\{\operatorname{dist}(x,\partial Q_\alpha),
\operatorname{dist}(y,\partial Q_\beta)\right\}
\leq|x-y|<\frac\rho m.
\end{align*}
Consequently, $V_j^{(m)}(x)=W_{q_m}(x)$ and
$V_j^{(m)}(y)=W_{q_m}(y)$, which implies that
\begin{align*}
I_j^{\rm near}\leq
F_{\Lambda_j^{(m)},1,\gamma}^{\rm cross}\left(W_{q_m};\mathscr Q_m\right).
\end{align*}
If $\gamma\in(0,\infty)$, then
 \begin{align*}
\lim_{j\to\infty}\Lambda_j^{(m)}\left(\frac\rho m\right)^{\gamma+1}=\infty.
\end{align*}
Using this and the fact that
$-\frac12\leq V_j^{(m)}\leq\frac12$, we conclude that, for all sufficiently large $j$
and any $x,y\in Q_0$ satisfying $|x-y|\geq\frac\rho m$,
 \begin{align*}
\left|V_j^{(m)}(x)-V_j^{(m)}(y)\right|
\leq1
\leq\Lambda_j^{(m)}\left(\frac\rho m\right)^{\gamma+1}
\leq\Lambda_j^{(m)}|x-y|^{\gamma+1}.
\end{align*}
Consequently,  $I_j^{\rm far}=0$
for all sufficiently large $j$.
If $\gamma\in(-\infty,-1)$, then
 \begin{align*}
I_j^{\rm far}
\leq\Lambda_j^{(m)}
\underset{\scriptscriptstyle |x-y|\geq\frac\rho m}
{\int_{Q_0}\int_{Q_0}}|x-y|^{\gamma-N}\,dy\,dx
\lesssim\Lambda_j^{(m)}\left(\frac\rho m\right)^\gamma\to0
\end{align*}
as $j\to\infty$, with $m$ and $\rho$ fixed. Thus, in both cases,
 \begin{align}\label{poujiao}
\varlimsup_{j\to\infty}
F_{\Lambda_j^{(m)},1,\gamma}^{\rm cross}(V_j^{(m)};\mathscr Q_m)
\leq\varlimsup_{j\to\infty}
F_{\Lambda_j^{(m)},1,\gamma}^{\rm cross}(W_{q_m};\mathscr Q_m).
\end{align}
Define $q_m^{\rm ext}:\mathbb R\to\mathbb R$ by setting, for any $t\in\mathbb R$,
\begin{align*}
q_m^{\rm ext}(t):=
\begin{cases}
-\dfrac{1}{2}, &\displaystyle t\in\left(-\infty,-\frac{1}{2}\right],\\
q_m(t),&\displaystyle t\in\left(-\frac{1}{2},\frac{1}{2}\right),\\
\dfrac{1}{2},&\displaystyle t\in\left[\frac{1}{2},\infty\right).
 \end{cases}
\end{align*}
Obviously, $q_m^{\rm ext}$ is nondecreasing and, for any $k\in\{0,\ldots,m\},$
\begin{align}\label{wuyuanzi}
Dq_m^{\rm ext}\left(\left\{a_k\right\}\right)=0.
\end{align}
Choose $\chi\in C_{\rm c}^\infty(\mathbb R^N)$ such that $\chi=1$ on an open
neighborhood $U$ of $\overline{Q}_0$. For any $z\in\mathbb{R}^N,$ let
\begin{align*}
W_m^{\rm ext}(z):=\chi(z)q_m^{\rm ext}(z_N).
\end{align*}
Then $W_m^{\rm ext}\in BV(\mathbb R^N)\cap L^\infty(\mathbb R^N)$ and
$W_m^{\rm ext}=W_{q_m}$ on $Q_0$.
By an argument similar to that  used in \eqref{eq:dou2}, we find that, for any $i\in\{1,\ldots, N-1\}$,
\begin{align*}
(D_iW_m^{\rm ext})\mathopen{\scalebox{1.6}{$\llcorner$}}U=0\ \ \text{and}\ \ (D_NW_m^{\rm ext})\mathopen{\scalebox{1.6}{$\llcorner$}}U
=\left(\mathcal L^{N-1}\otimes Dq_m^{\rm ext}\right)\mathopen{\scalebox{1.6}{$\llcorner$}}U.
\end{align*}
This implies that
\begin{align*}
|DW_m^{\rm ext}|\mathopen{\scalebox{1.6}{$\llcorner$}}U=\left(\mathcal L^{N-1}\otimes|Dq_m^{\rm ext}|\right)\mathopen{\scalebox{1.6}{$\llcorner$}}U,
\end{align*}
and hence, for any Borel sets $A'\subset\mathbb R^{N-1}$ and
$B\subset\mathbb R$ satisfying $A'\times B\subset U$,
\begin{align*}
|DW_m^{\rm ext}|(A'\times B)=\mathcal L^{N-1}(A')|Dq_m^{\rm ext}|(B).
\end{align*}
From this and \eqref{wuyuanzi}, we deduce that
$|DW_m^{\rm ext}|(S_{\mathscr Q_m})=0.$
Thus, all the assumptions of Lemma~\ref{lem-skeleton} are satisfied.
Using this and applying Lemma~\ref{lem-skeleton} with $U:=W_{q_m}$
and $\mathscr Q:=\mathscr Q_m$, we conclude that
\begin{align*}
\lim_{\lambda\to\lambda_\gamma}
F_{\lambda,1,\gamma}^{\rm cross}(W_{q_m};\mathscr Q_m)=0.
\end{align*}
This, together with \eqref{poujiao}, proves \eqref{yous},
which completes the proof of Proposition~\ref{prop-profend}.
\end{proof}

\section{The $\Gamma$-Limsup Inequality on Bounded Domains}
\label{sec-limsup}

The aim of this section is to show the following
$\Gamma$-limsup inequality on bounded domains.
\begin{theorem}\label{thm-limsup}
Let  $p\in[1,\infty)$, $\gamma\in\mathbb R\setminus\{0\}$, and
$\Omega\subset\mathbb R^N$ be a bounded open interval when $N=1$ or a
bounded Lipschitz domain when $N\geq2$.
\begin{itemize}
\item[$\mathrm{(i)}$] Let $u\in W^{1,p}(\Omega)$, and let
$\{\lambda_j\}_{j\in\mathbb N}\subset(0,\infty)$ satisfy
$\lim_{j\to\infty}\lambda_j=\lambda_\gamma$. Then there exists
$\{u_j\}_{j\in\mathbb N}\subset L^p(\Omega)$ satisfying $u_j\to u$ in
$L^p(\Omega)$ as $j\to\infty$ and
\begin{align*}
\varlimsup_{j\to\infty}F_{\lambda_j,p,\gamma}\left(u_j;\Omega\right)
\leq C_{N,p,\gamma}^{\rm cell}\int_\Omega|\nabla u|^p\,dx,
\end{align*}
where $C_{N,p,\gamma}^{\rm cell}$ is as in \eqref{baochang}.

\item[$\mathrm{(ii)}$] Let $u\in BV(\Omega)$,  and let
$\{\lambda_j\}_{j\in\mathbb N}\subset(0,\infty)$ satisfy
$\lim_{j\to\infty}\lambda_j=\lambda_\gamma$. Then there exists
$\{u_j\}_{j\in\mathbb N}\subset L^1(\Omega)$ satisfying $u_j\to u$ in
$L^1(\Omega)$ as $j\to\infty$ and
\begin{align*}
\varlimsup_{j\to\infty}F_{\lambda_j,1,\gamma}\left(u_j;\Omega\right)
\leq C_{N,1,\gamma}^{\rm cell}|Du|(\Omega).
\end{align*}
\end{itemize}
\end{theorem}

The proof of Theorem \ref{thm-limsup} is given in Subsection \ref{subsec-recovery-density}.
In Subsection \ref{sub5.1}, we modify the sequences
realizing the affine cell constant so that they agree with $\ell$ near the boundary of $Q_0$.
In Subsection \ref{sub5.2}, we use these sequences to construct recovery sequences
for continuous finite piecewise affine functions.
In Subsection \ref{subsec-recovery-density}, we extend this construction to Sobolev and
$BV$ functions by approximation and prove Theorem \ref{thm-limsup}.
Finally, in Subsection \ref{sec-domain}, we treat the critical case $\gamma=-p$ and the case
where the $\Gamma$-limit is zero,  completing  the proofs of Theorems \ref{thm-negdom} and \ref{thm-domfull}.

\subsection{Boundary Modification of Sequences Realizing the Affine Cell Constant}\label{sub5.1}

We use Lemmas \ref{lem-cellbds} and \ref{lem-cutoff} to establish the following boundary
modification result.

\begin{proposition}\label{prop-repair}
Let $p\in[1,\infty)$ and $\gamma\in\mathbb R\setminus\{0\}$.
For any $j\in\mathbb{N},$ let $\rho_j:=\frac{1}{100j}.$
Then there exist sequences $\{\Lambda_j\}_{j\in\mathbb N}\subset(0,\infty)$ and
$\{V_j\}_{j\in\mathbb N}\subset L^\infty(Q_0)$ having the following properties.
\begin{enumerate}[{\rm(i)}]
\item $\lim_{j\to\infty}\Lambda_j=\lambda_\gamma$, $\lim_{j\to\infty}V_j=\ell$ in $L^p(Q_0)$, and
\begin{align}\label{xiuzhu}
\lim_{j\to\infty}F_{\Lambda_j,p,\gamma}\left(V_j;Q_0\right)=C_{N,p,\gamma}^{\mathrm{cell}}.
\end{align}
\item For any $j\in\mathbb N$,
\begin{align}\label{xiubian}
V_j=\ell\quad\text{on}\quad
\left\{z\in Q_0:\operatorname{dist}(z,\partial Q_0)<\rho_j\right\}.
\end{align}
\item It holds that $\lim_{j\to\infty}\frac{\Lambda_j^{-\frac{p}{\gamma}}}{\rho_j}=0.$
\item If $\gamma\in(0,\infty)$, then
 \begin{align}\label{xiuwucha}
\lim_{j\to\infty}\frac{\|V_j-\ell\|_{L^\infty(Q_0)}^p}
{\Lambda_j^p\rho_j^{p+\gamma}}=0.
\end{align}
\end{enumerate}
\end{proposition}
\begin{proof}
Let $\{\widetilde\Lambda_n\}_{n\in\mathbb N}\subset(0,\infty)$ and $\{v_n\}_{n\in\mathbb N}\subset L^p(Q_0)$
be as in Lemma \ref{lem-cellbds}(ii). For any $j\in\mathbb N$, let $T_j$ be as in \eqref{jieding}
with $M:=j$.
From \eqref{deng} and the fact that $T_j(\ell)=\ell$ on $Q_0$,
we infer that, for any $n\in\mathbb{N},$
\begin{align*}
|T_j(v_n)-\ell|=|T_j(v_n)-T_j(\ell)|\leq|v_n-\ell|.
\end{align*}
This further implies that
$\lim_{n\to\infty}T_j(v_n)=\ell$ in $L^p(Q_0).$
Moreover, for any $\lambda\in (0,\infty)$,
 \begin{align*}
F_{\lambda,p,\gamma}\left(T_j(v_n);Q_0\right)\leq F_{\lambda,p,\gamma}\left(v_n;Q_0\right).
\end{align*}
Define the $1$-Lipschitz function $\theta:\mathbb R\to[0,1]$ by
setting, for any $t\in\mathbb R$,
\begin{align*}
\theta(t):=
\begin{cases}
0,&t\in(-\infty,3],\\
t-3,&t\in(3,4),\\
1,&t\in[4,\infty),
\end{cases}
\end{align*}
and, for any $z\in Q_0$, let
\begin{align*}
\eta_j(z):=\theta\left(\frac{\operatorname{dist}(z,\partial Q_0)}{\rho_j}\right)=
\begin{cases}
0,&\operatorname{dist}(z,\partial Q_0)\leq3\rho_j,\\
\displaystyle\frac{\operatorname{dist}(z,\partial Q_0)}{\rho_j}-3,
&3\rho_j<\operatorname{dist}(z,\partial Q_0)<4\rho_j,\\
1,&\operatorname{dist}(z,\partial Q_0)\geq4\rho_j.
\end{cases}
\end{align*}
Then, for any $z,w\in Q_0$,
\begin{align*}
|\eta_j(z)-\eta_j(w)|
&=\left|\theta\left(\frac{\operatorname{dist}(z,\partial Q_0)}{\rho_j}\right)
-\theta\left(\frac{\operatorname{dist}(w,\partial Q_0)}{\rho_j}\right)\right|\\
&\leq\left|\frac{\operatorname{dist}(z,\partial Q_0)}{\rho_j}
-\frac{\operatorname{dist}(w,\partial Q_0)}{\rho_j}\right|
\leq\frac{|z-w|}{\rho_j}.
\end{align*}
Thus, $\eta_j$ is a Lipschitz function on $Q_0$ with Lipschitz constant at most $\frac1{\rho_j}$.
For any $j\in\mathbb N$, let
\begin{align*}
E_j:=\left\{z\in Q_0:\eta_j(z)<1\right\}.
\end{align*}
Then $E_j$ is contained in $\{z\in Q_0:\operatorname{dist}(z,\partial Q_0)<4\rho_j\}$.
Thus, $|E_j|\leq C_N\rho_j$.

Now, for any $j,n\in\mathbb N$, let $V_{j,n}:=\eta_jT_j(v_n)+(1-\eta_j)\ell$
and $\Lambda_{j,n}:=\frac{\widetilde\Lambda_n}{\alpha_j}$,
where $\{\alpha_j\}_{j\in\mathbb N}\subset(0,1)$ satisfies $\lim_{j\to\infty}\alpha_j=1$.
Applying (i) and (ii) of Lemma \ref{lem-cutoff}, we find that, for any $j,n\in\mathbb N$,
 \begin{align}\label{xiuneng}
F_{\Lambda_{j,n},p,\gamma}\left(V_{j,n};Q_0\right)
&\leq\frac{F_{\alpha_j\Lambda_{j,n},p,\gamma}(T_j(v_n);Q_0)}{\alpha_j^p}
+\frac{F_{\alpha_j\Lambda_{j,n},p,\gamma}(\ell;E_j,Q_0)}{\alpha_j^p}\notag\\
&\quad+\frac{\mathcal H^{N-1}\left(\mathbb S^{N-1}\right)\left\|T_j(v_n)-\ell\right\|_{L^p(Q_0)}^p}
{|\gamma|(1-\alpha_j)^p\rho_j^p}\notag\\
&\leq\frac{F_{\widetilde\Lambda_n,p,\gamma}(v_n;Q_0)}{\alpha_j^p}
+\frac{C_{N,\gamma}\rho_j}{\alpha_j^p}
+\frac{\mathcal H^{N-1}\left(\mathbb S^{N-1}\right)\left\|T_j(v_n)-\ell\right\|_{L^p(Q_0)}^p}
{|\gamma|(1-\alpha_j)^p\rho_j^p}.
\end{align}
From Lemma \ref{lem-cellbds}(ii), we infer that
$\lim_{n\to\infty}v_n=\ell$ in $L^p(Q_0)$ and $\lim_{n\to\infty}\widetilde\Lambda_n=\lambda_\gamma.$ Thus,
 for any given $j\in\mathbb N$, $\lim_{n\to\infty}V_{j,n}=\ell$ in $L^p(Q_0)$,
$\lim_{n\to\infty}\Lambda_{j,n}=\lambda_\gamma$, and
$\lim_{n\to\infty}\frac{\Lambda_{j,n}^{-\frac{p}{\gamma}}}{\rho_j}=0.$
Moreover, when $\gamma\in(0,\infty)$, we have, for any given $j\in\mathbb N$,
 \begin{align*}
\lim_{n\to\infty}\frac{\|V_{j,n}-\ell\|_{L^\infty(Q_0)}^p}
{\Lambda_{j,n}^p\rho_j^{p+\gamma}}
\leq\lim_{n\to\infty}\frac{\left(j+\frac12\right)^p}
{\Lambda_{j,n}^p\rho_j^{p+\gamma}}=0.
\end{align*}
This, combined with \eqref{xiuneng} and the
definition of $C_{N,p,\gamma}^{\mathrm{cell}},$
 implies that, for any given $j\in\mathbb N$, there exists
$n(j)\in\mathbb N$ such that
 \begin{align*}
\left\|V_{j,n(j)}-\ell\right\|_{L^p(Q_0)}\leq\frac1j,\ \
\Lambda_{j,n(j)}^{-\frac p\gamma}\leq\frac1j,\ \
\frac{\Lambda_{j,n(j)}^{-\frac{p}{\gamma}}}{\rho_j}\leq\frac1j,
\end{align*}
and
\begin{align*}
F_{\Lambda_{j,n(j)},p,\gamma}\left(V_{j,n(j)};Q_0\right)
\leq\frac{C_{N,p,\gamma}^{\mathrm{cell}}+\frac1j}{\alpha_j^p}
+\frac{C_{N,\gamma}\rho_j}{\alpha_j^p}+\frac1j.
\end{align*}
When $\gamma\in(0,\infty)$, we can additionally require that
 \begin{align*}
\frac{\|V_{j,n(j)}-\ell\|_{L^\infty(Q_0)}^p}
{\Lambda_{j,n(j)}^p\rho_j^{p+\gamma}}\leq\frac1j.
\end{align*}
Finally, for any $j\in\mathbb N$, let
$\Lambda_j:=\Lambda_{j,n(j)}$ and $V_j:=V_{j,n(j)}$.
Letting $j\to\infty$, we then conclude that
\begin{align*}
\varlimsup_{j\to\infty}F_{\Lambda_j,p,\gamma}\left(V_j;Q_0\right)
\leq C_{N,p,\gamma}^{\mathrm{cell}}.
\end{align*}
On the other hand, since $C_{N,p,\gamma}^{\mathrm{cell}}=m_{p,\gamma}(\ell)$, it follows that
\begin{align*}
C_{N,p,\gamma}^{\mathrm{cell}}
\leq\varliminf_{j\to\infty}F_{\Lambda_j,p,\gamma}\left(V_j;Q_0\right),
\end{align*}
and hence \eqref{xiuzhu} holds. The other assertions of the present
proposition follow from the construction and the choice of $n(j)$. This completes the proof
of Proposition \ref{prop-repair}.
\end{proof}

\subsection{Recovery Sequences for Continuous Piecewise Affine Functions}\label{sub5.2}
We begin by  recalling the definition of continuous piecewise affine functions.

\begin{definition}\label{def-pa}
Let  $a\in\mathbb R^N\setminus\{\mathbf{0}\}$ and $c\in\mathbb R$.
A set of the form $\{x\in\mathbb R^N:a\cdot x\leq c\}$ is called a
\emph{closed affine half-space} in $\mathbb R^N$. A \emph{closed convex
polyhedron} $P$ in $\mathbb R^N$ is an intersection of finitely many closed
affine half-spaces with nonempty interior.
Let $U\subset\mathbb R^N$ be a bounded convex open set. A continuous function
$w:U\to\mathbb R$ is said to be \emph{continuous finite piecewise affine} if
there exist  finitely many closed convex polyhedra $\{P_i\}_{i=1}^s$ with
pairwise disjoint interiors, vectors $\{\xi_i\}_{i=1}^s$, and numbers
$\{b_i\}_{i=1}^s$ such that $U\subset\bigcup_{i=1}^sP_i$ and, for any
$i\in\{1,\ldots,s\}$ and any $x\in U\cap P_i$,
$w(x)=\xi_i\cdot x+b_i.$
\end{definition}
\begin{remark}
Let $U\subset\mathbb R^N$ be a bounded convex open set, and let
$w:U\to\mathbb R$ be continuous finite piecewise affine. By the assumption
that $U$ is convex and Definition \ref{def-pa}, we conclude that, for any $x,y\in U$,
\begin{align}\label{kong}
|w(x)-w(y)|\leq L|x-y|,
\end{align}
where $L:=\max_{1\leq i\leq s}|\xi_i|$.
\end{remark}

We will show the following  $\Gamma$-limsup inequality for continuous finite piecewise
affine functions in this subsection.
\begin{proposition}\label{prop-pa}
Let  $p\in[1,\infty)$, $\gamma\in\mathbb R\setminus\{0\}$, and
$\Omega\subset\mathbb R^N$ be a bounded open set.
Assume that $U\supset\overline\Omega$ is a bounded convex open set and
$w:U\to\mathbb R$ is continuous finite piecewise affine. Let
$\{\lambda_j\}_{j\in\mathbb N}\subset(0,\infty)$ satisfy $\lambda_j\to\lambda_\gamma$ as
$j\to\infty$. Then there exists $\{w_j\}_{j\in\mathbb N}\subset L^p(\Omega)$ satisfying
$w_j\to w$ in $L^p(\Omega)$ as $j\to\infty$ and
\begin{align}\label{fenshang}
\varlimsup_{j\to\infty}F_{\lambda_j,p,\gamma}\left(w_j;\Omega\right)
\leq C_{N,p,\gamma}^{\rm cell}\int_\Omega|\nabla w|^p\,dx.
\end{align}
\end{proposition}

\begin{proof}
Let $\{P_i\}_{i=1}^s$, $\{\xi_i\}_{i=1}^s$, and
$\{b_i\}_{i=1}^s$ be as in Definition \ref{def-pa}.
For any $i\in\{1,\ldots,s\}$, let
$D_i:=\Omega\cap \operatorname{int}(P_i).$
Since the interiors of the polyhedra are pairwise disjoint and
$U\subset\bigcup_{i=1}^sP_i$, it follows that
\begin{align}\label{papart}
\left|\Omega\setminus\left(\bigcup_{i=1}^sD_i\right)\right|=0.
\end{align}
We first prove the present proposition under the additional assumption
\begin{align}\label{feiling}
 a_*:=\min_{1\leq i\leq s}|\xi_i|>0.
\end{align}
Let $\{\Lambda_j\}_{j\in\mathbb{N}}\subset (0,\infty),$
$\{V_j\}_{j\in\mathbb{N}}\subset L^\infty(Q_0)$, and $\{\rho_j\}_{j\in\mathbb{N}}\subset (0,\infty)$
be as in Proposition \ref{prop-repair}.
By the fact that $\lim_{j\to\infty}\lambda_j=\lambda_\gamma$, we find that there exists a sequence
$\{k_j\}_{j\in\mathbb N}\subset\mathbb N$ such that
\begin{align}\label{paidx}
\lim_{j\to\infty}k_j=\infty\quad\text{and}\quad
\lim_{j\to\infty}\left(\frac{\Lambda_{k_j}}{\lambda_j}\right)^{\frac p\gamma}=0.
\end{align}
For any $j\in\mathbb N$, let $\widehat\Lambda_j:=\Lambda_{k_j}$,
$\widehat V_j:=V_{k_j}$, and $\widehat\rho_j:=\rho_{k_j}$.
Since $\lim_{j\to\infty}\widehat{\rho}_j=0,$
after discarding finitely many indices if necessary,
we may assume that, for any $j\in\mathbb{N},$
$0<\widehat\rho_j<\frac{1}{2}$. Let  $i\in\{1,\ldots,s\}$ and $j\in\mathbb{N}.$ Define
\begin{align}\label{pah}
h_{j,i}:=\left(\frac{|\xi_i|\widehat\Lambda_j}{\lambda_j}\right)^{\frac{p}{\gamma}}.
\end{align}
From \eqref{paidx}, we deduce that
\begin{align}\label{pahzero}
\lim_{j\to\infty}\max_{1\leq i\leq s}h_{j,i}=0.
\end{align}
Choose  matrix $R_i\in\mathcal O(N)$ such that $R_ie_N=\frac{\xi_i}{|\xi_i|},$ where
$i\in\{1,\ldots, s\}$ and $e_N:=(0,\ldots,0,1)$. For any $j\in\mathbb{N}$,
$i\in\{1,\ldots,s\}$, and $k\in\mathbb Z^N$, define
\begin{align}\label{shui}
Q_{j,i,k}:=h_{j,i}R_i(k+Q_0)=h_{j,i}R_ik+h_{j,i}R_iQ_0.
\end{align}
Let $\mathscr Q_{j,i}$ denote the finite family of those cubes whose
closures are contained in $D_i$. Let $Q\in\mathscr Q_{j,i}$. By \eqref{shui}, we conclude that
$Q=x_Q+h_{j,i}R_iQ_0$, where $x_Q$ is the center of $Q$.
Thus, for any $x\in Q$, there exists a unique $z\in Q_0$ such that
$x=x_Q+h_{j,i}R_iz$. In what follows, for $Q\in\mathscr Q_{j,i}$ and $x\in Q$, the
corresponding point $z\in Q_0$ is always understood to be given by
$z=R_i^{-1}\frac{x-x_Q}{h_{j,i}}$ or, equivalently, $x=x_Q+h_{j,i}R_iz$.
Recall that, for any $z\in\mathbb{R}^N,$ $\ell(z):=z_N.$
From Definition \ref{def-pa} and the definition of $R_i$, we infer that, for any $x\in Q$,
\begin{align*}
w(x)&=w(x_Q)+\xi_i\cdot(x-x_Q)=w(x_Q)+h_{j,i}\xi_i\cdot R_i z
=w(x_Q)+h_{j,i}(R_i^{-1}\xi_i)\cdot z\notag\\
&=w(x_Q)+|\xi_i|h_{j,i}e_N\cdot z=w(x_Q)+|\xi_i|h_{j,i}\ell(z).
\end{align*}
Now define $w_j$ as follows.  For any $i\in\{1,\ldots,s\}$, $Q\in\mathscr Q_{j,i}$,
and $x\in Q$, let
\begin{align}\label{padef}
w_j(x)&:=w(x_Q)+|\xi_i|h_{j,i}\widehat V_j
\left(R_i^{-1}\frac{x-x_Q}{h_{j,i}}\right)=w(x_Q)+|\xi_i|h_{j,i}\widehat V_j(z),
\end{align}
and, for any $x\in\Omega\setminus\left(\bigcup_{i=1}^s\bigcup_{Q\in\mathscr Q_{j,i}}Q\right),$
let $w_j(x):=w(x)$. Let $n_{j,i}:=\#\mathscr Q_{j,i},$ which denotes the number of cubes in the finite
family $\mathscr Q_{j,i}$. Since $\mathscr Q_{j,i}$ consists of cubes of edge length $h_{j,i}$
whose closures are contained in $D_i$, it follows that, for any $i\in\{1,\ldots,s\}$,
\begin{align*}
D_i\setminus\left(\bigcup_{Q\in\mathscr Q_{j,i}}\overline{Q}\right)\subset\left\{
x\in D_i:\operatorname{dist}(x,\mathbb R^N\setminus D_i)\leq\sqrt N\,h_{j,i}\right\}.
\end{align*}
This, together with \eqref{pahzero}, implies that
\begin{align}\label{pafill}
\lim_{j\to\infty}\left|D_i\setminus\left(\bigcup_{Q\in\mathscr Q_{j,i}}Q\right)\right|=0,
\end{align}
and hence
\begin{align}\label{pavol}
\lim_{j\to\infty}n_{j,i}h_{j,i}^N=|D_i|.
\end{align}
By this,  a change of variables, \eqref{padef}, \eqref{pahzero},
and the fact that $\lim_{j\to\infty}\widehat{V}_j=\ell$ in $L^p(Q_0),$ we find that
\begin{align*}
\|w_j-w\|_{L^p(\Omega)}^p&=\sum_{i=1}^s\sum_{Q\in\mathscr Q_{j,i}}
|\xi_i|^ph_{j,i}^{N+p}\left\|\widehat V_j-\ell\right\|_{L^p(Q_0)}^p\\
&=\sum_{i=1}^s|\xi_i|^pn_{j,i}h_{j,i}^{N+p}\left\|\widehat V_j-\ell\right\|_{L^p(Q_0)}^p\to0
\end{align*}
as $j\to\infty$. Thus,  to prove the  present proposition under the assumption \eqref{feiling},
it remains to show \eqref{fenshang}.

For any $j\in\mathbb N$, define $A_j:=\bigcup_{i=1}^s\bigcup_{Q\in\mathscr Q_{j,i}}(Q\times Q)$
and
\begin{align*}
R_j:=\lambda_j^p
\underset{\scriptscriptstyle |w_j(x)-w_j(y)|>\lambda_j|x-y|^{1+\gamma/p}}
{\iint_{(\Omega\times\Omega)\setminus A_j}}
|x-y|^{\gamma-N}\,dy\,dx.
\end{align*}
Since the cubes are pairwise disjoint, it follows that, for any $j\in\mathbb{N},$
\begin{align}\label{haq}
F_{\lambda_j,p,\gamma}\left(w_j;\Omega\right)=\sum_{i=1}^s\sum_{Q\in\mathscr Q_{j,i}}
F_{\lambda_j,p,\gamma}\left(w_j;Q\right)+R_j.
\end{align}
From Lemma \ref{lem-scaling} and \eqref{padef}, we deduce that,  for any  $i\in\{1,\ldots,s\}$,
$j\in\mathbb N$, and $Q\in\mathscr Q_{j,i}$,
\begin{align*}
F_{\lambda_j,p,\gamma}\left(w_j;Q\right)=|\xi_i|^ph_{j,i}^N
F_{\widehat\Lambda_j,p,\gamma}\left(\widehat V_j;Q_0\right).
\end{align*}
By this, \eqref{pavol}, and Proposition \ref{prop-repair}, we obtain
\begin{align*}
\lim_{j\to\infty}\sum_{i=1}^s\sum_{Q\in\mathscr Q_{j,i}}
F_{\lambda_j,p,\gamma}\left(w_j;Q\right)
&=\lim_{j\to\infty}F_{\widehat\Lambda_j,p,\gamma}\left(\widehat V_j;Q_0\right)
\sum_{i=1}^s|\xi_i|^pn_{j,i}h_{j,i}^N\notag\\
&=C_{N,p,\gamma}^{\rm cell}\sum_{i=1}^s|\xi_i|^p|D_i|
=C_{N,p,\gamma}^{\rm cell}\int_\Omega|\nabla w|^p\,dx.
\end{align*}
This, combined with \eqref{haq}, implies that, to prove \eqref{fenshang}, it suffices to show that
$\lim_{j\to\infty}R_j=0.$

From \eqref{pah}, Proposition \ref{prop-repair}(iii), and \eqref{feiling},
we infer that, for any $i\in\{1,\ldots,s\}$,
\begin{align*}
\lim_{j\to\infty}\frac{\lambda_j^{-\frac{p}{\gamma}}}{\widehat\rho_jh_{j,i}}
=\lim_{j\to\infty}|\xi_i|^{-\frac{p}{\gamma}}
\frac{\widehat\Lambda_j^{-\frac{p}{\gamma}}}{\widehat\rho_j}=0.
\end{align*}
For any $j\in\mathbb N$, let
\begin{align}\label{jianju}
d_j:=\frac{\widehat\rho_j}{2}\min_{1\leq i\leq s}h_{j,i}.
\end{align}
Then, for any $i\in\{1,\ldots,s\}$ and $j\in\mathbb N$,
\begin{align}\label{rt}
d_j<  \widehat\rho_jh_{j,i}
\end{align}
and
\begin{align}\label{lo}
\lim_{j\to\infty}\frac{\lambda_j^{-\frac{p}{\gamma}}}{d_j}=0.
\end{align}
On the other hand, by the definition of $\rho_j$,
we find that $\lim_{j\to\infty}\widehat\rho_j=0$. This, together with \eqref{rt}, implies that
$\lim_{j\to\infty}\frac{d_j}{h_{j,i}}=0.$
Let $j\in\mathbb{N}$. Define
\begin{align*}
R_j^{\rm near}:=\lambda_j^p\underset{\substack{\scriptscriptstyle |x-y|<d_j\\
\scriptscriptstyle |w_j(x)-w_j(y)|>\lambda_j|x-y|^{1+\gamma/p}}}
{\iint_{(\Omega\times\Omega)\setminus A_j}}|x-y|^{\gamma-N}\,dy\,dx
\end{align*}
and
\begin{align*}
R_j^{\rm far}:=\lambda_j^p\underset{\substack{\scriptscriptstyle |x-y|\geq d_j\\
\scriptscriptstyle |w_j(x)-w_j(y)|>\lambda_j|x-y|^{1+\gamma/p}}}
{\iint_{(\Omega\times\Omega)\setminus A_j}}
|x-y|^{\gamma-N}\,dy\,dx.
\end{align*}
Then $R_j=R_j^{\rm near}+R_j^{\rm far}$.

First, we estimate $R_j^{\rm near}$.
For any $i\in\{1,\ldots,s\}$ and $Q\in\mathscr Q_{j,i}$, define
\begin{align*}
Q^{\rm in}:=\left\{x\in Q:\operatorname{dist}(x,\partial Q)>d_j\right\}
\end{align*}
and
\begin{align*}
E_j:=\Omega\setminus\left(\bigcup_{i=1}^s
\bigcup_{Q\in\mathscr Q_{j,i}}Q^{\rm in}\right)=\bigcap_{i=1}^{s}\bigcap_{Q\in\mathscr Q_{j,i}}
\left(\Omega\setminus Q^{\rm in}\right).
\end{align*}
Then
\begin{align}\label{jindui}
\left\{(x,y)\in(\Omega\times\Omega)\setminus A_j:
|x-y|<d_j\right\}\subset E_j\times E_j.
\end{align}
By \eqref{rt}, we conclude that, for any $x\in Q\setminus Q^{\rm in}$ with $x=x_Q+h_{j,i}R_iz$,
\begin{align*}
\operatorname{dist}(z,\partial Q_0)=\frac{\operatorname{dist}(x,\partial Q)}{h_{j,i}}
\leq\frac{d_j}{h_{j,i}}<\widehat\rho_j.
\end{align*}
This, combined with  \eqref{xiubian} and \eqref{padef}, implies that, for any $x\in Q\setminus Q^{\rm in}$,
\begin{align*}
w_j(x)=w(x_Q)+|\xi_i|h_{j,i}\widehat V_j(z)=w(x_Q)+|\xi_i|h_{j,i}\ell(z)=w(x).
\end{align*}
On the other hand, from the definitions of $w_j$ and $w$, we infer that
$w_j=w$ on $\Omega\setminus\left(\bigcup_{i=1}^s\bigcup_{Q\in\mathscr Q_{j,i}}Q\right)$,
which, together with the fact that
\begin{align}\label{youshang}
E_j=\left[\Omega\setminus\bigcup_{i=1}^s
\bigcup_{Q\in\mathscr Q_{j,i}}Q\right]
\cup\left[\bigcup_{i=1}^s\bigcup_{Q\in\mathscr Q_{j,i}}
\left(Q\setminus Q^{\rm in}\right)\right],
\end{align}
further implies that $w_j=w$ on $E_j$. Using this, \eqref{jindui},
and \eqref{kong}, we find that
 \begin{align*}
&\left\{(x,y)\in(\Omega\times\Omega)\setminus A_j:|x-y|<d_j,\
|w_j(x)-w_j(y)|>\lambda_j|x-y|^{1+\gamma/p}\right\}\\
&\quad\subset\left\{(x,y)\in E_j\times E_j:
\lambda_j|x-y|^{1+\gamma/p}<|w(x)-w(y)|\leq L|x-y|\right\}\\
&\quad\subset\left\{(x,y)\in E_j\times E_j:
\lambda_j|x-y|^{\gamma/p}<L\right\}.
\end{align*}
Thus,
 \begin{align}\label{sad}
R_j^{\rm near}&\leq\lambda_j^p
\underset{\scriptscriptstyle \lambda_j|x-y|^{\gamma/p}<L}
{\int_{E_j}\int_{E_j}}|x-y|^{\gamma-N}\,dy\,dx\notag\\
&\leq\lambda_j^p\mathcal H^{N-1}(\mathbb S^{N-1})|E_j|
\underset{\scriptscriptstyle \lambda_jr^{\gamma/p}<L}
{\int_0^\infty}r^{\gamma-1}\,dr
=\frac{\mathcal H^{N-1}(\mathbb S^{N-1})}{|\gamma|}L^p|E_j|.
\end{align}
Using \eqref{youshang}, \eqref{pafill}, and \eqref{papart}, we obtain
\begin{align*}
|E_j|&=\left|\Omega\setminus\bigcup_{i=1}^s\bigcup_{Q\in\mathscr Q_{j,i}}Q\right|
+\sum_{i=1}^s\sum_{Q\in\mathscr Q_{j,i}}|Q\setminus Q^{\rm in}|\notag\\
&=\left|\Omega\setminus\bigcup_{i=1}^s
\bigcup_{Q\in\mathscr Q_{j,i}}Q\right|
+\sum_{i=1}^s n_{j,i}h_{j,i}^N
\left[1-\left(1-\frac{2d_j}{h_{j,i}}\right)^N\right]\to 0
\end{align*}
as $j\to\infty.$ From this and \eqref{sad}, it follows that
$\lim_{j\to\infty}  R_j^{\rm near}=0.$

Next, we estimate $R_j^{\rm far}$. We consider the following two ranges for $\gamma$.

Case 1: $\gamma\in(0,\infty)$.
In this case, let
 \begin{align*}
\mathcal T_j:=\left\{(x,y)\in\Omega\times\Omega:|x-y|\geq d_j,\
|w_j(x)-w_j(y)|>\lambda_j|x-y|^{1+\gamma/p}\right\}.
\end{align*}
We claim that, for all sufficiently large $j$,
\begin{align}\label{eq:painc}
|\mathcal T_j|=0.
\end{align}
From \eqref{padef} and the affine representation of $w$ on each cube
$Q\in\mathscr Q_{j,i}$, it follows that
 \begin{align*}
\|w_j-w\|_{L^\infty(\Omega)}\leq
\max_{1\leq i\leq s}\left(|\xi_i|h_{j,i}\right)
\|\widehat V_j-\ell\|_{L^\infty(Q_0)}.
\end{align*}
Using this, \eqref{pah}, \eqref{jianju}, and \eqref{kong}, we conclude that,
for almost all $(x,y)\in\mathcal T_j$,
 \begin{align*}
\frac{|w_j(x)-w_j(y)|^p}{\lambda_j^p|x-y|^{p+\gamma}}&\leq3^{p-1}
\left[\frac{L^p}{\lambda_j^p|x-y|^\gamma}+\frac{|w_j(x)-w(x)|^p}{\lambda_j^p|x-y|^{p+\gamma}}
+\frac{|w_j(y)-w(y)|^p}{\lambda_j^p|x-y|^{p+\gamma}}\right]\\
&\leq3^{p-1}\left[\frac{L^p}{\lambda_j^pd_j^\gamma}
+\frac{2\|w_j-w\|_{L^\infty(\Omega)}^p}{\lambda_j^pd_j^{p+\gamma}}\right]
\leq C\left[\frac{1}{\widehat\Lambda_j^p\widehat\rho_j^\gamma}
+\frac{\|\widehat V_j-\ell\|_{L^\infty(Q_0)}^p}{\widehat\Lambda_j^p\widehat\rho_j^{p+\gamma}}
\right]=:\varepsilon_j,
\end{align*}
where $L:=\max_{1\leq i\leq s}|\xi_i|$ is as in \eqref{kong} and the positive constant
$C$ depends only on $p$, $\gamma$, $a_*$, and $L$.
Using this, Proposition \ref{prop-repair}(iii), and \eqref{xiuwucha}, we obtain  $\lim_{j\to\infty}\varepsilon_j=0.$ However, since
$(x,y)\in\mathcal T_j$, it follows that
 \begin{align*}
1<\frac{|w_j(x)-w_j(y)|^p}{\lambda_j^p|x-y|^{p+\gamma}}.
\end{align*}
For all sufficiently large $j$, we have $\varepsilon_j\leq1$, and hence
$|\mathcal T_j|=0$.
This proves  the claim \eqref{eq:painc}.
Consequently, $R_j^{\rm far}=0$ for all sufficiently large $j$.

Case 2: $\gamma\in(-\infty,0)$.
In this case,
by the definition of $R_j^{\rm far}$, we find that
\begin{align*}
R_j^{\rm far}&\leq\lambda_j^p\underset{\scriptscriptstyle |x-y|\geq d_j}
{\int_\Omega\int_\Omega}|x-y|^{\gamma-N}\,dy\,dx
\leq\lambda_j^p|\Omega|\int_{|h|\geq d_j}|h|^{\gamma-N}\,dh
=\frac{\mathcal H^{N-1}(\mathbb S^{N-1})}{|\gamma|}|\Omega|\lambda_j^pd_j^{\gamma}.
\end{align*}
From this and \eqref{lo}, we deduce that
$\lim_{j\to\infty} R_j^{\rm far}=0.$

Combining the estimates of $R_j^{\rm near}$ and $R_j^{\rm far}$, we conclude that
$\lim_{j\to\infty}R_j=0$.

Finally, we remove the  additional assumption \eqref{feiling}. Fix $q\in\mathbb R^N\setminus\{{\bf0}\}$.
Choose a sequence $\{\varepsilon_m\}_{m\in\mathbb{N}}\subset(0,\infty)$ such that
$\lim_{m\to\infty}\varepsilon_m=0$ and, for any $m\in\mathbb{N}$ and $i\in\{1,\ldots,s\}$,
$\xi_i+\varepsilon_mq\neq0$. For any $m\in\mathbb{N}$ and $x\in U,$ define
$w^{(m)}(x):=w(x)+\varepsilon_mq\cdot x.$
Then $w^{(m)}\to w$ in $W^{1,p}(\Omega)$. Moreover, for any given
$m\in\mathbb{N},$ $w^{(m)}$ is continuous finite piecewise affine  and
\begin{align*}
a_*^{(m)}:=\min_{1\leq i\leq s}|\xi_i+\varepsilon_mq|>0.
\end{align*}
Thus, for any given $m\in\mathbb{N},$ there exists a sequence
$\{w_j^{(m)}\}_{j\in\mathbb N}\subset L^p(\Omega)$ such that
$w_j^{(m)}\to w^{(m)}$ in $L^p(\Omega)$ and
 \begin{align*}
\varlimsup_{j\to\infty}F_{\lambda_j,p,\gamma}(w_j^{(m)};\Omega)\leq
C_{N,p,\gamma}^{\mathrm{cell}}\int_\Omega|\nabla w^{(m)}|^p\,dx.
\end{align*}
Then, for any given $m\in\mathbb{N},$ there exists $J_m\in\mathbb{N}$
such that, for any $j\geq J_m$, $\|w_j^{(m)}-w^{(m)}\|_{L^p(\Omega)}\leq\frac1m$ and
 \begin{align*}
F_{\lambda_j,p,\gamma}(w_j^{(m)};\Omega)\leq C_{N,p,\gamma}^{\mathrm{cell}}
\int_\Omega\left|\nabla w^{(m)}\right|^p\,dx+\frac1m.
\end{align*}
Without loss of generality, we may assume that $\{J_m\}_{m\in\mathbb{N}}$ is strictly increasing.
For any given $j\geq J_1$, let $m(j)$ be the unique integer such that $J_{m(j)}\leq j<J_{m(j)+1}.$
Since $\{J_m\}_{m\in\mathbb{N}}$ is strictly increasing, it follows that $\lim_{j\to\infty}m(j)=\infty.$
For any $j\geq J_1$, define $w_j:=w_j^{(m(j))}$, and let $w_j:=w$ for $j<J_1$. Then we have, for any $j\geq J_1$,
\begin{align*}
\|w_j-w\|_{L^p(\Omega)}&\leq\left\|w_j-w^{(m(j))}\right\|_{L^p(\Omega)}+
\left\|w^{(m(j))}-w\right\|_{L^p(\Omega)}\\
&\leq\frac1{m(j)}+\left\|w^{(m(j))}-w\right\|_{L^p(\Omega)}
\end{align*}
and
 \begin{align*}
F_{\lambda_j,p,\gamma}(w_j;\Omega)\leq C_{N,p,\gamma}^{\mathrm{cell}}\int_\Omega
|\nabla w^{(m(j))}|^p\,dx+\frac1{m(j)}.
\end{align*}
Letting $j\to\infty$, we find that $\lim_{j\to\infty}w_j=w$ in $L^p(\Omega)$ and \eqref{fenshang} holds.
This completes the proof of Proposition \ref{prop-pa}.
\end{proof}

\subsection{Proof of Theorem \ref{thm-limsup}}\label{subsec-recovery-density}
The following standard density result follows from \cite[Proposition~2.8]{et99} (see also
\cite[Theorem~1]{vs14}).

\begin{lemma}\label{lem-padensity}
Let  $p\in [1,\infty)$, let  $\Omega\subset\mathbb R^N$ be a bounded open interval
when $N=1$ or a bounded Lipschitz domain when $N\ge2$, and let $u\in W^{1,p}(\Omega)$.  Then there exist a bounded
convex open set $U$ containing $\overline\Omega$ and a sequence
$\{u_m\}_{m\in\mathbb N}\subset C(U)$ of continuous finite
piecewise-affine functions such that $\lim_{m\to\infty}u_m= u$ in $W^{1,p}(\Omega).$
\end{lemma}
The following standard strict approximation result is precisely \cite[Theorem~5.3]{eg15}.
\begin{lemma}\label{lem-strictapp}
Let $\Omega\subset \mathbb{R}^N$ be an open set,  and let  $u\in BV(\Omega)$.
Then there  exists a sequence $\{u_m\}_{m\in\mathbb{N}}\subset W^{1,1}(\Omega)\cap C^\infty(\Omega)$ such that
$\lim_{m\to\infty}u_m=u$ in $L^1(\Omega)$ and
$$\lim_{m\to\infty}\int_{\Omega}|\nabla u_m(x)|\,dx=|Du|(\Omega).$$
\end{lemma}

Theorem \ref{thm-limsup} follows from Proposition \ref{prop-pa} by the density and the diagonal argument. For the convenience of the reader,
we provide the details here.

\begin{proof}[Proof of Theorem \ref{thm-limsup}]
We first show (i). Let  $u\in W^{1,p}(\Omega)$ and
$\{u_m\}_{m\in\mathbb{N}}$ be as in Lemma \ref{lem-padensity}.
By Proposition \ref{prop-pa}, we  conclude that, for any given
$m\in\mathbb{N}$, there exists a sequence $\{u_j^{(m)}\}_{j\in\mathbb{N}}\subset L^p(\Omega)$ such that
$u_j^{(m)}\to u_m$ in $L^p(\Omega)$ as $j\to\infty$ and
\begin{align*}
\varlimsup_{j\to\infty} F_{\lambda_j,p,\gamma}\left(u_j^{(m)};\Omega\right)\leq
C_{N,p,\gamma}^{\mathrm{cell}}\int_\Omega|\nabla u_m|^p\,dx.
\end{align*}
Thus, for any given $m\in\mathbb{N},$ there  exists  $J_m\in\mathbb{N}$ such that, for any $j\geq J_m$,
$\|u_j^{(m)}-u_m\|_{L^p(\Omega)}\leq\frac1m$ and
\begin{align*}
F_{\lambda_j,p,\gamma}\left(u_j^{(m)};\Omega\right)\leq C_{N,p,\gamma}^{\mathrm{cell}}
\int_\Omega\left|\nabla u_m\right|^p\,dx+\frac1m.
\end{align*}
Without loss of generality, we may assume that $\{J_m\}_{m\in\mathbb{N}}$ is strictly increasing.
Let   $j\geq J_1$, and let $m(j)$ be the unique integer such that $J_{m(j)}\leq j<J_{m(j)+1}.$
Define $u_j:=u_j^{(m(j))}$ for $j\geq J_1$, and let $u_j:=u$ for $j<J_1$. Since $\{J_m\}_{m\in\mathbb{N}}$ is strictly increasing,
it follows that  $\lim_{j\to\infty}m(j)=\infty.$ Then we have
\begin{align*}
\|u_j-u\|_{L^p(\Omega)}\leq
\left\|u_j^{(m(j))}-u_{m(j)}\right\|_{L^p(\Omega)}
+\left\|u_{m(j)}-u\right\|_{L^p(\Omega)}\leq\frac1{m(j)}+\left\|u_{m(j)}-u\right\|_{L^p(\Omega)}
\end{align*}
and
\begin{align*}
F_{\lambda_j,p,\gamma}\left(u_j;\Omega\right)\leq C_{N,p,\gamma}^{\mathrm{cell}}
\int_\Omega\left|\nabla u_{m(j)}\right|^p\,dx+\frac1{m(j)}.
\end{align*}
Letting $j\to\infty$ in the above estimates, we find that $\lim_{j\to\infty}u_j=u$ in $L^p(\Omega)$ and the
desired inequality holds, which proves (i).

Applying (i) with $p=1$, Lemma \ref{lem-strictapp}, and the same
diagonal argument used in the proof of (i),
we conclude that  (ii) holds;
 we omit the details
here.
This completes the proof of Theorem \ref{thm-limsup}.
\end{proof}

\subsection{The Critical Case and the Zero Limit}\label{sec-domain}

In this subsection, we complete the proofs of Theorems \ref{thm-negdom} and \ref{thm-domfull}.
We need the following lemma, which  is a part of \cite[Theorem~2]{bn20na} and \cite[Theorem~1]{bn18}.
\begin{lemma}\label{lem-critsmth}
Let  $p\in[1,\infty)$, and let $U\subset\mathbb R^N$ be a bounded open interval when $N=1$ or a bounded smooth domain when $N\geq2$.
Then there exists a positive constant $\kappa_{N,p}$, independent of $U$,
such that, as $\lambda\to0^+$, the family $F_{\lambda,p,-p}(\cdot;U)$ converges, in the sense of
$\Gamma$-convergence in $L^p(U)$, to the functional defined by setting, for any $u\in L^p(U)$,
\begin{align*}
u\longmapsto
\begin{cases}
\kappa_{N,p}\displaystyle\int_U|\nabla u|^p\,dx,
&p\in(1,\infty)\ \hbox{and}\ u\in W^{1,p}(U),\\
\kappa_{N,1}|Du|(U),
&p=1\ \hbox{and}\ u\in BV(U),\\
\infty,&\hbox{otherwise}.
\end{cases}
\end{align*}
\end{lemma}
By Lemma \ref{lem-critsmth} and Theorem \ref{thm-limsup}, we have the
following result.
\begin{proposition}\label{prop-critdom}
Let  $p\in[1,\infty)$, and let $\Omega\subset\mathbb R^N$ be a
bounded open interval when $N=1$ or a bounded Lipschitz domain when $N\geq2$.
Then the family $F_{\lambda,p,-p}(\cdot;\Omega)$ converges, in the sense of
$\Gamma$-convergence in $L^p(\Omega)$, to $\Psi_{p,-p}^{\Omega}$ as
$\lambda\to0^+$.
\end{proposition}
\begin{proof}
We first identify the cell constant. Choose an increasing sequence
$\{U_m\}_{m\in\mathbb{N}}$ of bounded smooth domains in $Q_0$ such that $\bigcup_{m\in\mathbb N}U_m=Q_0$ (see, for instance, \cite[Proposition~8.2.1]{d08}). Let $\{\lambda_j\}_{j\in\mathbb N}\subset(0,\infty)$ satisfy $\lambda_j\to0$ as $j\to\infty$. Let
$\{v_j\}_{j\in\mathbb N}\subset L^p(Q_0)$ satisfy
$v_j\to\ell$ in $L^p(Q_0)$ as $j\to\infty$. By the
$\Gamma$-liminf inequality in Lemma \ref{lem-critsmth}, we
obtain, for any $m\in\mathbb{N}$,
\begin{align*}
\varliminf_{j\to\infty}F_{\lambda_j,p,-p}\left(v_j;Q_0\right)
&\geq\varliminf_{j\to\infty}F_{\lambda_j,p,-p}\left(v_j;U_m\right)
\geq\kappa_{N,p}|U_m|.
\end{align*}
Letting $m\to\infty$ in the above inequality and then taking the infimum over
all admissible sequences in \eqref{baochang}, we find that
$C_{N,p,-p}^{\rm cell}\geq\kappa_{N,p}.$

To show the reverse inequality, choose a sequence
$\{V_m\}_{m\in\mathbb{N}}$ of bounded smooth domains such
that $Q_0\subset V_m$ for any $m\in\mathbb{N}$ and $|V_m\setminus Q_0|\to0$
as $m\to\infty$ (see, for instance, \cite[Theorem~8.3.1]{d08}).
For any $m\in\mathbb N$, choose a sequence
$\{\lambda_j^{(m)}\}_{j\in\mathbb N}\subset(0,\infty)$ satisfying
$\lambda_j^{(m)}\to0$ as $j\to\infty$. From the $\Gamma$-limsup inequality in Lemma \ref{lem-critsmth}, it follows that there exists a sequence
$\{w_j^{(m)}\}_{j\in\mathbb N}\subset L^p(V_m)$ satisfying
$w_j^{(m)}\to\ell$ in $L^p(V_m)$ as $j\to\infty$ and
\begin{align*}
\varlimsup_{j\to\infty}F_{\lambda_j^{(m)},p,-p}\left(w_j^{(m)};V_m\right)
\leq\kappa_{N,p}|V_m|.
\end{align*}
Hence, by \eqref{baochang}, we conclude that
\begin{align}\label{reverse}
C_{N,p,-p}^{\rm cell}\leq\varliminf_{j\to\infty}
F_{\lambda_j^{(m)},p,-p}\left(w_j^{(m)};Q_0\right)
\leq\varlimsup_{j\to\infty} F_{\lambda_j^{(m)},p,-p}\left(w_j^{(m)};V_m\right)
\leq\kappa_{N,p}|V_m|.
\end{align}
Letting $m\to\infty$ in \eqref{reverse} yields the reverse inequality $C_{N,p,-p}^{\rm cell}\leq\kappa_{N,p}.$
Therefore,
\begin{align}\label{yanggu}
C_{N,p,-p}^{\rm cell}=\kappa_{N,p}.
\end{align}

Next, we prove the $\Gamma$-liminf inequality on $\Omega$.
Let $\{\lambda_j\}_{j\in\mathbb N}\subset(0,\infty)$ satisfy
$\lambda_j\to0$ as $j\to\infty$. Let $u\in L^p(\Omega)$, and let
$\{u_j\}_{j\in\mathbb N}\subset L^p(\Omega)$ satisfy
$u_j\to u$ in $L^p(\Omega)$ as $j\to\infty$.
Define
\begin{align*}
L:=\varliminf_{j\to\infty}F_{\lambda_j,p,-p}\left(u_j;\Omega\right).
\end{align*}
There is nothing to prove when $L=\infty$, so we assume that $L<\infty$.
Choose an increasing sequence $\{\Omega_m\}_{m\in\mathbb{N}}$ of bounded smooth domains in $\Omega$ such that $\bigcup_{m\in\mathbb N}\Omega_m=\Omega$
(see, for instance, \cite[Proposition~8.2.1]{d08}).
By Lemma \ref{lem-critsmth}, we find that, for any $m\in\mathbb{N},$
\begin{align}\label{reverse1}
L\geq
\begin{cases}
\kappa_{N,p}\displaystyle\int_{\Omega_m}|\nabla u|^p\,dx,&p\in(1,\infty),\\
\kappa_{N,1}|Du|(\Omega_m),&p=1.
\end{cases}
\end{align}
Letting $m\to\infty$ in \eqref{reverse1} and using \eqref{yanggu},
we conclude that
\begin{align*}
\varliminf_{j\to\infty}F_{\lambda_j,p,-p}\left(u_j;\Omega\right)\geq\Psi_{p,-p}^{\Omega}(u),
\end{align*}
which shows the $\Gamma$-liminf inequality. On the  other hand, from
Theorem \ref{thm-limsup}, we infer that the $\Gamma$-limsup inequality
holds. This completes the proof of Proposition \ref{prop-critdom}.
\end{proof}

To complete the classification on bounded domains, we need the following lemma,
which is precisely \cite[Theorem~29]{n25}.
\begin{lemma}\label{lem-zeroreg}
Let  $p\in[1,\infty)$, $\gamma\in(-1,0)$, and $U\subset\mathbb R^N$ be a measurable set. Then the family
$F_{\lambda,p,\gamma}(\cdot;U)$ converges, in the sense of
$\Gamma$-convergence in $L^p(U)$, to the zero functional as $\lambda\to0^+$.
\end{lemma}

\begin{proof}[Proofs of Theorems \ref{thm-negdom} and \ref{thm-domfull}]
Theorem \ref{thm-negdom} follows from Theorems \ref{thm-pliminf}, \ref{thm-onelinf},
and \ref{thm-limsup}. Theorem \ref{thm-domfull} then follows from Theorem \ref{thm-negdom},
Proposition \ref{prop-critdom}, and Lemma \ref{lem-zeroreg}.
\end{proof}

\section{From Bounded Domains to the Whole Space}
\label{sec-whole}

In this section, we prove Theorem \ref{thm-whlfull}. We begin with the following result.
\begin{proposition}\label{prop-whlzero}
Let  $p\in[1,\infty)$ and $\gamma\in(-1,0)$. Then the family
$F_{\lambda,p,\gamma}(\cdot;\mathbb R^N)$ converges, in the sense of
$\Gamma$-convergence in $L^1_{\rm loc}(\mathbb R^N)$, to the zero
functional as $\lambda\to0^+$.
\end{proposition}

\begin{proof}
Since the $\Gamma$-liminf inequality holds automatically,
we just need to show the $\Gamma$-limsup inequality.
Let $\{\lambda_j\}_{j\in\mathbb N}\subset(0,\infty)$ satisfy
$\lambda_j\to0$ as $j\to\infty$. Let $u\in L^1_{\rm loc}(\mathbb R^N)$.
For any $M\in (0,\infty)$, let $Q_M:=(-M,M)^N$, and let
$T_M$ be the truncation operator defined in \eqref{jieding}.
Define $w_M:=\mathbf1_{Q_M}T_M(u).$
Then, for any given $R\in (0,\infty)$ and any $M\in[R,\infty)$,
\begin{align*}
\|w_M-u\|_{L^1(Q_R)}=\|T_M(u)-u\|_{L^1(Q_R)}
\leq\int_{Q_R\cap\left\{|u|>M\right\}}|u|\,dx\to0
\end{align*}
as $M\to\infty$. Thus, $w_M\to u$ in $L^1_{\rm loc}(\mathbb R^N)$
as $M\to\infty$. Using Lemma \ref{lem-zeroreg} with $U:=\mathbb{R}^N$, we
conclude that, for  any given $M\in (0,\infty)$,
there  exists  a sequence
$\{v_j^{(M)}\}_{j\in\mathbb{N}}\subset L^p(\mathbb{R}^N)$ such that
\begin{align*}
v_j^{(M)}\to w_M\quad\text{in }L^p(\mathbb R^N)
\ \ \text{and}\ \
F_{\lambda_j,p,\gamma}\left(v_j^{(M)};\mathbb R^N\right)\to0
\end{align*}
as $j\to \infty.$
From this, H\"older's inequality, and the same diagonal argument
used in the proof of Proposition \ref{prop-pa}, we infer that there exists
a sequence $\{u_j\}_{j\in\mathbb{N}}\subset L^1_{\mathrm{loc}}(\mathbb{R}^N)$ such that
$\lim_{j\to\infty}u_j=u$ in $L^1_{\mathrm{loc}}(\mathbb{R}^N)$ and
\begin{align*}
\lim_{j\to\infty}F_{\lambda_j,p,\gamma}\left(u_{j};\mathbb R^N\right)=0.
\end{align*}
This  completes the proof of Proposition \ref{prop-whlzero}.
\end{proof}

The remaining two subsections, Subsections \ref{sec-nondeg} and
\ref{sec-endpoint}, are devoted, respectively, to handling the nondegenerate range
$\gamma\in(-\infty,-1]\cup(0,\infty)$, with $(N,p,\gamma)\neq(1,1,-1)$,
and the exceptional endpoint case  $(N,p,\gamma)=(1,1,-1)$.
Theorem \ref{thm-whlfull} follows directly from Proposition
\ref{prop-whlzero} and Theorems \ref{thm-whlndeg} and \ref{thm-whlendpt}.

\subsection{The Nondegenerate Range}\label{sec-nondeg}

The aim of this subsection is the following result.

\begin{theorem}\label{thm-whlndeg}
Let  $p\in[1,\infty)$ and $\gamma\in(-\infty,-1]\cup(0,\infty)$. Assume that $(N,p,\gamma)\neq(1,1,-1)$. Then
$F_{\lambda,p,\gamma}(\cdot;\mathbb R^N)$ converges, in the sense of
$\Gamma$-convergence in $L^1_{\rm loc}(\mathbb R^N)$, to
$\Psi_{p,\gamma}^{\rm hom}$ as $\lambda\to\lambda_\gamma$.
\end{theorem}
We first prove the following  $\Gamma$-liminf inequality.
\begin{proposition}\label{prop-whllower}
Let $p\in[1,\infty)$ and
$\gamma\in(-\infty,-1]\cup(0,\infty)$.
Let $\{\lambda_j\}_{j\in\mathbb N}\subset(0,\infty)$ satisfy
$\lambda_j\to\lambda_\gamma$ as $j\to\infty$. Let $u\in L^1_{\rm loc}(\mathbb R^N)$, and let $\{u_j\}_{j\in\mathbb N}\subset L^1_{\rm loc}(\mathbb R^N)$ satisfy
$u_j\to u$ in $L^1_{\rm loc}(\mathbb R^N)$ as $j\to\infty$. Then
\begin{align}\label{zhengxia}
\varliminf_{j\to\infty}F_{\lambda_j,p,\gamma}\left(u_{j};\mathbb R^N\right)\geq
\begin{cases}
C_{N,p,\gamma}^{\rm cell}\displaystyle\int_{\mathbb R^N}|\nabla u|^p\,dx,
&p\in(1,\infty),\ u\in\dot W^{1,p}(\mathbb R^N),\\
C_{N,1,\gamma}^{\rm cell}|Du|(\mathbb R^N),
&p=1,\ u\in\dot{BV}(\mathbb R^N),\\
\infty,&\text{otherwise}.
\end{cases}
\end{align}
\end{proposition}

\begin{proof}
Let
\begin{align*}
L:=\varliminf_{j\to\infty}
F_{\lambda_j,p,\gamma}\left(u_{j};\mathbb R^N\right).
\end{align*}
If $L=\infty$, there is nothing to prove. We now assume $L<\infty$. Passing to a subsequence and relabelling if necessary, we may assume that
\begin{align*}
\lim_{j\to\infty}F_{\lambda_{j},p,\gamma}\left(u_{j};\mathbb R^N\right)= L.
\end{align*}
Fix $R\in (0,\infty)$ and  $Q_R:=(-R,R)^N$.

We first consider the case $p=1.$ In this case,  by the fact that
$\lim_{j\to\infty}u_j=u$ in $L^1(Q_R)$ and by Theorem \ref{thm-domfull}, we  find that $u\in BV(Q_R)$ and $C_{N,1,\gamma}^{\rm cell}|Du|(Q_R)\leq L.$
Letting $R\to\infty$, we show \eqref{zhengxia} in this case.

Now, let $p\in (1,\infty)$.
Let $M\in (0,\infty)$ and $T_M$ be  as  in
\eqref{jieding}.
Since $u_j\to u$ in $L^1_{\mathrm{loc}}(\mathbb{R}^N)$ as $j\to\infty$, it follows that
\begin{align*}
\left\|T_M(u_j)-T_M(u)\right\|_{L^p(Q_R)}^p
\leq(2M)^{p-1}\left\|T_M(u_j)-T_M(u)\right\|_{L^1(Q_R)}\to0
\end{align*}
as $j\to\infty$. By \eqref{deng},  we conclude that, for any $j\in\mathbb{N},$
\begin{align*}
F_{\lambda_j,p,\gamma}(T_M(u_j);Q_R)\leq F_{\lambda_j,p,\gamma}(u_j;Q_R)
\leq F_{\lambda_j,p,\gamma}(u_j;\mathbb R^N).
\end{align*}
This, together with Theorem \ref{thm-domfull}, implies that
$T_M(u)\in W^{1,p}(Q_R)$ and
\begin{align}\label{jiexia}
C_{N,p,\gamma}^{\rm cell}\int_{Q_R}|\nabla T_M(u)|^p\,dx\leq L.
\end{align}
From this and the Poincar\'e inequality, we infer that
\begin{align*}
&\|T_M(u)\|_{L^p(Q_R)}\notag\\
&\quad\leq
\left\|T_M(u)-\frac1{|Q_R|}\int_{Q_R}T_M(u)\,dx\right\|_{L^p(Q_R)} +|Q_R|^{\frac{1}{p}}\left|\frac1{|Q_R|}\int_{Q_R}T_M(u)\,dx\right|\notag\\
&\quad\lesssim
R\|\nabla T_M(u)\|_{L^p(Q_R)}
 +|Q_R|^{\frac{1}{p}-1}\|u\|_{L^1(Q_R)}\lesssim R\left(\frac{L}{C_{N,p,\gamma}^{\rm cell}}\right)^{\frac{1}{p}}
 +|Q_R|^{\frac{1}{p}-1}\|u\|_{L^1(Q_R)}.
\end{align*}
Thus,
\begin{align*}
\sup_{M\in(0,\infty)}\left\|T_M(u)\right\|_{W^{1,p}(Q_R)}<\infty.
\end{align*}
Using this and the fact that
the space $W^{1,p}(Q_R)$ is reflexive, we conclude that  there
exist a sequence $\{M_k\}_{k\in\mathbb N}\subset(0,\infty)$ and a function
$v\in W^{1,p}(Q_R)$ such that
$M_k\to\infty$ and $T_{M_k}(u)\to v$ weakly in $W^{1,p}(Q_R)$ as $k\to\infty$.
On the other hand,
\begin{align*}
\left\|T_{M_k}(u)-u\right\|_{L^1(Q_R)}\leq\int_{Q_R\cap\left\{|u|>M_k\right\}}|u|\,dx
\to0
\end{align*}
as $k\to\infty$. Thus, $u=v$ and $T_{M_k}(u)\to u$ weakly in $W^{1,p}(Q_R).$
By the weak lower semicontinuity of $L^p$-norm and \eqref{jiexia},
we conclude that
\begin{align}\label{5.3x}
C_{N,p,\gamma}^{\rm cell}\int_{Q_R}|\nabla u|^p\,dx
&\leq C_{N,p,\gamma}^{\rm cell}\varliminf_{k\to\infty}
\int_{Q_R}|\nabla T_{M_k}(u)|^p\,dx\leq L.
\end{align}
Letting $R\to\infty$ in \eqref{5.3x} yields \eqref{zhengxia} in this case.
This completes the proof of Proposition \ref{prop-whllower}.
\end{proof}

We now turn to the $\Gamma$-limsup inequality. The following lemma
uses the bounded-domain result to construct recovery sequences on
$\mathbb R^N$ for bounded functions that equal the same constant
near the boundary of a cube and outside it.
\begin{lemma}\label{lem-oneconst}
Let  $p\in[1,\infty)$ and $\gamma\in(-\infty,-1]\cup(0,\infty)$. Fix $R\in (0,\infty)$ and $a\in(0,R)$.  Define
\begin{align*}
Q:=(-R,R)^N\ \ \text{and}\ \
\mathcal C_a(Q):=\left\{x\in Q:\operatorname{dist}(x,\partial Q)<a\right\}.
\end{align*}
Suppose that $w\in L^\infty(\mathbb R^N)$ satisfies $w|_Q\in W^{1,p}(Q)$
if $p\in(1,\infty)$ and $w|_Q\in BV(Q)$ if $p=1$. Assume that
$w\equiv c$ on $Q^\complement\cup\mathcal C_a(Q)$ for some $c\in\mathbb{R}.$
Then, for any positive sequence $\{\lambda_j\}$ with $\lambda_j\to\lambda_\gamma$ as $j\to\infty$, there exists a sequence
$\{W_j\}_{j\in\mathbb N}\subset L^\infty(\mathbb R^N)$ such that
$W_j\to w$ in $L^1_{\mathrm{loc}}(\mathbb{R}^N)$ as $j\to\infty$ and
\begin{align*}
\varlimsup_{j\to\infty}F_{\lambda_j,p,\gamma}\left(W_j;\mathbb R^N\right)
\leq\Psi_{p,\gamma}^{Q}(w|_Q).
\end{align*}
\end{lemma}

\begin{proof}
Fix $M\in(\|w\|_{L^\infty(\mathbb R^N)},\infty)$ and $\alpha\in(0,1)$.
Let $T_M$ be the truncation operator as in \eqref{jieding}.
By Theorem \ref{thm-domfull} and \eqref{deng}, we conclude that there
exists a sequence $\{v_j\}_{j\in\mathbb N}\subset L^p(Q)$
such that $v_j\to w$ in $L^p(Q)$ as $j\to\infty$ and
\begin{align}\label{yunei}
\varlimsup_{j\to\infty}F_{\alpha\lambda_j,p,\gamma}\left(T_M v_j;Q\right)
\leq\varlimsup_{j\to\infty}F_{\alpha\lambda_j,p,\gamma}\left(v_j;Q\right)
\leq\Psi_{p,\gamma}^{Q}(w|_Q).
\end{align}
Since $M>\|w\|_{L^\infty(\mathbb R^N)}$, it follows that
\begin{align*}
\left\|T_M v_j-w\right\|_{L^p(Q)}=\left\|T_M v_j-T_M w\right\|_{L^p(Q)}\leq\left\|v_j-w\right\|_{L^p(Q)}\to0
\end{align*}
as $j\to\infty$. Thus, $T_Mv_j\to w$ in $L^p(Q)$ as $j\to\infty$.
Choose a Lipschitz function $\eta:Q\to[0,1]$ such that
$\eta\equiv 0$ on $\mathcal C_{\frac{a}{4}}(Q),$
$\eta\equiv 1$ on $Q\setminus\mathcal C_{\frac{a}{2}}(Q)$,
and $L_\eta\leq\frac{4}{a},$ where $L_\eta$ is its Lipschitz constant.
For any given $j\in\mathbb N$, define $W_j^\alpha$ by setting
\begin{align}\label{pinjie}
W_j^\alpha(x):=
\begin{cases}
\eta(x)T_M v_j(x)+[1-\eta(x)]w(x),&x\in Q,\\
w(x),&x\in Q^\complement.
\end{cases}
\end{align}
Then $W_j^\alpha$ is bounded by $M$ and equals $w$ on
$Q^\complement\cup\mathcal C_{\frac{a}{4}}(Q)$. Moreover,
\begin{align}\label{pinlian}
\left\|W_j^\alpha-w\right\|_{L^p(Q)}\leq\left\|T_M v_j-w\right\|_{L^p(Q)}\to0
\end{align}
as $j\to\infty$. This further implies that $W_j^\alpha\to w$ in $L^1_{\mathrm{loc}}(\mathbb R^N)$ as $j\to\infty$.

Now, we show that
\begin{align}\label{tonggu}
\varlimsup_{j\to\infty}F_{\lambda_j,p,\gamma}\left(W_j^\alpha;\mathbb R^N\right)
\leq\alpha^{-p}\Psi_{p,\gamma}^{Q}(w|_Q).
\end{align}
Fix $j\in\mathbb{N}.$ Let
\begin{align*}
D_1:=\left\{(x,y)\in Q\times Q:0<|x-y|<\frac{a}{4}\right\},
\end{align*}
\begin{align*}
D_2:=\left\{(x,y)\in(\mathbb R^N\times\mathbb R^N)
\setminus(Q\times Q):0<|x-y|<\frac{a}{4}\right\},
\end{align*}
and
\begin{align*}
D_3:=\left\{(x,y)\in\mathbb R^N\times\mathbb R^N:|x-y|\geq \frac{a}{4}\right\}.
\end{align*}
Then
\begin{align*}
\left\{(x,y)\in\mathbb R^N\times\mathbb R^N:x\ne y\right\}
=D_1\cup D_2\cup D_3,
\end{align*}
and hence
\begin{align}\label{fenjie}
F_{\lambda_j,p,\gamma}\left(W_j^\alpha;\mathbb R^N\right)
=\sum_{k=1}^3\lambda_j^p
\underset{\scriptscriptstyle |W_j^\alpha(x)-W_j^\alpha(y)|>
\lambda_j|x-y|^{1+\gamma/p}}{\iint_{D_k}}
|x-y|^{\gamma-N}\,dy\,dx.
\end{align}
Let $(x,y)\in D_1$. From \eqref{pinjie}, we infer that
\begin{align*}
W_j^\alpha(x)-W_j^\alpha(y)
&=\eta(x)\left[T_M v_j(x)-T_M v_j(y)\right]
+[1-\eta(x)]\left[w(x)-w(y)\right]\\
&\quad+\left[\eta(x)-\eta(y)\right]\left[T_M v_j(y)-w(y)\right].
\end{align*}
If $1-\eta(x)\ne0$, then $x\in\mathcal C_{\frac{a}{2}}(Q)$ and
$y\in\mathcal C_{\frac{3a}{4}}(Q)$. By the assumption on $w$, we obtain
$w(x)=w(y)=c$. Consequently,
$[1-\eta(x)][w(x)-w(y)]=0.$
From this, we deduce that
\begin{align*}
|W_j^\alpha(x)-W_j^\alpha(y)|
\leq\left|T_M v_j(x)-T_M v_j(y)\right|+L_\eta|x-y|\,\left|T_M v_j(y)-w(y)\right|,
\end{align*}
which further implies that
\begin{align*}
&\left\{(x,y)\in D_1:\
|W_j^\alpha(x)-W_j^\alpha(y)|>\lambda_j|x-y|^{1+\gamma/p}\right\}\\
&\quad\subset\left\{(x,y)\in Q\times Q:
|T_M v_j(x)-T_M v_j(y)|>\alpha\lambda_j|x-y|^{1+\gamma/p}\right\}\\
&\qquad\cup\left\{(x,y)\in Q\times Q:0<|x-y|^\gamma<
\frac{L_\eta^p|T_M v_j(y)-w(y)|^p}{(1-\alpha)^p\lambda_j^p}\right\}.
\end{align*}
Consequently,
\begin{align}\label{e5.6}
&\lambda_j^p\underset{\scriptscriptstyle |W_j^\alpha(x)-W_j^\alpha(y)|>
\lambda_j|x-y|^{1+\gamma/p}}
{\iint_{D_1}}|x-y|^{\gamma-N}\,dy\,dx\notag\\
&\quad\leq\alpha^{-p}F_{\alpha\lambda_j,p,\gamma}\left(T_M v_j;Q\right)
+\lambda_j^p\int_Q\int_{\left\{0<|h|^\gamma<
\frac{L_\eta^p|T_M v_j(y)-w(y)|^p}{(1-\alpha)^p\lambda_j^p}\right\}}
|h|^{\gamma-N}\,dh\,dy\notag\\
&\quad=\alpha^{-p}F_{\alpha\lambda_j,p,\gamma}\left(T_M v_j;Q\right)
+\frac{\mathcal H^{N-1}\left(\mathbb S^{N-1}\right)L_\eta^p}
{|\gamma|(1-\alpha)^p}\left\|T_M v_j-w\right\|_{L^p(Q)}^p.
\end{align}
Letting $j\to\infty$ in \eqref{e5.6} and using \eqref{yunei}, we conclude that
\begin{align}\label{shangu}
&\varlimsup_{j\to\infty}\lambda_j^p\underset{\scriptscriptstyle |W_j^\alpha(x)-W_j^\alpha(y)|>
\lambda_j|x-y|^{1+\gamma/p}}
{\iint_{D_1}}
|x-y|^{\gamma-N}\,dy\,dx\notag\\
&\quad\leq\alpha^{-p}
\varlimsup_{j\to\infty}F_{\alpha\lambda_j,p,\gamma}\left(T_M v_j;Q\right)
\leq\alpha^{-p}\Psi_{p,\gamma}^{Q}(w|_Q).
\end{align}
Let  $(x,y)\in D_2$. By symmetry, we can assume that $x\in Q^\complement$.
Since $|x-y|<\frac{a}{4}$, it follows that
$y\in Q^\complement\cup\mathcal C_{\frac{a}{4}}(Q).$
This, together with the definition of $W_j^\alpha$ and the assumption on $w$, implies that
$W_j^\alpha(x)=W_j^\alpha(y)=c.$
Thus, for any $j\in\mathbb N$,
\begin{align}\label{kualing}
\lambda_j^p\underset{\scriptscriptstyle |W_j^\alpha(x)-W_j^\alpha(y)|>
\lambda_j|x-y|^{1+\gamma/p}}{\iint_{D_2}}
|x-y|^{\gamma-N}\,dy\,dx=0.
\end{align}
Let $(x,y)\in D_3$. If $\gamma\in(0,\infty)$, then
$\lim_{j\to\infty}\lambda_j(\frac a4)^{1+\gamma/p}=\infty.$
From this, we infer that, for all sufficiently large $j$,
\begin{align*}
|W_j^\alpha(x)-W_j^\alpha(y)|
\leq2M\leq\lambda_j\left(\frac a4\right)^{1+\gamma/p}
\leq\lambda_j|x-y|^{1+\gamma/p},
\end{align*}
which implies that, for all sufficiently large $j$,
\begin{align}\label{shenme}
\lambda_j^p\underset{\scriptscriptstyle |W_j^\alpha(x)-W_j^\alpha(y)|>
\lambda_j|x-y|^{1+\gamma/p}}{\iint_{D_3}}
|x-y|^{\gamma-N}\,dy\,dx=0.
\end{align}
If $\gamma\in(-\infty,0)$, then $\lim_{j\to\infty}\lambda_j=0$.
Note that, for any $(x,y)\in Q^\complement\times Q^\complement$,
\begin{align*}
W_j^\alpha(x)=W_j^\alpha(y)=w(x)=w(y)=c.
\end{align*}
Thus,
\begin{align}\label{butong}
&\left\{(x,y)\in D_3:|W_j^\alpha(x)-W_j^\alpha(y)|>
\lambda_j|x-y|^{1+\gamma/p}\right\}\notag\\
&\quad\subset D_3\setminus(Q^\complement\times Q^\complement)\notag\\
&\quad\subset\left\{(x,y)\in Q\times\mathbb R^N:|x-y|\geq \frac{a}{4}\right\}
\cup\left\{(x,y)\in\mathbb R^N\times Q:|x-y|\geq \frac{a}{4}\right\}.
\end{align}
Using this, we have
\begin{align*}
&\lambda_j^p\underset{\scriptscriptstyle |W_j^\alpha(x)-W_j^\alpha(y)|>
\lambda_j|x-y|^{1+\gamma/p}}{\iint_{D_3}}
|x-y|^{\gamma-N}\,dy\,dx\\
&\quad\leq2\lambda_j^p\int_Q\int_{|x-y|\geq \frac{a}{4}}
|x-y|^{\gamma-N}\,dy\,dx
=2\lambda_j^p|Q|\frac{\mathcal H^{N-1}\left(\mathbb S^{N-1}\right)}{|\gamma|}
\left(\frac a4\right)^\gamma\to0
\end{align*}
as $j\to\infty$. By this and \eqref{shenme}, we conclude that
\begin{align*}
\lim_{j\to\infty}\lambda_j^p\underset{\scriptscriptstyle |W_j^\alpha(x)-W_j^\alpha(y)|>
\lambda_j|x-y|^{1+\gamma/p}}{\iint_{D_3}}
|x-y|^{\gamma-N}\,dy\,dx=0,
\end{align*}
which, combined with \eqref{shangu}, \eqref{kualing}, and \eqref{fenjie},
implies that \eqref{tonggu} holds.

Finally, choose $\{\alpha_n\}_{n\in\mathbb N}\subset(0,1)$ such that
$\alpha_n\to1$ as $n\to\infty$. From \eqref{pinlian},
\eqref{tonggu}, and the same diagonal argument
used in the proof of Proposition \ref{prop-pa}, we infer that there
exists a sequence $\{W_j\}_{j\in\mathbb N}\subset L^\infty(\mathbb R^N)$
such that
the desired  conclusions hold; we omit the details here.
This completes the proof of Lemma \ref{lem-oneconst}.
\end{proof}

The following lemma is a  one-dimensional variant of Lemma
\ref{lem-oneconst}.
\begin{lemma}\label{lem-twoconst}
Let $p,$ $R,$ $a,$ $Q,$ and $\mathcal{C}_a(Q)$ be as in Lemma \ref{lem-oneconst}
with $N=1$, and assume that
$\gamma\in(-\infty,-1)\cup(0,\infty)$.
Suppose that $w\in L^\infty(\mathbb R)$ satisfies $w|_Q\in W^{1,p}(Q)$
if $p\in(1,\infty)$ and $w|_Q\in BV(Q)$ if $p=1$. Assume that
there exist $c_-,c_+\in\mathbb{R}$ such that
$w\equiv c_-$ on $(-\infty,-R+a)$ and $w\equiv c_+$ on $(R-a,\infty)$.
Then, for any positive sequence $\{\lambda_j\}$ with
$\lambda_j\to\lambda_\gamma$ as $j\to\infty$, there exists a sequence $\{W_j\}_{j\in\mathbb N}\subset L^\infty(\mathbb R)$ such that
$W_j\to w$ in $L^1_{\mathrm{loc}}(\mathbb{R})$ as $j\to\infty$ and
\begin{align*}
\varlimsup_{j\to\infty}F_{\lambda_j,p,\gamma}(W_j;\mathbb R)
\leq\Psi_{p,\gamma}^{Q}(w|_Q).
\end{align*}
\end{lemma}

\begin{proof}
The proof is similar to that of Lemma \ref{lem-oneconst}, except for the estimate of \eqref{butong} with $\gamma\in (-\infty,-1)$.
By the assumption of $w$, we find that, for any $(x,y)\in(-\infty,-R]\times(-\infty,-R]$,
\begin{align*}
W_j^\alpha(x)=W_j^\alpha(y)=w(x)=w(y)=c_-
\end{align*}
and, for any $(x,y)\in[R,\infty)\times[R,\infty)$,
\begin{align*}
W_j^\alpha(x)=W_j^\alpha(y)=w(x)=w(y)=c_+.
\end{align*}
Thus,
\begin{align*}
&\left\{(x,y)\in D_3:\left|W_j^\alpha(x)-W_j^\alpha(y)\right|>
\lambda_j|x-y|^{1+\gamma/p}\right\}\\
&\quad\subset\left\{(x,y)\in Q\times\mathbb R:|x-y|\geq \frac{a}{4}\right\}
\cup\left\{(x,y)\in Q^\complement\times Q:|x-y|\geq \frac{a}{4}\right\}\\
&\qquad\cup\left((-\infty,-R]\times[R,\infty)\right)
\cup\left([R,\infty)\times(-\infty,-R]\right).
\end{align*}
Using this, we further obtain
\begin{align*}
&\lambda_j^p\underset{\scriptscriptstyle |W_j^\alpha(x)-W_j^\alpha(y)|>
\lambda_j|x-y|^{1+\gamma/p}}{\iint_{D_3}}
|x-y|^{\gamma-1}\,dy\,dx\\
&\quad\leq2\lambda_j^p\int_Q\int_{|x-y|\geq \frac{a}{4}}
|x-y|^{\gamma-1}\,dy\,dx
+2\lambda_j^p\int_{-\infty}^{-R}\int_R^\infty
(y-x)^{\gamma-1}\,dy\,dx\\
&\quad=\frac{4|Q|}{|\gamma|}\left(\frac a4\right)^\gamma\lambda_j^p
+\frac{2(2R)^{\gamma+1}}{\gamma(\gamma+1)}\lambda_j^p\to0
\end{align*}
as $j\to\infty$, where, in the last step, we used
the range $\gamma\in (-\infty,-1).$
This completes the proof of Lemma~\ref{lem-twoconst}.
\end{proof}

The following density lemma is well known; in the absence of a suitable
reference, we give a proof for the sake of completeness.
\begin{lemma}\label{lem-sobloc}
\begin{enumerate}[{\rm(i)}]
\item
Let  $p\in(1,\infty)$ and $u\in\dot W^{1,p}(\mathbb R^N)$. Then there exist
$\{\phi_m\}_{m\in\mathbb{N}}\subset C_{\rm c}^\infty(\mathbb R^N)$ and
$\{c_m\}_{m\in\mathbb{N}}\subset\mathbb R$ such that
\begin{align}\label{tiding}
v_m\to u\quad\text{in }L^1_{\rm loc}(\mathbb R^N)
\ \ \text{and}\ \ \nabla v_m\to\nabla u\quad\text{in }L^p(\mathbb R^N),
\end{align}
with $v_m:=c_m+\phi_m$ for any $m\in\mathbb N$, as $m\to\infty.$

\item Let $N\geq2$ and $u\in\dot{BV}(\mathbb R^N)$. Then there exist
$\{\phi_m\}_{m\in\mathbb{N}}\subset C_{\rm c}^\infty(\mathbb R^N)$ and
$\{c_m\}_{m\in\mathbb{N}}\subset\mathbb R$ such that
\begin{align}\label{bianding}
v_m\to u\quad\text{in }L^1_{\rm loc}(\mathbb R^N)
\ \ \text{and}\ \
|Dv_m|(\mathbb R^N)\to|Du|(\mathbb R^N),
\end{align}
with $v_m:=c_m+\phi_m$ for any $m\in\mathbb N$, as $m\to\infty$.

\item Let $u\in\dot{BV}(\mathbb R)$. Then there exist bounded functions
$\{v_m\}_{m\in\mathbb{N}}\subset\dot{BV}(\mathbb R)$ such that
\begin{align*}
v_m\to u\quad\text{in }L^1_{\rm loc}(\mathbb R)\ \ \text{and}\ \
|Dv_m|(\mathbb R)\to|Du|(\mathbb R)
\end{align*}
as $m\to\infty$.
Moreover, there exist constants
$\{c_m^-\}_{m\in\mathbb{N}},\{c_m^+\}_{m\in\mathbb{N}}\subset\mathbb R$ such that, for any $m\in\mathbb{N},$
$v_m\equiv c_m^{-}$ on $(-\infty,-m]$ and
$v_m\equiv c_m^{+}$ on $[m,\infty)$.
\end{enumerate}
\end{lemma}

\begin{proof}
We first prove (i). By \cite[Theorem~4]{hk95}, we find that there exist
$\{\phi_m\}_{m\in\mathbb{N}}\subset C_{\rm c}^\infty(\mathbb R^N)$ such that
$\|\nabla\phi_m-\nabla u\|_{L^p(\mathbb R^N)}\to0$
as $m\to\infty$. Let $R\in (1,\infty)$ and $B_R:=B(\mathbf{0},R)$.
For  any $m\in\mathbb{N},$ define $c_m:=(u-\phi_m)_{B_1}$ and $v_m:=c_m+\phi_m$.
Then, from the Poincar\'e inequality,  we deduce that,  as $m\to\infty,$
\begin{align}\label{ji}
\|v_m-u-(v_m-u)_{B_R}\|_{L^p(B_R)}\leq C_R\|\nabla v_m-\nabla u\|_{L^p(B_R)}\to 0,
\end{align}
where $C_R$ is a positive constant depending only on $R.$
Applying the Poincar\'e inequality again and  using the assumption that
$(v_m-u)_{B_1}=0$, we obtain, as $m\to\infty,$
\begin{align}\label{wen}
\left|(v_m-u)_{B_R}\right|
&=\left|\frac1{|B_1|}\int_{B_1}\left((v_m-u)_{B_R}-v_m(x)+u(x)\right)\,dx\right|\notag\\
&\leq C_R\left\|v_m-u-(v_m-u)_{B_R}\right\|_{L^p(B_R)}
\leq C_R\left\|\nabla v_m-\nabla u\right\|_{L^p(B_R)}\to 0.
\end{align}
 By this and \eqref{ji}, we conclude that \eqref{tiding} holds.
This completes the proof of  (i).

Next, we show (ii). Let $\rho\in C_{\rm c}^\infty(B_1)$ be nonnegative with
$\int_{\mathbb R^N}\rho\,dx=1$.
For any $m\in\mathbb{N}$, let $u_m:=\rho_{\frac{1}{m}}*u$.
Then, for any $m\in\mathbb{N},$
$u_m\in C^\infty(\mathbb R^N)\cap\dot W^{1,1}(\mathbb R^N)$,
\begin{align}\label{ki}
\|\nabla u_m\|_{L^1(\mathbb R^N)}
\leq|Du|(\mathbb R^N),
\end{align}
and  $u_m\to u$ in $L^1_{\rm loc}(\mathbb R^N)$ as $m\to\infty$.\
From the assumption $N\geq2$ and  from \cite[Theorem~4]{hk95}, we infer that
there exist $\{\phi_m\}_{m\in\mathbb{N}}\subset C_{\rm c}^\infty(\mathbb R^N)$ such that, for any $m\in\mathbb N$,
\begin{align}\label{fuk}
\|\nabla\phi_m-\nabla u_m\|_{L^1(\mathbb R^N)}<\frac1m.
\end{align}
For any $m\in\mathbb{N}$, define $c_m:=(u_m-\phi_m)_{B_1}$
and $v_m:=c_m+\phi_m$.
By an argument similar to that used in both \eqref{wen} and \eqref{ji}, we find that,
for any $m\in\mathbb{N},$
\begin{align*}
\left\|v_m-u_m-(v_m-u_m)_{B_R}\right\|_{L^1(B_R)}\lesssim
\left\|\nabla v_m-\nabla u_m\right\|_{L^1(B_R)}\leq \frac{1}{m}
\end{align*}
and
\begin{align*}
\left|(v_m-u_m)_{B_R}\right|\lesssim \frac{1}{m}.
\end{align*}
This, combined with the fact that
$u_m\to u$ in $L^1_{\rm loc}(\mathbb R^N)$ as $m\to\infty$,
implies that  $v_m\to u$ in $L^1_{\mathrm{loc}}(\mathbb{R}^N)$ as $m\to\infty.$
Moreover, using \eqref{fuk} and \eqref{ki}, we obtain, for any $m\in\mathbb{N},$
\begin{align*}
|Dv_m|(\mathbb R^N)=\|\nabla\phi_m\|_{L^1(\mathbb R^N)}
\leq\|\nabla u_m\|_{L^1(\mathbb R^N)}+\frac1m\leq |Du|(\mathbb R^N)+\frac1m.
\end{align*}
Thus,
\begin{align*}
|Du|(B_R)\leq\varliminf_{m\to\infty}|Dv_m|(B_R)
\leq\varliminf_{m\to\infty}|Dv_m|(\mathbb R^N)\leq\varlimsup_{m\to\infty}|Dv_m|(\mathbb R^N)\leq |Du|(\mathbb R^N).
\end{align*}
Letting $R\to\infty$ in the above inequality, we conclude that
\eqref{bianding} holds. This completes the proof of (ii).

Finally, we prove (iii). From \cite[Theorem~3.28]{afp00}, it follows that
$u$ has one-side limits at every point, which are denoted by $u(x-)$ and $u(x+)$. Moreover, for any given $a,b\in\mathbb{R}$ with $a<b$ and for  almost every $x\in (a,b),$
\begin{align}\label{qiutian}
u(x)=u(a+)+Du(a,x]
\end{align}
and
\begin{align}\label{tianqiu}
u(b-)-u(a+)=Du(a,b).
\end{align}
For any $m\in\mathbb{N}$, let $c_m^-:=u(-m+)$ and $c_m^+:=u(m-)$.
For any given $m\in\mathbb N$ and any $x\in\mathbb R$, define
\begin{align*}
v_m(x):=
\begin{cases}
c_m^-,&x\in(-\infty,-m],\\
u(x),& x\in(-m,m),\\
c_m^+,&x\in[m,\infty).
\end{cases}
\end{align*}
By \eqref{qiutian} and \eqref{tianqiu}, we conclude that, for any $m\in\mathbb N$,
\begin{align*}
\sup_{x\in\mathbb R}|v_m(x)|\leq |u(-m+)|+|Du|(-m,m)<\infty.
\end{align*}
Thus, for any given $m\in\mathbb N$, $v_m\in L^1_{\rm loc}(\mathbb R)$
is bounded. Moreover, it is obviously that
$v_m\to u$ in $L^1_{\rm loc}(\mathbb R)$ as $m\to\infty.$
Note that, for any $m\in\mathbb{N}$ and $\varphi\in C^\infty_{\mathrm{c}}(\mathbb{R}),$
\begin{align*}
&-\int_{\mathbb R}v_m(x)\varphi'(x)\,dx\\
&\quad=-c_m^-\int_{-\infty}^{-m}\varphi'(x)\,dx
-\int_{-m}^{m}u(x)\varphi'(x)\,dx
-c_m^+\int_m^\infty\varphi'(x)\,dx\\
&\quad=-c_m^-\varphi(-m)
+\int_{(-m,m)}\varphi\,dDu
+u(-m+)\varphi(-m)-u(m-)\varphi(m)
+c_m^+\varphi(m)\\
&\quad=\int_{(-m,m)}\varphi\,dDu,
\end{align*}
which implies that, for any given $m\in\mathbb{N},$
$Dv_m=Du\mathopen{\scalebox{1.6}{$\llcorner$}}(-m,m)$,
and hence $|Dv_m|=|Du|\mathopen{\scalebox{1.6}{$\llcorner$}}(-m,m)$.
This further yields
\begin{align*}
\lim_{m\to\infty}|Dv_m|(\mathbb R)=\lim_{m\to\infty}|Du|(-m,m)=|Du|(\mathbb R).
\end{align*}
This completes the proof of (iii), and hence Lemma \ref{lem-sobloc}.
\end{proof}

Now, we prove Theorem \ref{thm-whlndeg} by using Proposition \ref{prop-whllower}
and Lemma \ref{lem-sobloc}.
\begin{proof}[Proof of Theorem \ref{thm-whlndeg}]
Fix a sequence $\{\lambda_j\}_{j\in\mathbb{N}}\subset(0,\infty)$ such that
$\lambda_j\to\lambda_\gamma$ as $j\to\infty$. By Proposition \ref{prop-whllower},
we obtain the desired $\Gamma$-liminf  inequality.
It remains to show that, for any $u\in L^1_{\rm loc}(\mathbb R^N)$,
there exists a sequence $\{u_j\}_{j\in\mathbb{N}}$ such that
$u_j\to u$ in $L^1_{\rm loc}(\mathbb R^N)$ as $j\to\infty$
and
\begin{align*}
\varlimsup_{j\to\infty}F_{\lambda_j,p,\gamma}\left(u_j;\mathbb R^N\right)
\leq\Psi_{p,\gamma}^{\rm hom}(u).
\end{align*}
If $\Psi_{p,\gamma}^{\rm hom}(u)=\infty$, then there is nothing to prove.
We now assume $\Psi_{p,\gamma}^{\rm hom}(u)<\infty$.
From Lemma \ref{lem-sobloc}, it follows that there exists a sequence
$\{w_m\}_{m\in\mathbb{N}}\subset L^\infty(\mathbb R^N)$ such that
\begin{align}\label{nengjin}
w_m\to u\quad\text{in }L^1_{\rm loc}(\mathbb R^N)
\ \ \text{and}\ \
\Psi_{p,\gamma}^{\rm hom}(w_m)\to\Psi_{p,\gamma}^{\rm hom}(u)
\end{align}
as $m\to\infty$. For any given $m\in\mathbb{N},$ choosing $r_{m}$ sufficiently
large and applying Lemmas \ref{lem-oneconst} and \ref{lem-twoconst} with
$Q_{m}:=(-r_m,r_m)^N$, we conclude that there exists a sequence
$\{W_j^{(m)}\}_{j\in\mathbb{N}}\subset L^\infty(\mathbb R^N)$ such that
\begin{align*}
W_j^{(m)}\to w_m\quad\text{in }L^1_{\rm loc}(\mathbb R^N)
\end{align*}
as $j\to\infty$ and
\begin{align*}
\varlimsup_{j\to\infty}F_{\lambda_j,p,\gamma}\left(W_j^{(m)};\mathbb R^N\right)
\leq \Psi_{p,\gamma}^{Q_m}\left(w_m|_{Q_m}\right)=\Psi_{p,\gamma}^{\rm hom}(w_m).
\end{align*}
From this, \eqref{nengjin}, and the same
diagonal argument used in the proof of Proposition \ref{prop-pa},
 we infer that there exists
a sequence $\{u_j\}_{j\in\mathbb N}\subset L^1_{\mathrm{loc}}(\mathbb R^N)$
such that $u_j\to u$ in $L^1_{\mathrm{loc}}(\mathbb R^N)$ as
$j\to\infty$ and
\begin{align*}
\varlimsup_{j\to\infty}F_{\lambda_j,p,\gamma}\left(u_j;\mathbb R^N\right)
\leq\Psi_{p,\gamma}^{\mathrm{hom}}(u).
\end{align*}
We omit the details here. This completes the proof of Theorem \ref{thm-whlndeg}.
\end{proof}

\subsection{The Exceptional Endpoint Case: $(N,p,\gamma)=(1,1,-1)$}
\label{sec-endpoint}

In this subsection, we deal with the  case $(N,p,\gamma)=(1,1,-1)$.
 Recall that, for any $u\in \dot{BV}(\mathbb{R})$,
\begin{align*}
\Psi_{1,-1}^{\rm hom}(u):= C_{1,1,-1}^{\rm cell}
\left[|Du|(\mathbb R)+|u_+-u_-|\right].
\end{align*}
The following result is the main result of this section,  which follows directly
from Propositions \ref{prop-closinf} and \ref{prop-clossup} below.
\begin{theorem}\label{thm-whlendpt}
The family $F_{\lambda,1,-1}(\cdot;\mathbb R)$ converges, in the sense of
$\Gamma$-convergence in $L^1_{\rm loc}(\mathbb R)$, to $\Psi_{1,-1}^{\rm hom}$ as $\lambda\to0^+$.
\end{theorem}

\begin{proposition}\label{prop-closinf}
Let $\{\lambda_j\}_{j\in\mathbb N}\subset(0,\infty)$ satisfy
$\lambda_j\to0$ as $j\to\infty$, $u\in L^1_{\rm loc}(\mathbb R)$, and
$\{u_j\}_{j\in\mathbb N}\subset L^1_{\rm loc}(\mathbb R)$ satisfy
$u_j\to u$ in $L^1_{\rm loc}(\mathbb R)$ as $j\to\infty$.
If
\begin{align*}
\varliminf_{j\to\infty}F_{\lambda_j,1,-1}(u_j;\mathbb R)<\infty,
\end{align*}
then $u\in\dot{BV}(\mathbb R)$ and
\begin{align}\label{bixia}
\varliminf_{j\to\infty}F_{\lambda_j,1,-1}\left(u_j;\mathbb R\right)\geq C_{1,1,-1}^{\rm cell}\left[|Du|(\mathbb R)+|u_+-u_-|\right].
\end{align}
\end{proposition}

\begin{proof}
By Proposition \ref{prop-whllower}, we find that $u\in\dot{BV}(\mathbb R)$.
Passing to a subsequence, we may assume that
\begin{align*}
L:=\lim_{j\to\infty}F_{\lambda_j,1,-1}\left(u_j;\mathbb R\right)
=\varliminf_{j\to\infty}F_{\lambda_j,1,-1}\left(u_j;\mathbb R\right)<\infty.
\end{align*}
Fix $M\in(\|u\|_{L^\infty(\mathbb R)},\infty)$. Let $R\in (0,\infty)$ and
$E_R:=(-\infty,-R)\cup(R,\infty).$
Then, for any $j\in\mathbb{N},$
\begin{align*}
F_{\lambda_j,1,-1}\left(u_j;\mathbb R\right)
\geq F_{\lambda_j,1,-1}\left(u_j;(-R,R)\right)
+F_{\lambda_j,1,-1}\left(u_j;E_R\right).
\end{align*}
Using  Theorem \ref{thm-domfull}(i), we obtain
\begin{align*}
\varliminf_{j\to\infty}F_{\lambda_j,1,-1}\left(u_j;(-R,R)\right)
\geq C_{1,1,-1}^{\rm cell}|Du|(-R,R).
\end{align*}
Thus, to show \eqref{bixia}, it remains to prove that
\begin{align}\label{waixia}
\varliminf_{j\to\infty}F_{\lambda_j,1,-1}\left(u_j;E_R\right)
\geq C_{1,1,-1}^{\rm cell}|u_+-u_-|.
\end{align}
Let $T_M$  be as in \eqref{jieding}. For any $0<|s|<1/R$, let $v_j(s):=T_M(u_j(1/s))$ and
$v(s):=u(1/s)$,
and let $v_j(0)=v(0)=0$. Then, for any $\delta\in (R,\infty)$,
\begin{align*}
\int_{1/\delta<|s|<1/R}|v_j(s)-v(s)|\,ds&=\int_{R<|x|<\delta}|T_M(u_j(x))-T_M(u(x))|\,\frac{dx}{x^2}\\
&\leq\frac1{R^2}\int_{-\delta}^{\delta}|u_j(x)-u(x)|\,dx
\end{align*}
and
\begin{align*}
\int_{|s|<1/\delta}|v_j(s)-v(s)|\,ds\leq\frac{4M}{\delta}.
\end{align*}
Thus, for any $\delta\in (R,\infty)$,
\begin{align}\label{e5.33}
\left\|v_j-v\right\|_{L^1((-1/R,1/R))}\leq\frac1{R^2}
\left\|u_j-u\right\|_{L^1((-\delta,\delta))}+\frac{4M}{\delta}.
\end{align}
Letting $j,\delta\to\infty$ in \eqref{e5.33}, we conclude that
\begin{align}\label{fanlian}
v_j\to v\quad\text{in }L^1((-1/R,1/R)).
\end{align}
From the definition of $F_{\lambda,1,-1}$, we deduce that, for any $j\in\mathbb{N},$
\begin{align}\label{e5.34}
&F_{\lambda_j,1,-1}\left(v_j;(-1/R,1/R)\right)\notag\\
&\quad=\lambda_j\underset{\scriptscriptstyle |T_M(u_j(1/s))-T_M(u_j(1/t))|>\lambda_j}
{\int_{0<|s|<1/R}\int_{0<|t|<1/R}}\frac{dt\,ds}{|s-t|^2}=\lambda_j
\underset{\scriptscriptstyle |T_M(u_j(x))-T_M(u_j(y))|>\lambda_j}
{\int_{E_R}\int_{E_R}}\frac{dy\,dx}{|x-y|^2}\notag\\
&\quad=F_{\lambda_j,1,-1}\left(T_M(u_j);E_R\right)\leq
F_{\lambda_j,1,-1}\left(u_j;E_R\right)\leq
F_{\lambda_j,1,-1}\left(u_j;\mathbb R\right).
\end{align}
Letting $j\to\infty$ in \eqref{e5.34} and further using Theorem \ref{thm-domfull}(i)
and \eqref{fanlian}, we then conclude that $v\in BV(-1/R,1/R)$ and
\begin{align}\label{chaoji}
C_{1,1,-1}^{\rm cell}|Dv|(-1/R,1/R)
&\leq\varliminf_{j\to\infty}F_{\lambda_j,1,-1}\left(v_j;(-1/R,1/R)\right)\notag\\
&\leq\varliminf_{j\to\infty}F_{\lambda_j,1,-1}\left(u_j;E_R\right)\leq L<\infty.
\end{align}
On the other hand, we find that
\begin{align*}
v(0-)=\lim_{s\to0-}u(1/s)=u_-
\ \ \text{and}\ \ v(0+)=\lim_{s\to0+}u(1/s)=u_+,
\end{align*}
which, together with \cite[Theorem~3.28]{afp00}, implies that
\begin{align*}
|Dv|(-1/R,1/R)\geq |Dv|(\left\{0\right\})=|Dv(\left\{0\right\})|=|u_+-u_-|.
\end{align*}
From  this and \eqref{chaoji}, we infer that \eqref{waixia} holds, which
completes the proof of Proposition \ref{prop-closinf}.
\end{proof}

\begin{proposition}\label{prop-clossup}
Let $u\in\dot{BV}(\mathbb R)$ and $\{\lambda_j\}_{j\in\mathbb N}\subset(0,\infty)$ satisfy $\lambda_j\to0$ as $j\to\infty$.
Then there exists a sequence $\{u_j\}_{j\in\mathbb N}\subset L^1_{\rm loc}(\mathbb R)$ such that $u_j\to u$ in $L^1_{\rm loc}(\mathbb R)$ as $j\to\infty$ and
\begin{align}\label{bishang}
\varlimsup_{j\to\infty}F_{\lambda_j,1,-1}\left(u_j;\mathbb R\right)
\leq C_{1,1,-1}^{\rm cell}\left[|Du|(\mathbb R)+|u_+-u_-|\right].
\end{align}
\end{proposition}

\begin{proof}
For any given $m\in\mathbb{N}$ and any $x\in\mathbb R$, define
\begin{align*}
w_m(x):=
\begin{cases}
u(x),&x\in(-m,m),\\
u(-m+),&x\notin(-m,m).
\end{cases}
\end{align*}
Obviously, $w_m\to u$ in $L^1_{\rm loc}(\mathbb R)$ as $m\to\infty$. Furthermore,
by \eqref{qiutian}, we find that
\begin{align*}
\|w_m\|_{L^\infty(\mathbb R)}\leq |u(-m+)|+|Du|(-m,m)<\infty.
\end{align*}
Note that, for any $m\in\mathbb{N}$ and $\varphi\in C_{\rm c}^\infty(\mathbb R)$,
\begin{align*}
&-\int_{\mathbb R}w_m(x)\varphi'(x)\,dx\\
&\quad=-u(-m+)\int_{-\infty}^{-m}\varphi'(x)\,dx
-\int_{-m}^{m}u(x)\varphi'(x)\,dx
-u(-m+)\int_m^\infty\varphi'(x)\,dx\\
&\quad=-u(-m+)\varphi(-m)+\int_{(-m,m)}\varphi\,dDu
\\
&\quad\quad+ u(-m+)\varphi(-m)-u(m-)\varphi(m)+u(-m+)\varphi(m)\\
&\quad=\int_{(-m,m)}\varphi\,dDu+\left[u(-m+)-u(m-)\right]\varphi(m).
\end{align*}
This shows that, for any $m\in\mathbb{N},$
\begin{align*}
Dw_m=Du\mathopen{\scalebox{1.6}{$\llcorner$}}(-m,m)+
\left[u(-m+)-u(m-)\right]\delta_m,
\end{align*}
and hence
\begin{align*}
|Dw_m|(\mathbb R)=|Du|(-m,m)+|u(-m+)-u(m-)|,
\end{align*}
where $\delta_m$ is the unit Dirac mass at $m$.
Letting $m\to\infty,$ we have
\begin{align*}
\lim_{m\to\infty}|Dw_m|(\mathbb R)
=|Du|(\mathbb R)+|u_+-u_-|.
\end{align*}
For any given $m\in\mathbb{N},$
choosing $r_{m}$ sufficiently  large and applying
Lemma \ref{lem-oneconst} with
$p:=1$, $\gamma:=-1,$ and
$Q_{m}:=(-r_m,r_m)$, we conclude that
there exists  a sequence
$\{W_j^{(m)}\}_{j\in\mathbb{N}}\subset L^\infty(\mathbb R)$ such that
$W_j^{(m)}\to w_m$ in $L^1_{\rm loc}(\mathbb R)$ as $j\to\infty$ and
\begin{align*}
\varlimsup_{j\to\infty}F_{\lambda_j,1,-1}\left(W_j^{(m)};\mathbb R\right)
\leq\Psi_{1,-1}^{Q_m}(w_m|_{Q_m})=C_{1,1,-1}^{\rm cell}|Dw_m|(\mathbb R).
\end{align*}
Using this and a diagonal argument used in the proof of Proposition \ref{prop-pa}, we find that there exists a sequence
$\{u_j\}_{j\in\mathbb N}\subset L^1_{\rm loc}(\mathbb R)$ such that $u_j\to u$ in $L^1_{\rm loc}(\mathbb R)$ as $j\to\infty$ and \eqref{bishang} holds.
This completes the proof of Proposition \ref{prop-clossup}.
\end{proof}

\smallskip

\noindent\textbf{Acknowledgements}\quad
The main mathematical ideas and approaches in this article
originated from the authors. Generative AI tools were used as
supplementary aids during the exploratory and manuscript-preparation
stages, particularly in refining the presentation and examining
technical details. All mathematical arguments and proofs in the
final manuscript were independently verified and finalized by
the authors, who take full responsibility for their correctness.

Finally, we would like to point out that N. Picenni in this article
\emph{Lower bounds and Poincar\'e inequality for the Brezis--Seeger--Van
Schaftingen--Yung functionals} (arXiv:2610.05929)
independently and simultaneously obtained the closely related
$\Gamma$-liminf estimates for $\gamma<-p$.
We became aware of his work only after the present manuscript
had already been submitted to arXiv.
The first version of our article was submitted to arXiv on October 4, 2026,
and appeared there on October 7, whereas Picenni's article was submitted
on October 5 and appeared on October 6.
For $\gamma<-p$, Picenni's proof relies on dyadic discretization,
whereas ours is based on integral estimates for averaged difference quotients.
Moreover, we establish the $\Gamma$-convergence for all
$p\in[1,\infty)$ and $\gamma\in\mathbb R\setminus\{0\}$.

\smallskip

\noindent\textbf{Data Availability}\quad
This article has no associated data.

\section*{Declarations}

\noindent\textbf{Conflict of interest}\quad
On behalf of all authors, the corresponding author
states that there is no conflict of interest.

\bigskip

\noindent
Xiaosheng Lin

\smallskip

\noindent
School of  Science, Jimei University,
Xiamen 361021, The People's Republic of China

\smallskip

\noindent {\it E-mail}: \texttt{xslin@jmu.edu.cn}

\bigskip

\noindent Dachun Yang, Wen Yuan, and Yangyang Zhang

\smallskip

\noindent Laboratory of Mathematics and Complex Systems
(Ministry of Education of China),
School of Mathematical Sciences, Institute for Advanced Study,
Beijing Normal University,
Beijing 100875, The People's Republic of China

\smallskip

\noindent{\it E-mails:} \texttt{dcyang@bnu.edu.cn} (D. Yang)

\noindent\phantom{{\it E-mails:}} \texttt{wenyuan@bnu.edu.cn} (W. Yuan)

\noindent\phantom{{\it E-mails:}} \texttt{yangyzhang@bnu.edu.cn} (Y. Zhang)

\bigskip

\noindent Sibei Yang

\medskip

\noindent School of Mathematics and Statistics, Lanzhou University, Lanzhou 730000, The People's Republic of China

\smallskip

\noindent{\it E-mail:} \texttt{yangsb@lzu.edu.cn}

\end{document}